\documentclass{ut-thesis}

\usepackage{amssymb}
\usepackage{amsthm}
\usepackage{graphicx}
\usepackage[all]{xy}
\usepackage{mathrsfs}

\degree{Doctor of Philosophy}
\department{Mathematics}
\gradyear{2013}
\author{Kathleen Smith}
\title{Connectivity and Convexity Properties of the Momentum Map for Group Actions on Hilbert Manifolds}

\begin{document}

\newcommand{\real}{\mathbb{R}}
\newcommand{\complex}{\mathbb{C}}

\newtheorem{thm}{Theorem}[section]
\newtheorem{cor}[thm]{Corollary}
\newtheorem{lem}[thm]{Lemma}
\newtheorem{claim}[thm]{Claim}
\newtheorem{prop}[thm]{Proposition}
\newtheorem{fact}[thm]{Fact}
\newtheorem{note}[thm]{Definition}
\newtheorem{eg}[thm]{Example}

\theoremstyle{remark}
\newtheorem{rem}[thm]{Remark}
\newtheorem{notation}[thm]{Notation}
\newtheorem*{thma}{Theorem 5.4.5.}
\newtheorem*{thma'}{Theorem 5.4.4.}
\newtheorem*{thmb}{Theorem 4.3.5.}

\newcommand{\abs}[1]{\lvert#1\rvert}

\begin{preliminary}

\maketitle

\begin{abstract}

In the early $1980$s a landmark result was obtained by Atiyah and independently Guillemin and Sternberg:  the image of the momentum map for a torus action on a compact symplectic manifold is a convex polyhedron.  Atiyah's proof makes use of the fact that level sets of the momentum map are connected.  These proofs work in the setting of finite-dimensional compact symplectic manifolds.  One can ask how these results generalize.  A well-known example of an infinite-dimensional symplectic manifold with a finite-dimensional torus action is the based loop group.  Atiyah and Pressley proved convexity for this example, but not connectedness of level sets.  A proof of connectedness of level sets for the based loop group was provided by Harada, Holm, Jeffrey and Mare in $2006$.

In this thesis we study Hilbert manifolds equipped with a strong symplectic structure and a finite-dimensional group action preserving the strong symplectic structure.  We prove connectedness of regular generic level sets of the momentum map.  We use this to prove convexity of the image of the momentum map.

\end{abstract}

\begin{acknowledgements}

I would like to thank my supervisors, Lisa Jeffrey and Yael Karshon, for their guidance and patience over the years.
 I would also like to thank Paul Selick for many helpful discussions and advice.

\end{acknowledgements}

\tableofcontents

\end{preliminary}

\chapter{Introduction}

	In the early $1980$s Atiyah \cite{MA82} and independently and simultaneously Guillemin and Sternberg \cite{GS82} arrived at a now famous finite-dimensional abelian convexity result.  Their result is:
	
	\begin{thm}[Atiyah-Guillemin-Sternberg]
		\label{finite convexity}
		
			Let $(M, \omega)$ be a compact connected symplectic manifold.  Let $T$ be an $n-$torus and let $\lambda \colon T \times M \rightarrow M$ be a Hamiltonian action of $T$ on $M$ with momentum mapping $\mu \colon M \rightarrow \mathfrak{t}^*$.  Let $M^T$ denote the fixed point set of $\lambda$.  Then
		\begin{list}{}{}
	\item{(i)} the image $\mu(M^T)$ is a finite subset of $\mathfrak{t}^*$;
	\item{(ii)} $\mu(M)$ is the convex hull of $\mu(M^T)$.
\end{list}

\noindent In particular the image $\mu(M)$ is a convex polyhedron.
	\end{thm}

	Atiyah's proof of Theorem \ref{finite convexity} makes use of the following connectivity result:  Under the same hypotheses as Theorem \ref{finite convexity},
	
	\begin{thm}
		\label{finite connectivity}
		
		For every $c \in \mathfrak{t}^*$, the level $\mu^{-1}(c)$ is connected (or empty).
	\end{thm}

	He deduces Theorem \ref{finite convexity} from Theorem \ref{finite connectivity}.

	Over the last $30$ years there has been considerable interest in various infinite-dimensional Hamiltonian systems, namely, infinite-dimensional symplectic manifolds equipped with actions of finite-dimensional tori.  For example, Atiyah in \cite{MA82} asked whether Theorem \ref{finite convexity} could be extended in any interesting way to infinite-dimensions.  Atiyah and Pressley \cite{AP83}
 answered this question in the affirmative.  They proved an extension of Theorem \ref{finite convexity} for the based loop group, an infinite-dimensional symplectic manifold, with a finite-dimensional torus action.  Before we state this result more precisely we need the following definitions.
 
 Let $G$ be a compact, connected and simply connected Lie group.  Fix a $G$-invariant inner product on the Lie algebra $\mathfrak{g}$.  The \textbf{loop group} is defined as the set of maps from $S^1$ to $G$ that are of Sobolev class $H^1$.  We will denote the loop group by $M_1$.  So $$M_1 = H^1(S^1, G).$$  The subset $\Omega G$ of $M_1$ consisting of those loops $f \colon S^1 \rightarrow G$ for which $f(1)$ is the identity element in $G$ is called the \textbf{based loop group}.  We refer the reader to Chapter $6$ for more details regarding the loop group and the based loop group.
 
 Atiyah and Pressley in \cite{AP83} prove:  
 \begin{thm}
 	\label{loop group convexity}
 	
 	Let $G$ be a compact, connected and simply connected Lie group with maximal torus $T$.  Let $\Omega G$ be the based loop group.  Let $R := T \times S^1$ act on $\Omega G$ where
 	
 	\begin{list}{}{}
 		\item{(i)} the rotation group $S^1$ acts on $\Omega G$  by ``rotating the loop": \\ 
 		if $\gamma \in \Omega G$ and $e^{i \theta} \in S^1$, $\theta \in [0, 2 \pi]$, then $\left( e^{i \theta} \gamma \right)(s) := \gamma(s+ \theta) \gamma(\theta)^{-1}$, and;
 		\item{(ii)} the maximal torus acts on $\Omega G$ by conjugation: \\ 
 		if  $\gamma \in \Omega G$ and $t \in T$, then $(t \gamma)(s) := t \gamma(s) t^{-1}$.
 	\end{list}		
 		\noindent Note that these actions commute.  Then the image of the momentum map is convex and it is the convex hull of the images of the fixed points.
 
 \end{thm}
 
 \begin{rem}
 	Atiyah points out that the requirement that $G$ be simply connected may be weakened to semi-simple.  Notice that $\Omega G$ then has several connected components. In this case the image of each component of $\Omega G$ is a convex polyhedron; it is the convex hull of the images corresponding to the fixed points in that particular component.
 \end{rem}
 
 	We will not go into the very detailed proof of Theorem \ref{loop group convexity} which is \underline{specific} to this example of the based loop group.  	We do nevertheless note that Atiyah and Pressley in \cite{AP83} remark that their Theorem \ref{loop group convexity} could be proved by extending the method of proof of Theorem \ref{finite convexity} so as to cover their infinite-dimensional situation.  They do not carry out this argument nor do they provide any hints on what might be required to do so.
 	
 	In $2006$ in \cite{HHJM06}, Harada, Holm, Jeffrey, and Mare proved infinite-dimensional analogues (with respect to the based loop group $\Omega G$ example of Atiyah \cite{AP83}) of the well-known Theorem \ref{finite connectivity} result in finite-dimensional symplectic geometry.  Before we can recall these specific results we need another definition.
 	
 	The set $\Omega_{\textrm{alg}}$, the \textbf{algebraic based loop group}, is the subset of the based loop group $\Omega G$ consisting of loops which have a finite Fourier series (when $G$ is identified with a group of matrices).
 	
 	The main results of \cite{HHJM06} that we are concerned with are:
 	
 \begin{thm}
 	\label{lisa 1}
 	
 	Any level set of the momentum map $\mu$ of the $T \times S^1$ action restricted to $\Omega_{\textrm{alg}}$ is connected (for regular or singular values of the momentum map).
 \end{thm}
 	
 	\begin{thm}
 		\label{lisa 2 regular}
 	
	Let $\mu$ be the momentum map for the $T \times S^1$ action on $\Omega G$.  The level set $\mu^{-1}(c)$ of the momentum map is connected, provided that $c$ is a regular value.
 	\end{thm}

 \begin{rem}
The space $\Omega G$, being a Hilbert manifold, in particular has a topology.  Theorem \ref{lisa 2 regular} refers to the topology of $\Omega G$ as a Hilbert manifold.   The subset $\Omega_{\textrm{alg}}$ of $\Omega G$ can also be equipped with a topology.  Theorem \ref{lisa 1} refers to the direct  limit topology on $\Omega_{alg}$.  We direct the reader to Chapter $6$ for further details.  
	\end{rem}

	\begin{rem}
	The extra hypothesis that $c$ be a regular value of the momentum map in Theorem \ref{lisa 2 regular} is needed so that Morse-theoretic arguments in infinite-dimensions can be used in the proof.  In later years, Mare in \cite{ALM10} was able to eliminate the regular value hypothesis for the momentum map $\mu$.  Mare proved that the singular level sets of $\mu$ for the $T \times S^1$ action on $\Omega G$ are connected.  His argument works for the space of $C^{\infty}$ loops and also for the space of loops of Sobolev class $H^s$ for any $ s \geq 1$.	\end{rem}

\subsection{Thesis Outline}

The main results of this thesis are infinite-dimensional analogues of well-known connectedness and convexity results in finite-dimensional symplectic geometry.  Namely, we establish an analogue of Theorem \ref{finite convexity} and Theorem \ref{finite connectivity}.  We prove:

\begin{thma'}[Connectivity Theorem]
		
	Let $M$ be a connected strongly symplectic Hilbert manifold.  Suppose that we have an almost periodic $\mathbb{R}^n$ action on $M$ with momentum map $\mu : M \rightarrow \mathbb{R}^n$.  Suppose that the $\mathbb{R}^{n}$ action has isolated fixed points.  Suppose that there exists a complete invariant Riemannian metric on $M$ such that there exists a hyperplane $H$ of $\mathbb{R}^n$ such that for all $\xi \in \mathbb{R}^n \smallsetminus H$ the map  $\mu^{\xi} \colon \textrm{M} \rightarrow \mathbb{R}$ is bounded from one side and satisfies Condition (C) (See section \ref{Condition (C)}).  Then the momentum mapping $\mu$ satisfies
	
	\begin{list}{}{}
	\item[($A$)] The set $\{ c \in \mathbb{R}^n \hspace{2mm} | \hspace{2mm} c \textrm{ is a regular value of $\mu$ and } \mu^{-1}(c) \textrm{ is connected } \} \subseteq \mathbb{R}^n$ is residual.
	\end{list}
	
\end{thma'}

\begin{thma}[Convexity Theorem]
	
	Let $M$ be a connected strongly symplectic Hilbert manifold.  Suppose that we have an almost periodic $\mathbb{R}^n$ action on $M$ with momentum map $\mu : M \rightarrow \mathbb{R}^n$.  Suppose that the $\mathbb{R}^{n}$ action has isolated fixed points and suppose that $\mu(M)$ is closed.  Suppose that there exists a complete invariant Riemannian metric on $M$ such that there exists a hyperplane $H$ of $\mathbb{R}^n$ such that for all $\xi \in \mathbb{R}^n \smallsetminus H$ the map  $\mu^{\xi} \colon \textrm{M} \rightarrow \mathbb{R}$ is bounded from one side and satisfies Condition (C).  Then the momentum mapping $\mu$ satisfies
	
	\begin{list}{}{}
	
	\item[($B$)] the image $\mu(M)$ is convex.
	\end{list}
	
\end{thma}

Note that the Palais-Smale compactness condition, namely Condition (C) (see section \ref{Condition (C)}), is an important hypothesis for our connectedness and convexity theorems, Theorems \ref{connectedness} , \ref{convexity}.  Condition (C) is a ``compactness condition" on real-valued functions of class $C^1$ defined on a Riemannian manifold modelled upon a Hilbert space.  It is needed in order to extend Morse theory to our infinite-dimensional setting.

\noindent Let us now highlight the contents of each chapter in this thesis and, where appropriate, briefly explain how the respective material contributes to the main thesis results, the Convexity Theorem \ref{convexity}.

Chapter $2$, Background and Preliminaries, provides a basic review of relevant known facts and definitions from the theory of differential topology.  Throughout this thesis our manifold $M$ will always be a Hausdorff, paracompact Hilbert manifold modelled on a real separable Hilbert space.  That is, $M$ is equipped with an equivalence class of smooth (meaning $C^{\infty}$) atlases such that all charts take values in an infinite-dimensional separable real Hilbert space.

The purpose of Chapter $3$, Normal Forms, is to extend the existing theory on local normal forms for Hamiltonian group actions to infinite-dimensional Banach manifolds.  More specifically, we formalize the local linearization theorem for compact group actions on Banach manifolds (Theorem \ref{locallinearization}) originally noted by Weinstein (without proof) in \cite{AL69}.  We also establish a symplectic version of this local linearization theorem (Theorem \ref{locallinearizationsymplectic}).  In so doing, we provide a $G$-equivariant version of Moser's argument (Lemma \ref{moserthm}) suitable for our goal.  It is the symplectic version of the local linearization theorem that is needed later in the thesis to help prove Theorem \ref{even index} which is an infinite-dimensional analogue of a lemma of Atiyah \cite[Lemma $2.2$]{MA82} and Guillemin and Sternberg \cite[Theorem $5.3$]{GS82}.  

Chapter $4$, Connectedness - The Base Case, introduces the notion of what it means for a Riemannian metric on a Hilbert manifold $M$ to be standard near each critical point of a smooth real-valued function on $M$.  Suppose that we are given a complete Riemannian metric $g$ on a Hilbert manifold $M$ and let $f \colon M \rightarrow \mathbb{R}$ be a smooth function.  For $g$ to be standard (near each critical point $p$ of $f$) means that $g$ coincides with some Riemannian metric on $M$ whose gradient vector field is standard near each $p$.  
For a complete and precise definition see Definition \ref{standard def} and the subsequent Remark \ref{std rephrased}.  With this ``standard" hypothesis on the Riemannian metric we are able to provide an alternate proof of the known Global (Un) Stable Manifold Theorem, Theorem \ref{W^s submfld}, which tells us that the stable and unstable sets of $p$ are in fact manifolds.  However, the main feature of Chapter $4$ is the Connected Levels Theorem, Theorem \ref{connectedlevels}:

\begin{thmb}[Connected Levels]
			
		Let $M$ be a connected Hilbert manifold and let $f \colon M \rightarrow \mathbb{R}$ be a Morse function that is bounded from below and none of whose critical points have index or coindex equal to $1$.  Suppose that there exists a complete Riemannian metric on $M$ such that $f$ satisfies Condition $(C)$.  Then the level set $f^{-1}(c) \subset M$ is connected for every $c$ in $\mathbb{R}$.  
		
\end{thmb}	

\noindent This result is interesting in its own right.  Its proof relies on Morse theoretic arguments that follow from the fact that there exists a complete Riemannian metric on $M$ for which $f$ satisfies the Palais Smale Condition (C) and such that the negative gradient field of $f$ is standard near each critical point of $f$.  Notice that Theorem \ref{connectedlevels} establishes the connectivity of \underline{all} level sets of $f$.  The $n=1$ case of the Connectivity Theorem, Theorem \ref{convexity}, will follow from Theorem \ref{connectedlevels}; details of this $n=1$ claim are provided in the next chapter within the proof of Theorem \ref{convexity}.

Chapter $5$, Convexity and Connectedness, defines one of the main ingredients in the Connectivity and Convexity Theorems.  Specifically, the chapter begins by defining what is meant by an almost periodic $\mathbb{R}^n$ action on a Hilbert manifold $M$.  See Definitions \ref{almost periodic R} and \ref{almost periodic}.  The reader may think of an almost periodic $\mathbb{R}^n$ action as a generalization of a torus action.  We prove that in the presence of an almost periodic $\mathbb{R}^n$ action on $M$, the set of singular values of the resulting momentum map is contained in a countable union of hyperplanes (Theorem \ref{reg values residual}).  (In particular, the set of regular values of the momentum map is residual in $\mathbb{R}^n$.)  Then, the chapter ends with the statement and proof of the thesis main results, the Connectivity Theorem \ref{connectedness} and Convexity Theorem \ref{convexity}.  Following the method of Atiyah \cite{MA82}, the Connectivity Theorem is established by induction on the dimension of the almost periodic $\mathbb{R}^n$ action on $M$.  Note that in the finite-dimensional convexity result, Theorem \ref{finite convexity}, Guillemin and Sternberg prove convexity but not through connectedness (see \cite{GS82}).  They do not provide any results for connectedness.  Atiyah proves convexity using connectedness (see \cite{MA82}) but there is a gap in his argument for connectedness.  This occurs in his induction step where he claims that the connectedness of the regular level sets of the momentum map implies that all level sets of the momentum map are connected  \textit{by continuity}.  A nice example to illustrate the problem is provided below. 
	
	\begin{eg}
	Let $h \colon S^2 \rightarrow S^1$ be the map that sends  $(x_1, x_2, x_3) \mapsto e^{i \pi x_3}$.  This map has exactly one singular value (at $x_3=-1$).  All the regular level sets are connected; they are circles.  But the singular level set above $-1$, namely $\{(0,0,1), \hspace{1mm} (0,0,-1) \}$, is disconnected.  See Figure \ref{fig:example}
	
	\begin{figure}[htpb]
	\begin{center}
	\includegraphics{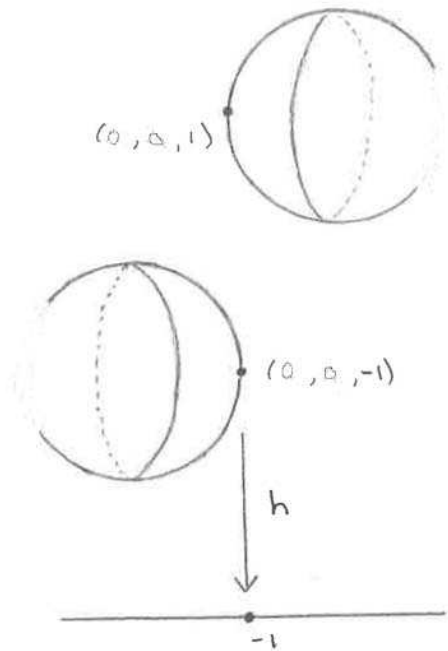}
	\end{center}
	\caption{Singular level set of $h$ above $-1$ for example $1.0.9$. }
	\label{fig:example}
	\end{figure}

	\end{eg}
	
	\noindent This \textit{by continuity} matter was resolved by Lerman and Tolman in \cite[sections $\S 4$ and $\S 5$]{LT97}.

Lastly, Chapter $6$ illustrates that the Convexity Theorem reproduces known infinite-dimensional convexity results for a significant example (see \cite{HHJM06}, \cite{AP83}).  Namely, it reproduces the connectivity and convexity results with regards to the based loop group.

\chapter{Background and Preliminaries}

This chapter consists of two parts.  We review a selection of well known results and some standard definitions from the theory of differentiable manifolds, differential topology and point set topology.  As well, we declare some notational conventions.

The material of these sections borrows from many sources.  We use Lang \cite{SL01}, Palais \cite{RP63} and Royden \cite{HLR88} for basic foundational results.
\section{Function-Analytic Preliminaries}		
		
	Let $M$ be a Hausdorff, paracompact Hilbert manifold modelled on a real separable Hilbert space $ \left( \mathbb{H}, \langle \cdot, \cdot \rangle \right)$.  That is, $M$ is equipped with an equivalence class of $C^{\infty}$ atlases such that all charts take values in a separable real Hilbert space $\mathbb{H}$.  
	
	Recall that a smooth \textbf{vector field}, say $X$, on $M$ is a smooth cross-section of the tangent bundle $TM$, i.e., a smooth map $X \colon M \rightarrow TM$ such that $\pi \circ X = id$.
	
\begin{note}
\label{smooth riemannian str}
	Let $M$ be a Hilbert manifold.  
	
	For each $x \in M$ a \textbf{strongly nondegenerate inner product} $g_x$ on $T_xM$ is a positive-definite, symmetric, bilinear form $$g_x( \cdot , \cdot ) \colon T_xM \times T_xM \rightarrow \mathbb{R}$$  \noindent such that the norm $\Vert \cdot \Vert_x = g_x( \cdot , \cdot )^{\frac{1}{2}}$ defines the topology of $T_{x}M$. Moreover, we require that $g_ x$ determine a bounded, invertible operator $T_xM \rightarrow  (T_xM)^*$ with bounded inverse. 
	
	For each point in $M$ there exists a neighbourhood $D \subseteq M$ and a chart with target a Hilbert space. 	Let $\phi$ be a chart in $M$ having as target a Hilbert space $( \mathbb{H} \hspace{1mm}, \hspace{1mm} \left\langle \cdot , \cdot \right\rangle )$ such that the following holds:  for each $x \in D$ we define the operator $G(x) \colon \mathbb{H} \rightarrow \mathbb{H}$ as follows:  Identify $T_xM$ with $\mathbb{H}$ by the Hilbert space isomorphism $$ \left( d \phi |_x \right)^{-1} \colon \mathbb{H} \rightarrow T_xM.$$ Then $$ \left\langle G(x)u,v \right\rangle = g \left( \left( d \phi |_x \right)^{-1} (u) , \left( d \phi |_x \right)^{-1} (v) \right) \textrm{ for all } u,v \in \mathbb{H}.$$ \noindent Thus $x \mapsto G(x)$ is a map from $D$ to the space of positive definite symmetric bounded operators on $\mathbb{H}$ with the operator norm.  If we require the map $x \mapsto G(x)$ to be smooth with respect to the operator topology (it follows that $x \mapsto G^{-1}(x)$ is also smooth) then we call $x \mapsto g_x( \cdot , \cdot )$ a \textbf{(smooth) Riemannian metric } (or \textbf{(smooth) Riemannian structure}) on $M$. 
	
	A (strong) \textbf{Riemannian manifold} $(M,g)$ is a manifold $M$ equipped with a smooth Riemannian metric $g$.
	
\end{note}	

\noindent Note that we require a strong Riemannian metric on $M$.  Fix one such metric on $M$.  For each $x \in M$, we will denote by $\langle \cdot , \cdot \rangle_x$ the inner product in the tangent space $T_{x}M$.  

\begin{rem}
	\label{same topology}
	Note that the topology given by the smooth Riemannian metric is the given topology of $M$ (see \cite[pg. $311$]{RP63}).
\end{rem}

		Let $f \colon M \rightarrow \mathbb{R}$ be a smooth function on $M$.  
		Then $df \colon TM \rightarrow \mathbb{R}$,  the differential of $f$, is a cross-section of the cotangent bundle, $T^*M$, of $M$.  Hence, there is a uniquely determined vector field $\nabla f \colon M \rightarrow TM$, the \textbf{gradient of $f$}, such that $df_x(v) = \langle v, \nabla f(x) \rangle_x$ for all $x \in M$, $ v \in T_xM$.   
		
		The reader should note that $\nabla f$ will play a central role throughout this thesis.
		
		Recall that a \textbf{critical point of $f$} is a point $x \in M$ such that $df_x \colon T_xM \rightarrow \mathbb{R}$ satisfies $df_x = 0$, equivalently where $\nabla f_x$ vanishes.  Throughout this thesis let us denote the set of critical points of $f$ by $Crit(f)$, i.e., $$Crit(f) := \{ x \in M \hspace{2mm} | \hspace{2mm} df_x = 0 \}.$$  If $df_x \neq 0$ then the point $x \in M$ is called a \textbf{regular point of f}.  Let $c \in \mathbb{R}$.  
		If the level set $f^{-1}(c)$ consists only of regular points of $f$ then $c$ is a \textbf{regular value of $f$}.  If the level set $f^{-1}(c)$ contains at least one critical point of $f$ then we say that $c$ is a \textbf{critical value of $f$}.

\begin{note}
\label{Hessian}
	
	At a critical point $p$ of $f$ there is a uniquely determined continuous bilinear form $H_p(f) \colon T_{p}M \times T_{p}M \rightarrow \mathbb{R}$, the \textbf{Hessian of f at p}, such that if $\phi$ is any chart around $p$ $$H_{p}(f)(u, v) = d^{2}(f \circ \phi^{-1}) \left( d \phi |_p (u), d \phi |_p (v) \right),$$ where $d^2$ is defined below. 
\end{note}

\begin{rem}
\label{Hessian remark}
	\begin{enumerate}
		\item Suppose that $h$ is a continuously differentiable mapping of an open set $W$ of a Hilbert space $E$ into $\mathbb{R}$.  Then $dh$ is a continuous mapping of $W$ into the Hilbert space  $\mathcal{L}(E; \mathbb{R})$.  If that mapping is differentiable at a point $x \in W$, recall that $h$ is \textbf{twice differentiable} at $x$, and the derivative of $dh$ at $x$ is called the \textbf{second derivative} of $h$ at $x$, and written $d^2h |_x$.  This is an element of $\mathcal{L}(E; \mathcal{L}(E; \mathbb{R}))$.  We make the canonical identification of  $\mathcal{L}(E; \mathcal{L}(E; \mathbb{R}))$ with the space $\mathcal{L}(E \times E; \mathbb{R}))$ of continuous bilinear mappings of $E \times E$ into $\mathbb{R}$:  we recall that this is done by identifying $u \in \mathcal{L}(E; \mathcal{L}(E; \mathbb{R}))$ with the bilinear mapping $(s,t) \rightarrow (u \cdot s) \cdot t$.

		\item Note that the Hessian quadratic form in Definition \ref{Hessian} is independent of the choice of chart $\phi$.  Moreover, $H_p(f)$ determines a \underline{bounded} operator $A \colon T_pM \rightarrow T_pM$ by $$H_{p}(f)(u,v) = \langle Au, v \rangle_p$$  Because $H_p(f)$ is symmetric, the operator $A$ is self-adjoint.

		In what follows, we choose a smooth Riemannian metric and then identify $H_p(f)$ with the operator $A$.  The interpretation will be clear from the context.  
	\end{enumerate}
\end{rem}
	The critical point $p$ is called (strongly) \textbf{nondegenerate} if $A$ is invertible with bounded inverse.  Henceforth, we \textit{assume that $f$ has only nondegenerate critical points}.  

\begin{note}
	Let $p \in Crit(f)$.  The \textbf{index} (\textbf{coindex}) of $p$ is the index (coindex) of the Hessian $H_p(f)$, i.e., the supremum of the dimensions of all linear spaces where $H_p(f)$ is negative (positive) definite.  We shall denote the index of $p$ by $index_p(f)$ and the coindex by $coindex_p(f)$.
\end{note}
	 
\begin{eg}
	Let $\mathbb{H}$ be a Hilbert space and let $\mathbb{H}_{\pm} \subset \mathbb{H}$ be closed subspaces such that $\mathbb{H} = \mathbb{H}_{+} \oplus \mathbb{H}_{-}$.  Let $x := (x_+, x_- ) \in \mathbb{H}$ and let $f^{\mathbb{H}} \colon \mathbb{H} \rightarrow \mathbb{R}$ be a smooth function defined by $f^{\mathbb{H}}(x) = ||x_+||^2 - ||x_-||^2$.  For $p=0 \in Crit(f^{\mathbb{H}})$ we see that $\textrm{index}_p(f^{\mathbb{H}}) = \dim(\mathbb{H}_-)$ and $\textrm{coindex}_p(f^{\mathbb{H}}) = \dim(\mathbb{H}_+)$.
\end{eg}

\section{Two Important Theorems}
\medskip

\subsection{Baire Category Theorem}

\begin{note}  Let $M$ be a topological space.  A set $E \subset M$ is said to be \textbf{nowhere dense} if $\left( \overline{E} \right)^{\circ} = \emptyset$, i.e., $\overline{E}$ has empty interior.
\end{note}

Notice that $E$ is nowhere dense is equivalent to $$M = \left( \left( \overline{E} \right)^{\circ} \right)^c = \overline{ \left( \left( E^c \right)^{\circ} \right)}.$$
\noindent That is to say that $E$ is nowhere dense if and only if $E^c$ has dense interior.

\begin{thm}[Baire Category Theorem]
	\label{Baire}
	Let $M$ be a complete metric space.
	\begin{list}{}{}
	\item[(i)] If $\{ V_n \}_{n=1}^{\infty}$ is a sequence of dense open sets, then $\displaystyle \cap_{n=1}^{\infty} V_n$ is dense in $M$.
	\item[(ii)] If $\{ E_n \}_{n=1}^{\infty}$ is a sequence of nowhere dense sets, then $M \neq \displaystyle \cup_{n=1}^{\infty} E_n$.
	\end{list}	
\end{thm}

\begin{note}
 A subset $E \subset M$ is of \textbf{first Baire category} (or is \textbf{meager}) if $$E = \bigcup_{n=1}^{\infty} E_n$$
 \noindent where each $E_n$ is nowhere dense.  A set $F$ is called \textbf{residual} if $F^c$ is of first Baire category.
\end{note}

\begin{rem} The reader should think of first Baire category as being the topological analogue of sets of measure zero (so ``small''), and residual as being the topological analogue of sets of full measure (so ``big'').
\end{rem}

Let us collect some facts about residual sets and meager sets.  Let $M$ be a complete metric space.
\begin{enumerate}
\item A set $F \subset M$ is residual if and only if $F$ contains a countable intersection of open dense sets. \\
Indeed, if $F$ is residual then there exist nowhere dense sets $\{ E_n \}$ such that 
$$F^c  = \bigcup_{n=1}^{\infty} E_n \subset \bigcup_{n=1}^{\infty} \overline{E_n}.$$
Taking complements of this equation yields $$\bigcap_{n=1}^{\infty} \left( \overline{E_n} \right)^c \subset F,$$
\noindent i.e., $F$ contains a set of the form $\cap_{n=1}^{\infty} V_n$ where each $V_n  := \left( \overline{ E_n }\right)^c$ is an open dense subset of $M$.

\item A countable union of sets of first Baire category is of first Baire category.

\item If a set is of first Baire category then any subset of this set also is of first Baire category.

\item A countable intersection of residual sets is residual.
\end{enumerate}

\begin{rem}
	\label{Baire restatement}
	
	The Baire Category Theorem \ref{Baire} may now be re-stated as follows.  If $M$ is a complete metric space, then
	\begin{list}{}{}
	\item[(i)] all residual sets are dense in $M$, and
	\item[(ii)] $M$ is not of first Baire category.
	\end{list}	
 
\end{rem}

\subsection{Existence and Uniqueness Theorem for ODEs}

Let $M$ be an infinite-dimensional Hilbert manifold modelled on a real separable Hilbert space $( \mathbb{H}, \langle \cdot , \cdot \rangle )$.  Recall that given a smooth (meaning $C^{\infty}$) map $F \colon M \rightarrow \mathbb{R}^n$, a point $x \in M$ is called a \textbf{regular point of $F$} if the linear map $dF_x \colon T_xM \rightarrow T_{F(x)}\mathbb{R}^n$ is surjective.  A point $x \in M$ is called a \textbf{singular point of $F$} if it is not regular.  A point $y \in \mathbb{R}^n$ is called a \textbf{singular value} of $F$ if at least one point $x \in F^{-1}(y)$ is a singular point of $F$ and is called a \textbf{regular value} of $F$  if every $x \in F^{-1}(y)$ is a regular point of $F$, i.e., $y \in \mathbb{R}^n$ is called a regular value of $F$ if it is not a singular value for $F$.  Note that  if $F^{-1}(y) = \emptyset$, then $y$ is considered to be a regular value of $F$ because the definition of regular value is vacuously true. By the Implicit Function Theorem (see \cite{SL01} Chapter $1$, $\S 5$ page $19$), if $x$ is a regular point of $F$ and $y=F(x)$, then there is a neighbourhood $U_x \subset M$ of $x$  such that $U_x \cap F^{-1}(y)$ is a smooth submanifold of $M$.  Thus, if $y$ is a regular value of $F$ then $F^{-1}(y)$ is a smooth submanifold of $M$. 

Recall that if $X$ is a smooth vector field on $M$ then  
a \textbf{solution curve} for $X$ is a smooth map $\sigma$ of an open interval $(a, b) \subseteq \mathbb{R}$ into $X$ such that $\sigma'(t) = \left( X \circ \sigma \right) (t)$ for all $t \in (a,b)$. If $0 \in (a, b) $ and $x := \sigma(0)$ then we call $x$ the \textbf{initial condition} of the solution $\sigma$.
 
The next theorem is commonly called the local existence and uniqueness theorem for ordinary differential equations (or vector fields).  A detailed exposition of this fundamental theorem is presented in Chapter $IV$ of \cite{SL01} or Palais \cite{RP63} $\S 2$.

\begin{thm}[Local Existence and Uniqueness for Ordinary Differential Equations]
	\label{local ODE thm}
	
	 Let $X$ be a smooth vector field on an open set $\mathscr{O}$ in a Hilbert space $\mathbb{H}$.  Given $x \in \mathscr{O}$ there is a neighbourhood $U$ of $x$ included in $\mathscr{O}$, an $\epsilon > 0$, and a smooth map $\phi \colon U \times (- \epsilon, \epsilon ) \rightarrow \mathbb{H}$ such that:
\begin{enumerate}
	\item If $x' \in U$ then the map $\sigma_{x'} \colon (-\epsilon, \epsilon ) \rightarrow \mathbb{H}$ defined by $\sigma_{x'}(t) = \phi(x',t)$ is a solution of
$X$ with initial condition $x'$;
	\item If $\sigma \colon (a, b) \rightarrow \mathbb{H}$ is a solution curve of $X$ with initial condition $x' \in U$ then $\sigma(t) = \sigma_{x'}(t)$ for all $t \in (a,b) \cap (- \epsilon, \epsilon )$.
\end{enumerate}
\end{thm}

\begin{proof}
	See Palais \cite{RP63} $\S2$ or Lang \cite{SL01} Chapter $IV$.
\end{proof}

The next result is a consequence of Theorem \ref{local ODE thm} for vector fields.

\begin{lem}
	\label{max soln curve}
	Let $M$ be a Hilbert manifold and let $X$ be a smooth vector field on $M$.  For each $x \in M$ there exists a unique solution curve $\sigma_x$ of $X$ with initial condition $x$ such that every solution curve of $X$ with initial condition $x$ is a restriction of $\sigma_x$
\end{lem}

\begin{proof}
	See Palais \cite{RP63} $\S 6$.
\end{proof}

 The solution curve $\sigma_x$ above in Lemma \ref{max soln curve} is called the \textbf{maximum solution curve of $X$ with initial condition $x$}.  Define $\alpha \colon M \rightarrow (0,\infty]$ and $\beta \colon M \rightarrow [- \infty, 0)$ by the requirement
that the domain of $\sigma_x$ is $\left( \alpha(x), \beta(x) \right)$. The function $\alpha$ and $\beta$ are called respectively the positive and
negative escape time functions for $X$.

\begin{note}
	\label{domain D}
Let $M$ be a Hilbert manifold.  Let $$D := D(X) = \{ (x, t) \in M \times \mathbb{R} \hspace{2mm} | \hspace{2mm} \alpha(x) < t < \beta(x) \}$$ and for each $t \in \mathbb{R}$ let $D_t := D_t(X) = \{x \in M \hspace{2mm} | \hspace{2mm} (x, t) \in D \}$. Define $\phi \colon D \rightarrow M$ by $\phi(x,t) = \sigma_x(t)$ and $\phi_t \colon D_t \rightarrow M$ by $\phi_t(x) = \sigma_x(t)$. The set $\{ \phi_t \}$ is called the \textbf{maximum local one parameter group generated by X} or the \textbf{flow generated by $X$}.
\end{note}

\begin{thm}
	\label{cont ODE}
	
	In the set up of Definition \ref{domain D}, $D$ is open in $M \times \mathbb{R}$ and $\phi \colon D \rightarrow M$ is smooth.  For each $t \in \mathbb{R} \textrm{ the set } D_t$ is open in $M$ and $\phi_t$ is a smooth diffeomorphism of $D_t$ onto $D_{-t}$ having $\phi_{-t}$ as its inverse.  If $x \in D_t$ and $\phi_t(x) \in D_s$ then $x \in D_{t+s}$ and $\phi_{t+s}(x) = \phi_s \left( \phi_t (x) \right)$.
\end{thm}

\chapter{Normal Forms}
		
	The purpose of this chapter is to extend the existing theory on local normal forms for Hamiltonian group actions to infinite-dimensional Banach manifolds.  More specifically, we formalize the local linearization theorem for compact group actions on Banach manifolds and establish a symplectic version of this local linearization theorem.  In so doing, we provide a $G$-equivariant version of Moser's argument suitable for our goal.
	
	\section{Statements}

	Our initial result is similar to the finite-dimensional local linearization theorem for compact group actions, found in \cite{JK53}.  In fact, in \cite{AL69} Weinstein notes without proof that the local linearization theorem holds for smooth actions of compact groups on Banach manifolds.  Following this lead (and for the sake of completeness here), we state and prove the following version of the local linearization theorem.

		\begin{thm}[The Local Linearization Theorem]
		\label{locallinearization}
		 Let a compact Lie group $G$ act on a real Banach manifold $M$ and let $m$ be a fixed point. 
		 Then there exists a $G$-equivariant diffeomorphism $f$ from an invariant neighbourhood of the origin in $T_{m}M$ onto an invariant neighbourhood of $m$ in $M$.
	\end{thm}

		We shall now review some relevant definitions and notions to be used in a symplectic version of the local linearization theorem, Theorem \ref{locallinearizationsymplectic}.  In the process we will point out differences from the finite-dimensional case when necessary.
		
		To begin, we wish to call attention to the fact that there exist various definitions of differential forms and other related such concepts.  For example, see \cite[Chapter VIII: Infinite Dimensional Differential Geometry]{KM97}.  For our purposes, it is enough to use the definitions found in \cite[p.61 and p.124]{SL01}.  That is, if $E$ is a real Banach \emph{space} and $U$ an open chart of $E$, then a \textbf{differential form} of degree $r$ (or simply an \textbf{$r$-form}) on $U$ is an $r$-multilinear and alternating (in the last $r$ variables) smooth map $U \times E \times \cdots \times E \rightarrow E$.  Let $L_{a}^{r}(TU)$ denote the bundle of $r$-multilinear continuous alternating 
forms on $U$.  Then $L_{a}^{r}(TU)$ is equal to $U \times L_{a}^{r}(E)$.  Thus, a differential form of degree $r$ on $U$ is a section of $L_{a}^{r}(TU)$ and is entirely determined by the projection on the second factor $L_{a}^{r}(E)$. The usual definition of the exterior derivative, and the proof of the Poincar\'{e} lemma, apply without modification \cite{SL01}.   
		
		Next, recall that on a vector space $E$, a bilinear form $\omega :E \times E \rightarrow \mathbb{R}$ is said to be \textbf{weakly nondegenerate} if for every $v \in E$,
		\begin{equation}
			\label{def1}
			\left( \omega(v,w) = 0 \hspace{3mm} \forall w \in E \right) \Rightarrow v = 0.
		\end{equation}

		Now assume $E$ is a Banach space.  Its dual, $E^*$, is the space of bounded linear functionals on $E$.  Recall also that $\omega$ defines a linear map $\omega^{\sharp } \colon E \rightarrow E^{*} : u \mapsto \omega(u,\cdot)$.  So weak nondegeneracy means $ker(E \rightarrow E^{*}) = 0$, this is, $E \rightarrow E^{*}$ is injective.	 If this map is also surjective, then $\omega$ is said to be \textbf{strongly nondegenerate}.
		
		In what follows we require our symplectic form to be nondegenerate in the strong sense.  Let $M$ be a Banach manifold endowed with a closed differential $2$-form $\omega$, which at each $m$ in $M$ is strongly nondegenerate as a bilinear form on $T_{m}M$.  Said in other words, $T_{m}M \rightarrow T^{*}_{m}M$ is a linear homeomorphism.  \textit{Notice that continuity of the inverse of this map is equivalent to the openness of  $T_{m}M \rightarrow T^{*}_{m}M$, which immediately follows from the Open Mapping theorem as $T_{m}M \rightarrow T^{*}_{m}M$ is surjective here.}
		
	\begin{thm}[The Local Linearization Theorem - symplectic version]
		\label{locallinearizationsymplectic}
		Let a compact Lie group $G$ act on a strongly symplectic Banach manifold $(M, \omega)$.  Let $m$ be a fixed point. 
				Then there exists a $G$-equivariant symplectomorphism $f$  from an invariant neighbourhood of the origin in $T_{m}M$ onto an invariant neighbourhood of $m$ in $M$.
			\end{thm}

	\section{Proofs}
	
	The proof of Theorem \ref{locallinearization} is obtained by analogy with the finite-dimensional argument.  We begin with a simple lemma.  Suppose $G$ is a compact Lie group, and that $G$ acts on a Banach manifold $M$.

	\begin{lem}
		\label{equivariant1}
		Let $F$ be a diffeomorphism from an invariant neighbourhood of $m$ in $M$ onto a neighbourhood of the origin in a vector space $V$ such that $F(m)=0$.  Suppose that $G$ acts on $V$.  Define the $G$-average of $F$ as $\tilde{F}(u) := \int_{g \in G} (gF(g^{-1}u))dg$, where $dg$ is the normalized Haar measure on $G$.  Then the average $\tilde{F}$ is $G$-equivariant.
	\end{lem}

\begin{proof}
	Let $U \subset M$ be an invariant neighbourhood of $m$ and let $F : U \rightarrow V$ be any diffeomorphism onto a neighbourhood of the origin in $V$.  To ensure the  existence of such a diffeomorphism we use that there exists a chart near $m$ and that every open neighbourhood of $m$ contains an invariant open neighbourhood of $m$.  Let $\tilde{F} : U \rightarrow V$ given by $\tilde{F}(u) = \int_{g \in G} (gF(g^{-1}u))dg$ be its average.   We want to show $\tilde{F}(h \cdot u) = h \cdot \tilde{F}(u)$ $\forall h \in G$, $u \in U$.
	
	Consider  \begin{eqnarray*} \tilde{F}(h \cdot u) & = & \int_{g \in G} \left( g F \left( g^{-1} h \cdot u \right) \right) dg \textrm{ , by definition of } \tilde{F} \\
																   	   & = & h \left( \int_{g \in G} \left( h^{-1} g F \left( g^{-1}h \cdot u \right) \right) dg \right) \\
														 		   	   & = &  h \left( \int_{g \in G} h^{-1}g F \left( (h^{-1}g)^{-1} \cdot u \right) dg \right) \\
																   	   & = & h \left( \int_{g \in G} j F \left( j^{-1} \cdot u \right) dj \right)  \textrm{ , where $j=h^{-1}g$. Note $dg$ is invariant under $g \mapsto j$}\\
													   				   & = &	 h \cdot \tilde{F}(u) \textrm{ , as wanted.} \\
	\end{eqnarray*}
	
\end{proof}

	\begin{lem}
		\label{identity}
		Let $F$ be a diffeomorphism from an invariant neighbourhood of $m$ in $M$ to a neighbourhood of the origin in $V =T_{m}M$ with the isotropy action.   Let $\tilde{f}$ be its average.  Suppose the derivative of $F$ at $m$ is the 
		identity mapping on $T_{m}M$.  Then $d \tilde{f}|_{m} : T_{m}M \rightarrow T_{m}M$ is the identity.
	\end{lem}

\begin{proof}
	Let $U \subset M$ be an invariant neighbourhood of $m$ and let $F : U \rightarrow T_{m}M$ be any diffeomorphism onto a neighbourhood of the origin in $T_{m}M$.
	
	We have for all $g \in G$, $g : U \rightarrow U$ and $g_{*} : T_{m}M \rightarrow T_{m}M$.  By definition $dg|_{m} = g_{*}$ and $dg_{*}|_{0} = g_{*}$ because $g_*$ is a linear map and $dg_*$ is also linear.
	
	So the average of $F$ is $\tilde{f} := \int_{g \in G} \left( g_{*}F \left( g^{-1} \cdot u \right) \right) dg$.  Therefore,  
	
	\begin{eqnarray*} 
	d \tilde{f}|_{m}(\cdot) & = & \int_{g \in G} d \left( g_{*} F g^{-1} \right) |_{m}(\cdot) dg \\
						& = & \int_{g \in G} \left( dg_{*}|_{0} \circ dF|_{m} \circ dg^{-1}|{m} \right) (\cdot)dg \textrm{ , by the chain rule} \\
						& = &  \int_{g \in G} \left( g_{*} \circ dF|_{m} \circ g_{*}^{-1} \right) (\cdot)dg \textrm{ , by the above choice of notation and since $dg_{*}|_{0} = g_{*}$} \\
						& = & \int_{g \in G} \left( g_{*} \circ g_{*}^{-1} \right) (\cdot) dg \textrm{ , because } dF|_{m} = {\rm identity} \textrm{  by assumption} \\
						& = & \int_{g \in G} (\cdot) dg \\
						& = & \textrm{identity} 
	\end{eqnarray*}

\end{proof}

	\begin{proof}[Proof of Theorem \ref{locallinearization}]
	Let $U \subset M$ be an invariant neighbourhood of $m$.	Let $F : U \rightarrow T_{m}M$ be any smooth map such that $dF |_{m} : T_{m}M \rightarrow T_{m}M$ is the 
	identity mapping.  
		 
	 Take any $g \in G$.  Note $g$ acts on both $U$ and $T_{m}M$; $ g : U \rightarrow U$ and $ g_{*} : T_{m}M \rightarrow T_{m}M$.  Let $dg|_{m} = g_{*}$ and $g_{*}|_{0} = g_{*}$.
	
	Consider  $g_{*} \circ F \circ g^{-1} : U \rightarrow T_{m}M$.  By construction, this map is also a diffeomorphism such that its derivative at $m$ is the identity mapping on $T_{m}M$.  The average $\tilde{f} : 
	U \rightarrow T_{m}M$, which is defined by $\tilde{f}(u) :=  \int_{g \in G} \left( g_{*} F \left( g^{-1} \cdot u \right) \right) dg$ where $dg$ is the invariant Haar measure on $G$, is a $G$-equivariant diffeomorphism such that 
	$d \tilde{f}|_{m} = {\rm identity}_{T_{m}M}$ by lemma \ref{equivariant1} with $V = T_{m}M$ and lemma \ref{identity}. 

	By the inverse function theorem for Banach manifolds (see \cite{SL01}) we can invert $\tilde{f}$ on a neighbourhood of $m$ to obtain the desired diffeomorphism $f$, as required. 
		
	\end{proof}
	
	In the paper \cite{AL69} Darboux's theorem for Banach manifolds is explained.  In \cite{AL71} a remark as to how to establish an equivariant version of the Darboux-Weinstein theorem is made.  To help in the analysis in the proof of Theorem \ref{locallinearizationsymplectic}, we will need an equivariant local version of Moser's theorem.  Toward this end, and using similar techniques found in \cite{AL69} and \cite{AL71}, we will employ the next lemma.
			
	\begin{lem}[Moser's Theorem]
		\label{moserthm}
		Let $M$ be a Banach manifold with strongly symplectic forms $\omega_{0}$ and $\omega_{1}$.  Let $m$ be in $M$.   Assume $\omega_{0}$ and $\omega_{1}$ coincide on $T_{m}M$.  
		Then there exists a neighbourhood $U$ of $m$ and there exists a diffeomorphism $\psi$ from $U$ to an open subset of $M$ such that $\psi^*\omega_1 = \omega_0|_{U}$.
\end{lem}

	\begin{proof}
	Denote $\omega_{t} := (1-t) \omega_{0} + t \omega_{1}$, where $\omega_{0} := \psi^{*} \omega |_{m}$ and $\omega_{1} := \omega$.  By the Poincar\'{e} Lemma \cite{SL01}, there exists a $1$-form $\sigma$ on $U$ 
	such that $\omega_{1} - \omega_{0} = d \sigma$.  Observe that we can arrange for $\sigma|_{T_mM} = 0$.  We now look for a smooth, time dependent, vector field $X_{t} : M \rightarrow M$ on a neighbourhood of $m$ with $X_{t}|_{m}=0$ and $\iota(X_t)\omega_t = - \sigma$.

	The main idea is to determine a family of diffeomorphisms $\psi_t \in$ Maps($(U \rightarrow M)$ with $\psi_t^{*} \omega_t = \omega_{0}|_U$ by representing them as the flow of a family of time-dependent vector fields $X_t$ on a neighbourhood of $m$. Thus we suppose that			

		\begin{equation}
			\label{eqns}
			\frac{d}{dt} \psi_t = X_t \circ \psi_t, \,\,\psi_0 = \textrm{id}.
		\end{equation}

		So we know 
		\begin{eqnarray*}
			\psi_t^*\omega_t = \omega & \Leftrightarrow& \frac{d}{dt}\left( \psi^*\omega_t \right) = 0 \textrm{  , for all $t$}\\
									&\Leftrightarrow& \psi_t^* \left(  \frac{d}{dt} \omega_t + \mathcal{L}_{X_t}\omega_t  \right) = 0 \textrm{  , where $\mathcal{L}_{X_t}$ is the Lie derivative of $\omega_t$ along $X_t$}\\
									&\Leftrightarrow & \psi_t^* \left( d\sigma + \iota(X_t)\,d\omega_t + d(\iota(X_t)\omega_t)\right) = 0 \textrm{  , by using Cartan's formula and the choice of $\sigma$ }\\
									&\Leftrightarrow & \psi_t^* \left( d\sigma + d(\iota(X_t)\omega_t) \right) = 0 \textrm{  , since $\omega_t$ is closed by assumption}\\
									&\Leftrightarrow & X_t \textrm{ satisfies the linear (over $\mathbb{R}$) equation } d\sigma+d(\iota(X_t)\omega_t) = 0 \\
									&\Leftrightarrow & d(\sigma + \iota(X_t) \omega_t) = 0.
		\end{eqnarray*}

		This last identity will hold if 
		\begin{equation}
			\label{eqn2}
			\sigma + \iota(X_t)\omega_t = 0.
		\end{equation}

				Observe that for all $t$, $\omega_t$ is strongly nondegenerate at $m$.  Thus, there exists a neighbourhood $U$ of $m$ such that for all $t$ \hspace{1mm} $\omega_t$ is strongly nondegenerate on $U$.  Let $\omega_{t}(X_{t},\cdot) = - \sigma$ where $\omega_{t} : T_{m}M \rightarrow T^{*}_{m}M$, $(\sigma)_{m} \in T^{*}_{m}M$.  Recall that if $s \mapsto A_s$ is a smooth family of invertible operators then the family $A_s^{-1}$ of inverses is smooth.  So $X_{t} = - (\omega_{t})^{-1} \sigma$ is a smooth (and also smooth in $t$), time-dependent vector field taking values in $M$.  So, for any choice of 1-form $\sigma$ equation (\ref{eqn2}) can always be solved for $X_t$.  
		Therefore, (reading this argument backwards) we see that we can always find an $X_t$ that satisfies $\frac{d}{dt} \omega_t + \mathcal{L}_{X_t}\omega_t = 0$ .  
		
		Hence, by integrating $X_{t}$\footnote{ See \cite{SL01} chapters $IV$ and $V$ for explicit conditions that guarantee integrability of a vector field on a Banach manifold} (and shrinking $U$ again if necessary), there exists a family $\psi_t$ of diffeomorphisms such that (\ref{eqns}) holds. From this we easily deduce $\psi_t^*\omega_t = \omega_0|_{U}$  and accordingly the required conditions are satisfied.  Let $\psi = \psi_1$.  That is to say, there exists an isotopy $\psi : U \times [0,1] \rightarrow M : (q, t) \mapsto \psi_{t}(q)$,\hspace{1mm} $\psi_{t} \in$ Maps$(U \rightarrow M)$, and $\psi_0 = \textrm{id}$  with  $\psi^*\omega_t = \omega_0$ for all $t \in [0,1]$. 
	\end{proof}

	\begin{proof}[Proof of Theorem \ref{locallinearizationsymplectic}]
				Let $U \subset M$ be an invariant neighbourhood of $m$.  Proceeding in the same manner as the proof of Theorem \ref{locallinearization}, let $F : U \rightarrow T_{m}M$ be any smooth map such that $dF|_{m} = {\rm identity}_{T_{m}M}$.  The average $\psi : U \rightarrow T_{m}M$, given by $$\psi(u) :=  \int_{g \in G} \left( g_{*} F \left( g^{-1} \cdot u \right) \right) dg$$ \noindent where $dg$ is the Haar measure on $G$, is smooth, $G$-equivariant (c.f. Lemma \ref{equivariant1}), and satisfies $d \psi |_{m} = {\rm identity}_{T_{m}M}$ (c.f. Lemma \ref{identity}).

	Given a symplectic form $\omega$ on $M$, let $\omega_{0} := \psi^{*} (\omega |_{m})$ and $\omega_{1} := \omega$.  These are $G$-invariant symplectic forms on $U \subset M$.  Notice that $\omega_{0}$ and $\omega_{1}$ 
	coincide on $T_{m}M$ becuase $d \psi|_m = id_{T_xM}$.	 Consider now the family $\omega_{t} := (1-t) \omega_{0} + t \omega_{1}$ of closed $2$-forms on $U$.  We can assume that $\omega_{t}$ is a symplectic form for all $t \in [0,1]$ by shrinking 
	$U$ if necessary.  We want a $G$-equivariant map $\psi_{t} : U \rightarrow T_{m}M$ such that $\frac{d}{dt} \psi^{*}_{t} \omega_{t} = 0$.  That is, we need a local \emph{equivariant} Moser's theorem.  This map is obtained by
	Lemma \ref{moserthm} (applied to a neighbourhood of $m$) with an additional restriction.  The $G$-equivariance of the $\psi_t$ provided in \ref{moserthm} can be achieved by restricting the choice of $\sigma$ to $G$-invariant $\sigma$; all of the constructions can then be made `equivariantly' with respect to $G$.
	
	Therefore, by the inverse function theorem \cite{SL01} we invert $\psi$ on a neighbourhood of $m$ to get the desired symplectomorphism $f$.
	\end{proof}

\chapter{Connectedness - The Base Case}

	We begin this chapter by collecting some facts on Morse Theory and gradient flows which are relevant and needed to prove the main results of this chapter, the Connected Levels Theorem (Theorem \ref{connectedlevels}).

\section{Morse Functions and Their Gradient Flows}

\begin{lem}
	\label{flow on H}
	
	Let $\mathbb{H}$ be a Hilbert space and let $\mathbb{H}_{\pm} \subset \mathbb{H}$ be closed subspaces such that $\mathbb{H} = \mathbb{H}_{+} \oplus \mathbb{H}_{-}$.  Let $f^{\mathbb{H}} \colon \mathbb{H} \rightarrow \mathbb{R}$ be defined by $$f(x_+,x_-) = \| x_+ \|^2 - \| x_- \|^2.$$	Then the trajectory of $- \nabla f^{\mathbb{H}}$ starting at $x = (x_+,x_-) \in \mathbb{H}$ is given by
	$$t \mapsto (e^{-2t}x_+, e^{2t}x_-).$$ 
\end{lem}

\begin{proof}
	 Note that $ \left. (e^{-2t}x_+, e^{2t}x_-) \right|_{t=0}  = x$.  It is enough to show that $$\left. \frac{d}{dt} \right|_{t=0} \left( e^{-2t} x_+, e^{2t}x_- \right) \left( = \left( -2x_+, 2x_- \right) \right) = - \nabla f^{\mathbb{H}} \left|_x \right..$$  \noindent Recall that the gradient vector field $\nabla f^{\mathbb{H}}$on $\mathbb{H}$ is defined by the property that for all $x \in \mathbb{H}$, for all $v \in \mathbb{H}$, $df \left|_x \right. (v) = \langle \nabla f^{\mathbb{H}} \left|_x \right., v \rangle$.  So it is enough to show that for all $x, v \in \mathbb{H}$, $df \left|_x \right. (v) = - \langle (-2x_+, 2x_- ), v \rangle$.  
	 
	   Let $x = (x_+,x_-) \in \mathbb{H}$ and $v = (v_+, v_-) \in \mathbb{H}$.  Then
	  \begin{eqnarray*}
  df|_x(v) & = & df|_{(x_+,x_-)}(v_+, v_-) \\
  			 & = & D_{v_+} \left( ||x_+||^2 \right) - D_{v_-} \left( ||x_-||^2 \right) \textrm{ because $f^{\mathbb{H}}(x)=||x_+||^2-||x_-||^2$}\\
  			 & = & \left. \frac{d}{dt} \right|_{t=0} ||x_+ + tv_+ ||^2  - \left. \frac{d}{dt}\right|_{t=0} ||x_- + tv_- ||^2 \\
  			 & = & \left. \frac{d}{dt}\right|_{t=0} \left( ||x_+||^2 + 2t \langle x_+, v_+ \rangle + t^2 ||v_+||^2 \right) \\     
			 & - & \left. \frac{d}{dt}\right|_{t=0} \left( ||x_-||^2 + 2t \langle x_-, v_- \rangle + t^2 ||v_-||^2 \right) \\
			& = & 2 \langle x_+, v_+ \rangle - 2 \langle x_-, v_- \rangle \\
			& = & - \langle (-2x_+, 2x_-), v \rangle \\				  	  
	  \end{eqnarray*}

  Therefore, $(e^{-2t}x_+, e^{2t}x_-)$ gives the desired flow.
\end{proof}

	\begin{note}
	A smooth function $f \colon M \rightarrow \mathbb{R}$ on a Hilbert manifold $M$ is called a \textbf{Morse function} if all of its critical points are strongly nondegenerate.  That is, for every $x \in Crit(f)$, the operator $\nabla^2f |_x \colon T_xM \rightarrow T_xM$ obtained from the Hessian via the Riemannian metric is a linear isomorphism.
\end{note}	
	
\begin{rem}  \begin{enumerate}
\item Note that whether or not a function is Morse is independent of a choice of Riemannian metric.

\item Some references in the literature have weak nondegeneracy, that is the Hessian $H_p(f)$ induces only an injective map $\nabla^2f(x) \colon T_xM \rightarrow T_xM$, i.e. $\ker \left( \nabla^2f |_x \right)  = 0$, in their definition of a Morse function.

\end{enumerate}
\end{rem}	
	
	 In Morse theory, the Morse lemma introduces special coordinates around a critical point.  We recall this fundamental lemma now for Hilbert manifolds.
	 
\begin{lem}[The Morse Lemma]
	\label{Morse lemma}
	Let $f \colon M \rightarrow \mathbb{R}$ be a smooth function and let $p \in Crit(f)$.  Suppose that $p$ is strongly nondegenerate.  Then there exists an open neighbourhood $B \subset M$ of $p$ and a chart $\phi \colon B \rightarrow \mathbb{H}$ around $p$ with target a Hilbert space $\mathbb{H}$ such that $\phi(p) = 0$ and $\left( f \circ \phi^{-1} \right) (v) = \| Pv \|^2 - \| (I-P)v \|^2$ on $\phi(B)$, where $P$ is an orthogonal projection in $\mathbb{H}$ to a closed subspace (i.e., $Pv \in \mathbb{H}_+$ and $(I-P)v \in \mathbb{H}_-$ where $\mathbb{H} = \mathbb{H}_+ \oplus \mathbb{H}_-$).
\end{lem}

\begin{proof}
	See Palais \cite{RP63} page $307$.
\end{proof}

\begin{rem}
\begin{enumerate}
\item	It is an immediate consequence of the Morse Lemma that a nondegenerate critical point of a smooth function, say $f$, on a Hilbert manifold is isolated in $Crit(f)$.  In particular, if $f$ is a Morse function then the set $Crit(f)$ is discrete.

\item Note that weak nondegeneracy does not work in this setting; in fact weakly nondegenerate critical points need not be isolated in $Crit(f)$.  For example let $M = \ell_2 = \left\lbrace \{ x_k \} \subseteq \mathbb{R} \hspace{2mm} | \hspace{2mm} \displaystyle \sum_{k=1}^{\infty} |x_k|^2 < \infty \right\rbrace$.  Define $f \colon \mathbb{H} \rightarrow \mathbb{R}$ by $f(x) = - \displaystyle \sum_{k=1}^{\infty} \frac{ \cos(kx_k) }{k^4}$ ($f$ is smooth).  Then $0 \in Crit(f)$.  Moreover $0$ is weakly nondegenerate.  But any neighbourhood of $0$ has infinitely many critical points.  See \cite{JT77}, pg. $51$ for details.
\end{enumerate}
\end{rem}

	In the Morse Lemma \ref{Morse lemma}, the coordinate chart $\phi$ is called a \textbf{Morse chart} for the function $f$.  Note that the index at $p$ equals the dimension of the range of $I-P$ and the coindex of $p$ equals the dimension of the range of $P$, where $P$ is the projection from Lemma \ref{Morse lemma} (\cite{RP63} pg. $303$).
	
\begin{note}
	\label{standard def}
	Let $X$ be a vector field on a manifold $M$.  The vector field $X$ is said to be \textbf{standard near a point $p$} in $M$ if there exists a chart $\phi \colon U_p \rightarrow B_0 \subset \mathbb{H}$, where $B_0$ is a neighbourhood of $0$ in $\mathbb{H}$,  such that $p \mapsto 0$ and there exists a decomposition $\mathbb{H} = \mathbb{H}_+ \oplus \mathbb{H}_-$ such that $\phi$ intertwines the vector field near $p$ with the vector field on $\mathbb{H}$ whose value at the point $(x_+, x_-)$ is equal to $( -2 x_+, 2 x_- )$. 
\end{note}

\begin{rem}
\label{std rephrased}

\begin{enumerate}
\item Let $(M, g)$ be a Riemannian manifold and let $f$ be a smooth function on $M$.  If there exists a neighbourhood of a point $p \in M$ and a Morse chart near $p$ which is also an isometry (with respect to the metric on $\mathbb{H}$), then the gradient vector field $\nabla_gf$ of $f$ is standard near $p$. 

We will say that the Riemannian metric is standard (near each critical point $p$ of $f$) if 
the gradient vector field with respect to this metric is standard near each $p$.

\item  Note that the flow generated by a smooth vector field which is standard near a point $p$ is locally conjugate to the flow generated by its linearization.
\end{enumerate}
\end{rem}

Suppose that we are given a complete Riemannian metric $g$ on a Hilbert manifold $M$ and let $f \colon M \rightarrow \mathbb{R}$ be a smooth real-valued function on $M$.  Let us collect together some basic properties of $- \nabla_gf$, the negative gradient of $f$ with respect to $g$:

\begin{enumerate}	

\item   $- \nabla_gf$ has the property that $((\nabla_gf)f)(p) = 0$ if and only if $p \in Crit(f) \subset M$.  Therefore $Crit(f)$ is the set of zeros of the real-valued function $||\nabla_gf||$;

\item   The flow of the vector field $- \nabla_g f$ is a one-parameter group of diffeomorphisms $\rho^M_t \colon D_t \rightarrow M$ for $t \in \mathbb{R}$.  We require that  $\rho^M_0 = \textrm{id}$ and $ \frac{d \rho^M_t}{dt}\hspace{1mm}|_m = - \nabla_g f \hspace{1mm}|_{\rho^M_t(m)}$.

\item  The value of $f$ decreases along any non-constant flow line, $t \mapsto \rho^M_t$, of  $- \nabla_gf$.  We can easily see this, by Rolle's theorem, from the following calculation:
	\begin{eqnarray*} \frac{d}{dt} f \left( \rho^M_t ( \cdot ) \right) & = & df \left( \dot{\rho}^M_t ( \cdot ) \right) \textrm{ by def of $df$} \\
															& = & \langle \nabla_g f( \cdot ) \textrm{ , } \dot{\rho}^M_t ( \cdot ) \rangle \textrm{ by def of $\nabla_g f$} \\
															& = & \langle \nabla_g f( \cdot ) \textrm{ , } - \nabla_g f( \cdot ) \rangle \textrm{ by def of $\rho^M_t$} \\
														   	& = & - \| \nabla_g f( \cdot ) \|^{2} \\
												 		   	& \leq & 0 \
	\end{eqnarray*}
	with equality only if $p \in Crit(f)$.  That is, by Rolle's theorem, $( - \nabla_g f)(f)$ is negative off the critical set of $f$.

\end{enumerate}

	Next we establish that a Morse chart that is also an isometry intertwines the negative gradient flow on the neighbourhood with the negative gradient flow on the vector space.

\begin{lem}
  	\label{std flow}
	
	 Let $M$ be a Hilbert manifold and let $f \colon M \rightarrow \mathbb{R}$ be a Morse function. Let $p \in Crit(f)$ and let $U_p \subset M$ be a neighbourhood of $p$.  Let $\mathbb{H}$ be a Hilbert space and let $\mathbb{H}_\pm \subset \mathbb{H}$ be closed subspaces such that $\mathbb{H} = \mathbb{H}_{+} \oplus \mathbb{H}_{-}$.  
	 Let $\phi \colon U_p \rightarrow \mathbb{H}$ be an isometry such that $\phi(U_p) = B_{+} \times B_{-}$ where $B_{\pm} \subset \mathbb{H}_{\pm}$ are unit balls in $\mathbb{H}_{\pm}$ respectively.   Assume that $\phi$ is a Morse chart.  
	 Let $\rho_t^M$ be the gradient flow of $-f$ on $M$.  By Lemma \ref{flow on H} the negative gradient flow of $f^{\mathbb{H}}(x) = ||x_+||^2-||x_-||^2$ 	on $\mathbb{H}$ is $$\rho_{t}^{\mathbb{H}}(x_+,x_-) = (e^{-2t}x_+, e^{2t}x_-).$$  
	Then for all $t \in \mathbb{R}$ and for any $m \in U_p \cap (\rho^M_t)^{-1} (U_p) $, $$\phi \left( \rho_t^M (m) \right)  = \rho^{\mathbb{H}}_t \left(  \phi (m) \right). $$   
\end{lem}

\begin{proof}
		
	Let $ t \in \mathbb{R}$.  Let $m \in U_p \cap (\rho^M_t)^{-1} (U_p) $. \
	
	$$ \xymatrix{
                            & U_p \cap (\rho^M_t)^{-1} (U_p) \ar[r]_(.7){\phi} \ar[d]^{\rho^M_t} 
                            & \mathbb{H} \ar[d]_{\rho^{\mathbb{H}}_t}   \\
  							& \rho^M_t(U_p) \cap U_p \ar[r]^(.7){\phi} & \mathbb{H}
 } $$ \\
	
	We first show that $\phi$ intertwines the vector field $- \nabla_g f$ on $M$ with the vector field $\left(  x \mapsto (-2x_+, 2x_-) \right)$ on $\mathbb{H}$.  That is, we need to show that 
	$$d \phi_m \left( - \nabla_gf |_m \right) = \left(  x \mapsto (-2x_+, 2x_-) \right) |_{\phi(m)}.$$ 
	
	Consider $ d \phi_m \colon T_{m}U_p \rightarrow T_{\phi(m)} \mathbb{H}$.  Note that $T_{\phi(m)} \mathbb{H} = \mathbb{H}$ and that $T_mU_p = T_mM$ because $U_p$ is open.  So $d \phi_m$ is a bijective linear map between $T_mM$ and $\mathbb{H}$.  It follows that $d \phi_m(- \nabla_gf|_m) \in \mathbb{H}$.   But recall $\phi$ is a Morse chart and that $f^{\mathbb{H}}(x_+, x_-) = ||x_+||^2 - ||x_-||^2$ by hypothesis. Hence, $d \phi_m \left( - \nabla_gf|_m \right)$ decomposes into a positive and negative part.  Namely,  $d\phi_m \left( -\nabla_g f |_m \right) = - (2 x_+, -2x_- )$.  Since $\phi$ is an isometry it follows that $$d \phi_m \left( - \nabla_gf |_m \right) = \left(  x \mapsto (-2x_+, 2x_-) \right) |_{\phi(m)}$$ \noindent as wanted.
	
	Next we show that $\phi$ intertwines the flow $\rho^M_t$ on $U_p \subset M$ with the flow $\rho^{\mathbb{H}}_t$ on $B_+ \times B_- \subset \mathbb{H}$.  Assume that $t>0$.  The case $t < 0$ is similar.  Let $\gamma \colon [ 0 , t ] \rightarrow M$ be a maximal trajectory for $- \nabla_gf$ such that $\gamma(0)$, $\gamma(t) \in U_p$. Note that $\gamma^{-1} (U_p)$ is an interval.

\[ \xymatrix{
                    & [0,t] \ar[ld]_{\gamma} 
                                                  \ar[rd]^{\gamma^\star}
       \ar@{}[r]
                                        &  &  \\
 U_p \ar[rd]_{f} \ar[rr]^(.4){\phi} &                                                                                & B_+ \times B_-
 \ar[ld]^{f^{\mathbb{H}}=||x_+||^2-||x_-||^2}
        \\
                    &  \mathbb{R}                                                                        &  & }
                    \]
\medskip 
 
The diffeomorphism $\phi$ takes $\gamma$ to a maximal trajectory, say $\gamma^\star := \gamma \circ \phi$, in $B_+ \times B_-$ for $\left(  x \mapsto (-2x_+, 2x_-) \right) |_{B_+ \times B_-}$.  Since $\phi$ is a diffeomorphism between $U_p$ and $B_+ \times B_-$ that is also an isometry, we have that $$\rho_{t'}^M(m) \in U_p \textrm{ if and only if } \rho_{t'}^{\mathbb{H}}( \phi(m)) \in B_+ \times B_- \textrm{ for all $t' \in [0,t]$}.$$ 
 
 \noindent That is, the ``entry" and ``exit" values of $f$ (with respect to the flow $\rho^M|_{U_p})$ and $f^{\mathbb{H} } $ (with respect to the flow $\rho^{ \mathbb{H} }|_{ B_{+} \times B_{-} } $) are the same.
  
	\begin{figure}[htpb]
	\begin{center}
	\resizebox{11cm}{13cm}{\includegraphics{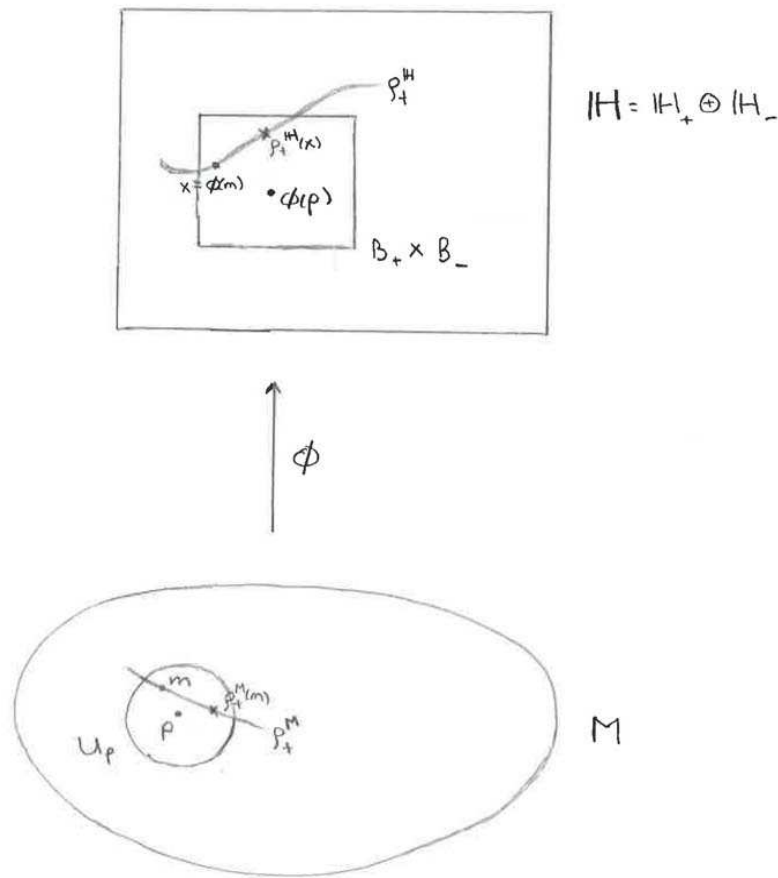}}
	\end{center}
	\caption{ Intertwining Gradient Flows}
	\label{fig:intertwine}
	\end{figure}
	 
\underline{On $M$}:  Consider $\rho^M_t$, an arbitrary flow line of $- \nabla_g f$ on $M$.  We know by definition that $\frac{d}{dt} \rho^M_t(m) = - \nabla_g f |_{\rho^M_t(m)}$.  So we have that $$\frac{d}{dt} f \left( \rho^M_t (m) \right)  = \left( ( - \nabla_g f ) f \right) (m) = - || \nabla_g f(m) ||^2.$$  This implies that $f \left( \rho^M_t (m) \right) $ is (monotonically) decreasing in $t$, i.e., $f$ is decreasing along non-constant flow lines of $- \nabla_g f$.
 We define the \textit{entry time} of $\rho^M_t$ on $U_p$ as the point \newline $t_\rho := \inf \{ \tau \hspace{2mm} | \hspace{2mm} [0, \tau ] \subseteq \gamma^{-1} \left( U_p \right) \}$.  Then the \textit{entry point} of $\rho^M_t$ on $U_p$ is $x_\rho := \gamma(t_\rho) \in \overline{U_p} \subseteq M$.  
 Similarly, we define the \textit{exit time} of $\rho^M_t$ on $U_p$ as the point 
 $\tilde{t}_\rho := \sup \{ \tau \hspace{2mm} | \hspace{2mm} [ \tau, t ] \subseteq \gamma^{-1} \left( U_p \right) \}$.  Then the \textit{exit point} of $\rho^M_t$ on $U_p$ is $y_\rho := \gamma (\tilde{t}_\rho) \in \overline{U_p} \subseteq M$. 
 
 From these entry/exit point definitions we see that $f(x_\rho) > f(y_\rho)$ since $f$ is decreasing along $\rho^M_t$. 

\underline{On $\mathbb{H}$}:  Recall again that by assumption, $\phi$ is a diffeomorphism between $U_p$ and its image $\phi(U_p) = B_+ \times B_-$.  So $d \phi_m$ is a bijective linear map between the sets $\{$ vector fields on $M$ $\}$ and $\{$ vector fields on $\mathbb{H}$ $\}$.  Consequently, all that remains is to consider $\rho^{\mathbb{H}}_t$, the corresponding flow of $- \nabla_g f$ on $\mathbb{H}$.  Recall that $\rho^{\mathbb{H}}_t(x_+, x_-) = ( e^{-2t}x_+, e^{2t}x_-)$ by Lemma \ref{flow on H}.  Suppose that $||B_+|| = ||B_-|| = 1$.  Observe that $\rho^{\mathbb{H}}_t$ meets $B_+ \times \partial B_-$ at the point $\left( ||x_-||x_+, \frac{1}{||x_-||}x_- \right) $.  Also observe that $\rho^{\mathbb{H}}_t$ meets $\partial B_+ \times B_-$ at the point $\left( \frac{1}{||x_+||} x_+, ||x_+||x_- \right) $.

 We define the \textit{entry point}, respectively \textit{exit point}, of $\rho^{\mathbb{H}}_t$ with $B_+ \times B_-$ as follows: \\
 \noindent
 Case 1:  If both $x_+$ and $x_-$ are nonzero, then the entry point is $\left( \frac{1}{||x_+||} x_+, ||x_+||x_- \right) $ and the exit point is $\left( ||x_-||x_+, \frac{1}{||x_-||}x_- \right) $.\\
 Case 2:  If $x_+ = 0$ but $x_- \neq 0$, then $\rho^{\mathbb{H}}_t(x_+, x_-) = (0, e^{2t}x_-)$.  Therefore, for large $t \hspace{2mm} \rho^{\mathbb{H}}_t$ never meets $\partial B_+ \times B_-$.  That is, $\rho^{\mathbb{H}}_t$ never meets $B_+ \times B_-$.  For $t \ll 0$, the entry point is  $\left( 0,  \frac{1}{||x_-||}x_-  \right) $ and there is no exit point.  That is, $\rho^{\mathbb{H}}_t$ enters $B_+ \times B_-$ and converges to $\phi(p)=(0,0) \in B_+ \times B_-$.\\
 Case 3:  If $x_-= 0$ but $x_+ \neq 0$, then $\rho^{\mathbb{H}}_t(x_+, x_-) = (e^{-2t}x_+, 0)$. So for large $t$, the entry point of $\rho^{\mathbb{H}}_t$  is  $\left( \frac{1}{||x_+||} x_+, 0 \right) $ and there is no exit point because $\rho^{\mathbb{H}}_t$ never meets $B_+ \times \partial B_-$.  That is, $\rho^{\mathbb{H}}_t$ enters $B_+ \times B_-$ and converges to $\phi(p)=(0,0)$, i.e.,  $\rho^{\mathbb{H}}_t$ never exits.   For $t \ll 0$, $ \rho^{\mathbb{H}}_t$ never meets $\partial B_+ \times B_-$.  That is, $\rho^{\mathbb{H}}_t$ never meets $B_+ \times B_-$. \\
Case 4.  If $(x_+, x_-) = (0,0)$ then $\rho^{\mathbb{H}}_t (x_+, x_-)$ is constant.  For all $t \in \mathbb{R}$, $\rho^{\mathbb{H}}_t$ will either never meet $B_+ \times B_-$ or it will enter at the point $\left( \frac{1}{||x_+||}x_+,||x_+||x_- \right)$ and exit at the point $\left( ||x_-||x_+,\frac{1}{||x_-||}x_- \right)$.

\noindent Thus, $f^{\mathbb{H}} > 0$ at each entry point and $f^{\mathbb{H}} < 0$ at each exit point for the flow on $B_+ \times B_-$.  Consequently, $f > 0$ at each entry point and $f < 0$ at each exit point for the flow on $U_p$.  Hence, if any trajectory on $M$ exits $U_p$ it does not return.

It now follows from the local existence and uniqueness results for ODEs (see Lang \cite{SL01} Chapter IV) , that our result $\phi \left( \rho_t^M (m) \right)  = \rho^{\mathbb{H}}_t \left(  \phi (m) \right) $ holds.
\end{proof}

	The last lemma shows us that near each critical point of $f$ we can always modify a Riemannian metric on $M$ so that the negative gradient vector field of $f$ is standard near each critical point of $f$.  Stated more precisely, 

\begin{lem}
	\label{modify stdmetric}
	
	Let $M$ be a Hilbert manifold.  Let $f \colon M \rightarrow \mathbb{R}$ be a Morse function.  Let $g$ be a Riemannian metric on $M$.  For each $p \in Crit(f)$, let $U_p$  be a neighbourhood of $p$.  Then there exists a Riemannian metric $\tilde{g}$ on $M$ such that:	
\begin{list}{}{}
	\item{(i)} for all $p \in Crit(f)$ there is a neighbourhood $V_p$ of $p$ in $U_p$ such that $- \nabla_{\tilde{g}} f$ is standard on $V_p$;
	\item{(ii)} $\tilde{g}$ coincides with $g$ outside of $\hspace{1mm} \bigcup_{p \in Crit(f)} U_p$
\end{list}
\end{lem}

\begin{rem}
	This lemma serves as motivation for Lemma \ref{condition C preserved} in Section $\S 4.3$ (Connected Levels) which gives a direct proof of a stronger result.
\end{rem}

\begin{proof}  
	We can shrink $U_p$ such that the $\overline{U_p}$ are disjoint.  Let $p \in Crit(f)$.  By the Morse Lemma \ref{Morse lemma}, there exists a neighbourhood $B_p \subseteq U_p$ of $p$ and a Morse chart $\phi_p \colon B_p \rightarrow \mathbb{H}$ such that $\phi_p(p)=0$ and $\left( f \circ \phi_p^{-1} \right) (v) = \| Pv \|^2 - \| (I-P)v \|^2$ on $\phi_p(B_p)$.  
		
	Let $\lambda_p \colon \mathbb{H} \rightarrow \mathbb{R}$ be a bump function.  That is, let $\lambda_p$ be a smooth function satisfying: \\
	$\bullet \hspace{2mm} 0 \leq \lambda_p(x) \leq 1$, and \\
	$\bullet \hspace{2mm} \lambda_p(x) = 1$ near $0$, and \\
	$\bullet \hspace{2mm} supp\left( \lambda_p (x) \right) \subseteq \phi_p(B_p)$. \\
	
	Let $m \in B_p$ and $X, Y \in T_mB_p$.  Then define the new metric  

 	 \[ \tilde{g}|_m(X,Y) = \left\{ \begin{array}{ll}
	\displaystyle \left( 1 - \lambda_p( \phi(m) ) \right) g|_m(X,Y) + \lambda_p( \phi(m)) \langle X_m,Y_m \rangle_{\phi(m)} & \textrm{ if $m \in U_p$} \\
	g|_m & \textrm{ if $m \not\in \cup_{p \in Crit(p)} U_p$ .} 
\end{array} \right. \] 	

\noindent where $\langle \cdot, \cdot \rangle_{\phi(m)}$ denotes the inner product coming from $\mathbb{H}$.

	By construction this new metric $\tilde{g}$ satisfies properties $(i)$ and $(ii)$, as wanted.  
\end{proof}

\section{Stable and Unstable Manifolds}

	Let us start this section by reviewing some known definitions and giving some important assumptions.  We will then state and prove the Global (Un)Stable Manifold Theorem \ref{W^s submfld}.  Lastly, we finish this section by examining a couple of additional results pertaining to the stable manifold.

\begin{note}
	\label{stable unstable set def}
	Let $M$ be a Hilbert manifold.  Let $f \colon M \rightarrow \mathbb{R}$ and let $p \in Crit(f)$.  Fix a metric $g$ on $M$.  The \textbf{stable set} $W^{s}(p)$ of $p$ is defined to be the set of all points $x \in M$ such that the $- \left( \nabla_{g}f \right) $-trajectory $\rho^M_{t}(x)$ starting at $x$ is defined for all $t$ in $\mathbb{R}^+$ and $\displaystyle \lim_{t \rightarrow \infty} \rho^M_{t}(x) = p$.  That is, $$W^{s}(p) = \{ x \in M \hspace{2mm} | \hspace{2mm} \rho^M_{t}(x) \textrm{ is defined for all $t \in \mathbb{R}^+$ and } \lim_{t \rightarrow \infty} \rho^M_{t}(x) = p \}.$$

	The \textbf{unstable set} $W^{u}(p)$ of $p$ is defined to be the set of all points $x \in M$ such that the $- \left( \nabla_{g}f \right) $-trajectory $\rho^M_{t}(x)$ starting at $x$ is defined for all $t$ in $\mathbb{R}^-$ and $\displaystyle \lim_{t \rightarrow - \infty} \rho^M_{t}(x) = p$.  That is, $$W^{u}(p) = \{ x \in M \hspace{2mm} | \hspace{2mm} \rho^M_{t}(x) \textrm{ is defined for all $t \in \mathbb{R}^-$ and } \lim_{t \rightarrow - \infty} \rho^M_{t}(x) = p \}.$$

\end{note}

	In the rest of this section we assume that $M$ is a complete Riemannian Hilbert manifold (see below) and $f \colon M \rightarrow \mathbb{R}$ is a Morse function that is bounded from below and satisfies Condition (C).  By complete we mean that $M$ is a complete metric space in the metric induced from the Riemannian metric.
	
	For the reader's convenience we recall how this metric on $M$ is defined.  Given $x$ and $y$ in $M$ we define $$\rho(x,y) = \inf \int_0^1 \| \sigma '(t) \| dt$$ where the infimum is over all $C^1$ paths $\sigma \colon [0,1] \rightarrow M$ such that $\sigma (0) =x$ and $\sigma (1) =y$.  Just as in the finite dimensional case one shows that $\rho$ is a metric on $M$ which is consistent with the manifold topology (see Palais \cite{RP63}, $\S 9$ pg. $311$).  
	
	We recall Condition (C) of Palais and Smale for $f$:
	\label{Condition (C)}
	
	\begin{list}{}{}
	\item Condition (C) (\textbf{Palais-Smale condition}):
	\item If $\{ x_n \} \subset M$ is any sequence in $M$ for which $|f(x_n)|$ is bounded and for which $||df|_{x_n}|| \rightarrow 0$, then $\{ x_n \}$ has a convergent subsequence $\{x_{n_k} \} \rightarrow p$
	\end{list}
	
	\begin{rem}
	\begin{enumerate}
	\item	If $M$ is finite dimensional and compact then for \underline{any} choice of Riemannian metric for $M$ the completeness, the boundedness below and the Condition (C) assumptions are automatically satisfied.  Note also that if $M$ is finite dimensional but not necessarily compact then Condition (C) for a smooth real-valued function is satisfied automatically for proper maps.
	  
	\item Condition (C) is a condition on $f$ that for many purposes can replace the compactness of the manifold.  As a rule in extending finite dimensional results in differential topology to infinite dimensions, we transfer the compactness condition from the space $M$ itself to the function on $M$.
		\end{enumerate}
	\end{rem}
	The Global (Un)Stable Manifold Theorem, Theorem \ref{W^s submfld}, is an important result that tells us that the sets $W^s(p)$ and $W^u(p)$ are (immersed) submanifolds of $M$ that have the same codimension as the stable and unstable subspaces, respectively, of the linearization of $f$ at $p$.  The proof of Theorem \ref{W^s submfld} is an adaptation of the proof presented in \cite[Chapter $1$, $\S 1.7$]{BCL06}.

\begin{lem}[The Global (Un)Stable Manifold Theorem]
	\label{W^s submfld}

	Let $M$ be a Hilbert manifold.  Let $f \colon M \rightarrow \mathbb{R}$ be a Morse function and let $p \in Crit(f)$.  Fix a Riemannian metric on $M$ such that the negative gradient 
	vector field of $f$ is standard near $p$. Then $W^s(p)$ is a connected submanifold of $M$ of codimension equal to $index_{p}(f)$ and $W^u(p)$ is a connected submanifold of $M$ of codimension equal to $coindex_{p}(f)$.
\end{lem}

\begin{proof} 
	Let $p \in Crit(f)$ and let $U \subset M$ be a neighbourhood of $p$.  Let $\rho^M_t$ be the negative gradient flow of $f$ on $M$.

	The local stable set of $p$ (relative to $U$) is defined as the set 
	\begin{eqnarray*} 
	W_{loc}^s(p) & := & \{ x \in U \hspace{2mm} | \hspace{2mm} \rho_t^M(x) \textrm{ is defined for all $t \geq 0$, $\rho_t^M(x) \in U \hspace{2mm} \forall \hspace{1mm} t \geq 0$ and } \lim_{t \rightarrow \infty} \rho_{t}^M(x) = p \} \\
						& = & \{ x \in U \hspace{2mm} | \hspace{2mm} \rho_t^U(x) \textrm{ is defined for all $t \geq 0$ and } \lim_{t \rightarrow \infty} \rho_{t}^U(x) = p \} \
	\end{eqnarray*}

	\noindent where $\rho^M_U$ is the negative gradient flow of $f$ on $U$.
		
	Let $D^U_t$ be the domain of definition of $\rho^U_t$.  Then $W_{loc}^s(p)$ may be equivalently expressed as the set $$\{ x \in U \hspace{2mm} | \hspace{2mm} x \in D^U_t \textrm{ for all $t$ } \geq 0 \textrm{ and } \lim_{t \rightarrow \infty} \rho_{t}^U(x) = p \} $$
	Similarly, the local unstable set of $p$ (relative to $U$) is defined as the set
	\begin{eqnarray*} 
	W_{loc}^u(p) & := & \{ x \in U \hspace{2mm} | \hspace{2mm} \rho_t^U(x) \textrm{ is defined for all $t \leq 0$ and } \lim_{t \rightarrow -\infty} \rho_{t}^U(x) = p \}. \\
						& = & \{ x \in U \hspace{2mm} | \hspace{2mm} x \in D^U_t \textrm{ for all $t \leq 0$ and } \lim_{t \rightarrow -\infty} \rho_{t}^U(x) = p \}. \
	\end{eqnarray*}
	
	Note that $W_{loc}^s(p) \subseteq W^s(p)$ and  $W_{loc}^u(p) \subseteq W^u(p)$.  Moreover, $W_{loc}^s(p)$ and $W_{loc}^u(p)$ are both nonempty since they each contain $p$.
	
	Let $\mathbb{H}$ be a Hilbert space and let $\mathbb{H}_\pm \subset \mathbb{H}$ be closed subspaces such that $\mathbb{H} = \mathbb{H}_+ \oplus \mathbb{H}_-$.  We shall identify a neighbourhood of $p$ with a neighbourhood of $0$ in $\mathbb{H}$.  Let $\phi \colon U \rightarrow \mathbb{H}$ be a Morse chart with properties:

	$\bullet \hspace{2mm} \phi$ is an isometry, and \\
	\indent $\bullet \hspace{2mm} \phi(U) = B_+ \times B_-$ where $B_\pm \subset \mathbb{H}_\pm$ are unit balls in $\mathbb{H}_\pm$ respectively.\
	
	Note that we have $$\mathbb{H}_{+} := \{ x \in \mathbb{H} \hspace{2mm} | \hspace{2mm} \lim_{t \rightarrow \infty} \rho_t^{\mathbb{H}}(x) = 0 \}$$ $$\mathbb{H}_{-} := \{ x \in \mathbb{H} \hspace{2mm} | \hspace{2mm} \lim_{t \rightarrow - \infty} \rho_t^{\mathbb{H}}(x) = 0 \}$$ 
	\noindent where $\rho_t^\mathbb{H}(x) = ||x_+||^2 - ||x_-||^2.$
	
	\medskip
		 \[ \xymatrix{
                   &
 M \supset U \ar[rd]_{f} \ar[rr]^(.4){\phi} &                   & B_+ \times B_- \subset \mathbb{H}_+ \oplus \mathbb{H}_- \ar[ld]^{||x_+||^2-||x_-||^2}
      \\
                 &&  \mathbb{R}                                                                          && }
                    \]
	\medskip
	
	It follows that $W^{\mathbb{H},s}\left( \phi (p) \right) = W^{\mathbb{H},s}(0) = \mathbb{H}_+$ and $W^{\mathbb{H},u}\left( \phi (p) \right) = W^{\mathbb{H},u}(0) = \mathbb{H}_-$.

	The proof of this Lemma requires that:
	
	\noindent \underline{Step 1}:  We must show that  $W^s_{loc}(p)$ ( respectively, $W^u_{loc}(p)$ ) is a manifold. \\
	\noindent \underline{Step 2}:  We must extend the local results of Step 1 to $W^s(p)$ ( respectively $W^u(p)$ ). \\

	Step 1:  We wish to show that the set $W^s_{loc}(p)$ is a submanifold of $U$.
	
	By Lemma \ref{std flow}, recall that $\phi$ intertwines the flow on $U \subset M$ with the flow on $B_+ \times B_- \subset \mathbb{H}$.  More precisely, $\phi \colon U \rightarrow B_+ \times B_-$ is a diffeomorphism such that for all $t \in \mathbb{R}$ and for any $m \in U_p \cap (\rho^M_t)^{-1} (U_p) $ we have that 
	\begin{eqnarray*}
	\phi \left( \rho_t^{M} (m) \right) & = & \rho_t^{\mathbb{H}} \left( \phi (m) \right) \\
				& = & \rho_t^{\mathbb{H}} (x_+, x_- ) \textrm{ because $\phi(m) = ( x_+, x_-) \in B_+ \times B_-$} \\
				& = & ( e^{-2t}x_+, e^{2t}x_- ) \textrm{, by Lemma \ref{flow on H}.}
	\end{eqnarray*}
	
 \noindent Thus, it is sufficient to show that $W^{\mathbb{H},s}_{loc} \left( \phi (p) \right) = B_+ \times \{ 0 \}$. 
 
\medskip				
				$$ \xymatrix{
                            & \mathbb{R} \cup  U  \ar[r]^(0.35){\textrm{Id} \times \phi}
                            & \mathbb{R} \times B_+ \times B_-         \\ 
                            & { | \bigcup} & { | \bigcup} 												  \\
                             & D_t \ar[d] \ar[r]^{\cong} & D^{\mathbb{H}}_t \ar[d] \\
                            & U  \ar[r]^{ \phi} 
                            & B_+ \times B_- \\
                                & { | \bigcup} & { } 												  \\
  							& W^s_{loc}(p)  &  B_+ \times \{0\} \ar[uu]_{i} 
 } $$ 
		\medskip 

	\noindent Note that $W^{\mathbb{H},s}_{loc} \left( \phi (p) \right) \subseteq W^{\mathbb{H},s} \left( \phi (p) \right)$. Moreover, recall that  $W^{\mathbb{H},s} \left( \phi (p) \right) = \mathbb{H}_+$ and that $W^{\mathbb{H}, s}_{loc} \left( \phi(p) \right) = W^{\mathbb{H},s} \left( \phi(p) \right) \cap \phi(U)$.  Therefore,  
	 \begin{eqnarray*}
	W^{\mathbb{H},s}_{loc} \left( \phi (p) \right) & = & \mathbb{H}_+ \cap ( B_+ \times B_- ) \\
												& = & B_+ \times \{ 0 \} \
	\end{eqnarray*}
 	\noindent as wanted.  By the properties of $\phi$, observe that $W^s_{loc} \left( \phi (p) \right)$ is connected.
	
	\noindent Therefore $W^{s}_{loc}(p)$ is a connected submanifold of $M$ which contains $p$ with codimension $index_{p}(f)$.

	Step 2:  By using $\rho^M_t$, the negative gradient flow of $f$ on $M$, we wish to extend the local results of Step 1 to the global stable manifolds $W^s(p)$ and $W^u(p)$. 
	
Fix an $x \in M$.  	Fix a time $T \in \mathbb{R}$.  Suppose that $\rho^M_T \colon ( \rho^M_T )^{-1}(U) \rightarrow U \cap D^M_{-T}$ is a diffeomorphism.
	
	\medskip
		$$ \xymatrix{
                            &D_T^M \ar@<1ex>[r]^{\rho^M_T}  
                            & D_{-T}^M \ar@<1ex>[l]^{\rho^M_{-T}}    \\  
                            & { \bigcup} & {\bigcup} \\
  							& (\rho^M_T)^{-1}(U) \ar[r]^{\rho^M_T} & U \cap D^M_{-T} 
 } $$ 
	\medskip

\noindent Note that the set $( \rho^M_T )^{-1}(U) \subseteq M$ is open because $\rho^M_T$ is continuous.  To prove Step 2 it is enough to show that $W^s(p) \cap  ( \rho^M_T )^{-1}(U)$ is a submanifold of $M$.
		\medskip				

				$$ \xymatrix{
                            & \hbox{$\overbrace{(\rho^M_T)^{-1}(U)}^{\textrm{open}}$} \ar[r]^{\rho^M_T}_{\cong}  
                            & \hbox{$\overbrace{U \cap D_{-T}^M }^{\textrm{open}}$} \\
                            & { | \bigcup} & { | \bigcup} 												  \\
  							& W^s(p) \cap (\rho^M_T)^{-1}(U) \ar[r]^{\rho^M_T}_{\cong}   & W^s_{loc}(p) \cap D^M_{-T} 
 } $$

		\medskip

	We claim that:  $$q \in W^s(p) \cap \left( \rho_T^M \right) ^{-1}(U) \textrm{ if and only if } \rho_T^M(q) \in W^s_{loc}(p).$$  
	
	\noindent It will follow from the claim that the image under $\rho^M$ of $W^s(p) \cap \left( \rho_T^M \right) ^{-1}(U)$  is equal to the submanifold $W^s_{loc}(p) \cap \left( U \cap D^M_{-T} \right)$.   In other words, the set $W^s(p)$ inherits the structure of a manifold from that of $W^s_{loc}(p)$ by the set of maps $\{ \rho^M( t, \cdot ) \}$.  Therefore $W^{s}(p)$ is a connected submanifold of $M$ which contains $p$ with codimension $index_{p}(f)$. 
	
	Proof of claim:  $(\Rightarrow)$ Let $q \in W^s(p) \cap \left( \rho_t^M \right)^{-1}(U)$.  Then $q \in W^s(p)$ and $q \in \left( \rho_t^M \right)^{-1}(U)$.  This implies, respectively, that $\rho_t^M(q) \in W^s(p)$ and $\rho_t^M(q) \in U$.   So $\rho_t^M(q) \in W^s(p) \cap U$.  But $W^s(p) \cap U = W^s_{loc}(p)$ (this follows from the fact that all entry values of $f$ (with respect to $\rho^M$) are bigger than all exit values.  This fact appeared in the proof of Lemma \ref{std flow}).
	
	$(\Leftarrow)$ Let $\rho^M_T(q) \in W^s_{loc}(p)$.  That is, $ q \in \rho^M_{-T} \left( W^s_{loc}(p) \right)$.  However \begin{eqnarray*}
	\rho^M_{-T} \left( W^s_{loc}(p) \right) & = & \left( \rho_T^M \right) ^{-1}\left( W^s_{loc}(p) \right)\textrm{ , by Theorem \ref{cont ODE}  $(\rho^M_{T})^{-1}=\rho^M_{-T}$} \\
											& = & (\rho_T^M )^{-1} \left( W^s(p) \cap U \right) \\
											& = & W^s(p) \cap (\rho_T^M )^{-1}(U) \
	\end{eqnarray*} 
	
\noindent Thus $q \in W^s(p) \cap (\rho_T^M )^{-1}(U)$, and completing the proof of the claim.

	It follows that $W^s(p)$ is a submanifold of $M$.
	
	The analogous results for $W^u(p)$ follows by giving all of the same arguments as above but by considering the vector field $\nabla_gf$ (instead of $-\nabla_gf$).
\end{proof}

\begin{rem}  Both a Local (Un) Stable Manifold theorem and a Global (Un)Stable Manifold theorem for Banach manifolds exist in the literature (\cite[Chapter 1]{BCL06}, \cite[Chapters 5 and 6]{MS87}).  These references do \textbf{not} assume that the vector field is standard a point in the manifold.  Let us briefly review what is known:  
\begin{enumerate}
\item Known proofs of the \underline{Local} (Un)Stable Manifold theorem are based on methods such as the \textit{``graph transform method''} or the \textit{``orbit space method''}.  A brief description of these methods is provided below. 
	
	\begin{list}{}{}
	\item $\bullet$ For detailed information on the so called \textit{``graph transform method''} see \cite{MS87}; 1987, Chapter 5.  The Hadamard approach, this so called ``graph transform method", to proving the Local (Un)Stable Manifold theorem uses what is known as a graph transform.  This method constructs the stable and unstable manifolds as graphs over the linearized stable and unstable spaces, respectively.
	This method is more geometrical in nature than the next Liapunov-Perron orbit space method.
	\item $\bullet$ For detailed information on the so called \textit{``orbit space method''} see \cite{BCL06}: Chapter 1.  The Liapunov-Perron orbit space method is another approach used to  prove the Local (Un)Stable Manifold theorem.   This method (in the context of ordinary differential equations) deals with the integral equation formulation of the ordinary differential equations and constructs the invariant manifolds as a fixed point of an operator that is derived from the integral equation of a function whose elements have the appropriate interpretations as stable and unstable manifolds.
	\end{list} 
	
\item A complete proof for the \underline{Global} (Un)Stable Manifold Theorem is also given in \cite{BCL06}: Chapter 1, Section $\S$ 1.7.  This proof identifies $W^s(p)$ and $W^u(p)$ as particular images of injective immersions of manifolds.  Note, again, that all of the aforementioned results are established for  \textit{Banach} manifolds.  In particular they are true for Hilbert manifolds.  Their proofs become simpler in the Hilbert manifold setting. For example, if $M$ is a Hilbert manifold in the Global (Un)Stable Manifold Theorem \cite{BCL06}, then the regularity of the norm implies that $W^s(p)$ and $W^u(p)$  are actually images of  the tangent space to $W^s(p)$, say $E^s_p$, and the tangent space to $W^u(p)$, say $E^u_p$, (respectively) where $T_pM = E^s_p \oplus E^u_p$.
\end{enumerate}
\end{rem}

\begin{lem}
	\label{transverse}
	
	Let $M$ be a Riemannian Hilbert manifold and $f \colon M \rightarrow \mathbb{R}$ a Morse function.  Let $x$ be a regular point for $f$.  Fix a Riemannian metric on $M$ such that for every critical point $p$ of $f$ the negative gradient vector field of $f$ with respect to that Riemannian metric is standard near $p$.  Then there exists a neighbourhood $U_x$ of $x$ in $M$ such that $U_x \cap f^{-1}\left(  f(x) \right) $ is a manifold.  Moreover, let $p \in Crit(f)$.  Then, after possibly shrinking $U_x$, the set $\left( U_x \cap f^{-1}\left(  f(x) \right) \right)  \cap W^s(p)$ is a submanifold of $U_x \cap f^{-1}\left(  f(x) \right) $ with codimension equal to $index_{p}(f)$.  This submanifold either passes through $x$ or is empty.

\end{lem}

\begin{proof}
	By the Implicit Function theorem we know that there exists a neighbourhood $U_x$ of $x$ such that $U_x \cap f^{-1}\left(  f(x) \right) $ is a smooth manifold and that 
	$$T_{x} f^{-1}\left(  f(x) \right) = ker(df|_{x} :T_{x}M \rightarrow \mathbb{R}).$$  
	
	Let $p \in Crit(f)$.  If $W^s(p) \cap \{ x \} \neq \emptyset$ then we claim that $U_x \cap f^{-1}( f(x) )$ is transverse to $W^s(p)$ at $x$ (hence, near $x$).  By the definition of transversality, it suffices to find a $v \in T_{x}W^{s}$ such that $d_{x}f(v) \neq 0$.  Take $v = - \nabla_g f_x$, the negative $g$-gradient of $f$ at $x$.  
	
	From transversality, it follows that after possibly shrinking the neighbourhood $U_x$, the set $\left( U_x \cap f^{-1}\left(  f(x) \right)\right)  \cap W^{s}(p)$ is a smooth submanifold of $U_x \cap f^{-1}\left(  f(x) \right) $ and that the codimension of 	$\left( U_x \cap f^{-1}\left(  f(x) \right)\right)  \cap W^{s}(p)$ in $U_x \cap f^{-1}\left(  f(x) \right) $ is equal to the codimension of $W^{s}(p)$ in $M$.  This codimension is equal to $index_{p}(f)$.
\end{proof}

	Recall that $M$ is a connected Riemannian Hilbert manifold and $f \colon M \rightarrow \mathbb{R}$ a Morse function.  Fix a Riemannian metric on $M$ such that $f$ is bounded from below and satisfies Condition (C).  Let $\{ p_i \}$, $ i \in I$ be the set of critical points of index equal to $0$.  Define $$M_{0} := \bigsqcup_{i \in I} W^s(p_i).$$  Thus, $M_0$ is the disjoint union of the (open) stable manifolds with index zero.  

\begin{lem}
	\label{complement}
	
	Let $M$ be a complete connected Riemannian manifold and $f \colon M \rightarrow \mathbb{R}$ a Morse function that is bounded from below.  Fix a Riemannian metric on $M$ such that $f$ satisfies Condition (C) and that for every critical point $p$ of $f$ the negative gradient vector field of $f$ with respect to the Riemannian is standard near $p$.  Suppose that none of the critical points of $f$ have index equal to $1$. Then the complement of $M_0$ is a locally finite union of submanifolds of codimension at least two.
\end{lem}

\begin{rem}
	Recall that a collection of subsets of a topological space is said to be \textbf{locally finite}, if each point in the space has a neighbourhood that intersects only finitely many of the sets in the collection.  
\end{rem}

\begin{proof}
	From Palais \cite{RP63} we know that: 
	
\begin{list}{}{}
\item{(i)} $\hspace{2mm}$ (Prop. $1$ pg.$314$) if $a$, $b \in \mathbb{R}$ then there is at most a finite number of critical points $p$ of $f$ that satisfy $a < f(p) < b$.
\item{(ii)} $\hspace{2mm}$ (Prop. $3$ pg.$321$) if $\sigma_{t}(x)$ is any maximal solution curve  of $- \nabla_g f$ starting at the point $x$, then $\sigma_t(x)$ is defined for all $t > 0$, and $\displaystyle \lim_{t \rightarrow \infty} \sigma_{t}(x)$ exists and is a critical point of $f$.
\end{list} 

	Note that for each $c \in\mathbb{R}$, the set $\{ \hspace{1mm} x \in M \hspace{1mm} | \hspace{1mm} f(x) < c \hspace{1mm} \}$ is open in $M$ because $f$ is continuous.  Moreover, each point $x \in M$ is contained in at least one of these sets.  Thus for all $ c \in \mathbb{R}$,

	$$\{ \hspace{1mm} x \in M \hspace{1mm} | \hspace{1mm} f(x) < c \hspace{1mm} \} \cap \left( M \smallsetminus M_0 \right) \underbrace{=}_{ \textrm{by $(ii)$} } \{ \hspace{1mm} x \in M \hspace{1mm} | \hspace{1mm} f(x) < c \hspace{1mm} \} \cap 
	\bigcup_{p \in a}W^s(p).$$
where $a=Crit(f)$ such that $index_p(f) \geq 2$  and $f(p) < c$ by $(i)$ the union is finite.

But recall, by Lemma \ref{W^s submfld} we know that codim$\left(W^s(p) \right) = \textrm{ index}_p(f)$,which is greater than or equal to two.  Therefore, $M \smallsetminus M_0$ is a locally finite union of submanifolds with codimension at least two. 	
\end{proof}

\begin{lem}
	\label{M_0 connected}
	
Let $M$ be a complete connected Riemannian manifold and $f \colon M \rightarrow \mathbb{R}$ a Morse function that is bounded from below.  Fix a Riemannian metric on $M$ such that $f$ satisfies Condition (C) and that for every critical point $p$ of $f$ the negative gradient vector field of $f$ with respect to the Riemannian is standard near $p$.  Suppose that none of the critical points of $f$ have index equal to $1$.  Then $M_0$ is connected.
\end{lem}

\begin{proof}
Let $M_0$ be as in Lemma \ref{complement}.  Recall that $M_0 = \sqcup_{i \in I} W^{s}(p_i)$ where $\{ p_i \}$ $(i \in I)$ is the set of critical points of index equal to zero.  Lemma \ref{complement} ensures that $I \neq \emptyset$.  By hypothesis, no critical points of $f$ have index equal to $1$, so $M_{0}^{c}$ is a locally finite union of submanifolds of codimension at least 2 by Lemma \ref{complement}.  This implies that $M_0$ is connected.  We give more details:
	
	For each $x \in M$, there exists a neighbourhood $U_x$ of $x$ such that $U_x \cap M_0$ is path connected and dense in $U_x$. This can be established by using Lemma \ref{complement} and the definition of a submanifold.

	Let $p$, $q \in M_0$, $p \neq q$.  Let $\gamma \colon [0,1] \rightarrow M$ be such that $\gamma(0) = p$ and $\gamma(1) = q$.  Choose $U_{x_i}$ as above, $i = 0, \ldots , N-1$, such that  
	
	\begin{list}{}{}
	\item $\bullet$ the collection of $U_{x_i}$ cover the path $\gamma$, and
	\item $\bullet \hspace{2mm} U_{x_i} \cap U_{x_{i+1}} \neq \emptyset$ for all $i$, and
	\item $\bullet \hspace{2mm} p \in U_{x_0}$, $q \in U_{x_N}.$
	\end{list}
	
		\begin{figure}[htpb]
	\begin{center}
	\includegraphics{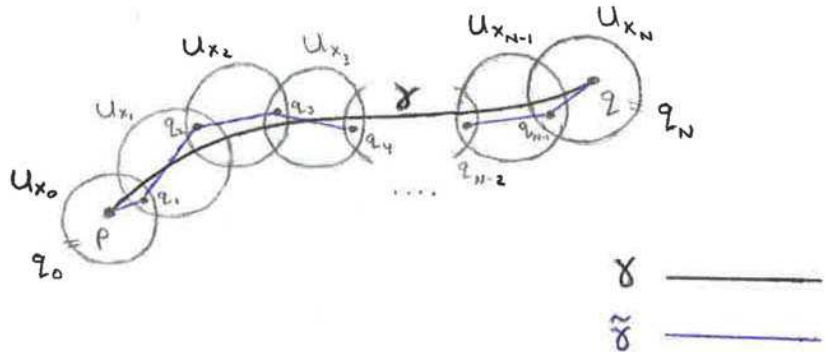}
	\end{center}
	\caption{ Construction of a path $\tilde{\gamma}$ which avoids $M_0^c$}
	\label{fig:cover}
	\end{figure}
					
	It follows that for all $x$, $U_{x_i} \cap U_{x_{i+1}} \cap M_0$ is nonempty.   Let $q_0 =p$, $q_N = q$.   For each $i = 0, \ldots, N-2$ choose a point $q_{i+1} \in U_{x_i} \cap U_{x_{i+1}} \cap M_0$.  For each $i = 0, \ldots , N-1$, we may construct a path $\gamma_{i+1}$ connecting $q_{i}$ to $q_{i+1}$ in $U_{x_i} \cap M_0$.  We can do so because $U_{x_i} \cap M_0$ is path connected.  Now so as to finish concatenate the $\gamma_{i+1}$ to construct a path $\tilde{\gamma} := \gamma_1 \gamma_2 \cdots \gamma_N$.  Notice that $\tilde{\gamma}$ is a path between $p$ and $q$ which does not intersect $M_0^c$, by construction. That is, $\tilde{\gamma}$ is a path in $M_0$. Hence, the open set $M_0 \subset M$ is path connected and so also connected.
	
\end{proof}

\begin{rem}
	\label{global min value}
	\begin{enumerate}
	\item In the set-up of Lemma \ref{M_0 connected}, $f$ attains its global minimum  since the critical point set of $f$ is discrete (see \cite[Section $\S 15$, 	Theorem $4$, Corollary $2$]{RP63}).

\end{enumerate}
\end{rem}

The following notation will from here on will be used throughout this thesis:  from remark \ref{global min value}, let $p_0 \in Crit(f)$ denote the unique critical point of $f$ with index zero and let $f(p_0) := c_0$ denote the global minimum value of $f$ on $M$.

\section{Connected Levels}

	Let us start with a couple of technical lemmas.  The first lemma provides a list of properties satisfied by a metric $g$ whose gradient vector field, $\nabla_g f$, is standard near a critical point $p$ of $f$.
	
\begin{lem}
	\label{standardproperties}
	
	Let $M$ be a complete connected Riemannian Hilbert manifold and $f \colon M \rightarrow \mathbb{R}$ a Morse function that is bounded from below and satisfies Condition (C).  Let $M_0$ be the open stable manifold with index zero.  Suppose that, for every critical point $p$ of $f$ not in $M_0$,  
	the Riemannian metric on $M$ is standard near $p$. Suppose that none of the critical points of $f$ have index equal to $1$.  Then for each $x \in M$ there exists a connected neighbourhood $U_x$ of $x$ such that
	\begin{list}{}{}
	\item[(i)] $U_{x} \cap M_0$ is open, connected, and dense in $U_x$, and
	\item[(ii)] if $x$ is a regular point of $f$ then for all $c \in \mathbb{R}$, $(U_{x} \cap M_0) \cap f^{-1}(c)$ is open, connected, and dense in $U_{x} \cap f^{-1}(c).$
	\end{list}	
\end{lem}

\begin{proof}
	Recall that $M_{0} = W^{s}(p_0)$ where $p_0 \in Crit(f)$ is the unique critical point of index zero.  Then $M_0$ is open and connected by Lemma \ref{W^s submfld}. Let $x \in M$.  Choose a connected neighbourhood $U_x$ of $x$.  \\
	\indent \underline{For property $(i)$};   Let $E = M_{0}^{c}$. Observe that $U_x \cap M_0 = U_{x} \setminus E $.  The $U_{x} \cap M_{0}$ is open because $M_0$ is.  It follows that $U_x \cap M_0$ is open in $U_x$.  Also note that $E$ is a locally finite union of submanifolds of $M$ of codimension $2$ or more, by Lemma \ref{complement}.  Hence, $U_x \cap M_0 \subset U_x$ is connected and dense in $U_x$.  

	\underline{For property $(ii)$}; Let $c \in \mathbb{R}$.

	If $x$ is a \textit{regular point} of $f$ then $U_x \cap M_0 \cap f^{-1}(f(x))$ and $f^{-1}(f(x)) \cap U_x$ are path connected by the implicit function theorem.  It follows that $$(U_x \cap M_0) \cap  f^{-1}\left( f(x) \right) \subset U_x \cap  f^{-1}\left( f(x) \right)$$ is open (in the relative topology).  By Lemma \ref{transverse} we know  that the set $(U_x \cap M_0) \cap  f^{-1}\left( f(x) \right) $ is a smooth submanifold of $U_x \cap f^{-1}\left( f(x) \right)$ with codimension equal to $index_{p_0}(f) \geq 2$.  Then it follows that $$(U_x \cap M_0) \cap  f^{-1}\left( f(x) \right) \subset U_x \cap  f^{-1}\left( f(x) \right)$$ is connected and dense because its complement has codimension at least $2$. \

\end{proof}

Let $f$ be a Morse function on a connected Riemannian manifold $M$.  Let $d$ be the distance function coming from the Riemannian metric $g$ on $M$.  Note that the set $Crit(f)$ has no accumulation points. This follows by the Morse Lemma \ref{Morse lemma} applied to $f$.

For each point in $M$ there exists a neighbourhood $D \subset M$ and a chart with target a Hilbert space.  Let $\phi$ be a chart in $M$ having as target a Hilbert space $( \mathbb{H} , \left\langle \cdot , \cdot \right\rangle )$.  For each $ x \in D$ let $G(x) \colon \mathbb{H} \rightarrow \mathbb{H}$ be an operator defined as in Definition \ref{smooth riemannian str}.  
Then each $G(x)$ is an invertible linear operator that is bounded with bounded inverse. 
Recall that by Lemma \ref{modify stdmetric},  
for each critical point $p \in Crit(f)$, there exists a
neighbourhood $U_p \subset M$ of $p$ on which there is a standard metric (cf. remark \ref{std rephrased})  $g_p$.  For each $x \in U_p (:= D)$ we define the operator $G_p(x) \colon \mathbb{H} \rightarrow \mathbb{H}$ as above.  Using ingredients similar to Palais \cite[Lemma $2$ pg $311$]{RP63},  we can shrink $U_p$ such that there exist constants $a_p := ||G_p||, b_p := ||G^{-1}_p|| > 0$ such that throughout the neighbourhood
$$\frac{1}{b_p}\|v\|_{g_p} \le\|v\|_g\le a_p\|v\|_{g_p}$$
\noindent for all $x \in U_p$, for all $v \in T_xM$.

Since $Crit(f)$ has no accumulation points, for each $p\in Crit(f)$ there exists $R_p>0$
which is less than half the distance (in the distance function $d$)
from $p$ to any other point in $Crit(f)$. 
Thus the balls of radius $R_p$ (in the metric space $(M,d)$)
about $p$ do not intersect.

Since $Crit(f)$ is countable, write $Crit(f) =\{p_1,p_2,p_3,\ldots, p_j,\ldots\}$. 
For each $j=1,\ldots\infty$, let $U_{p_j}$ be the open ball of radius $r_j$
about $p_j$ (in the distance~$d$) where $r_j$ is chosen to be sufficiently
small so that
$$r_j<\min\,\{R_{p_j},\frac{1}{2j}\}$$
and $U_{p_j}$ is contained in the domain of $g_{p_j}$.

Set $U:=\cup_{j=1}^\infty U_{p_j}$,
$\hat{U}:=\cup_{j=1}^\infty \overline{U_{p_j}}$, and $V:=M \smallsetminus \hat{U}$.

\begin{lem}\label{Vopen}
$V$ is open.
\end{lem}

\begin{proof}
Suppose not.
Then there exists a convergent sequence $(x_m)\to x$ such that $x\in V$
and $x_m\in\hat{U}$ for all~$m$.

Each set $\overline{U_{p_j}}$ can contain only finitely many points from the
sequence~$(x_m)$ since otherwise the limit~$x$ would lie
in~$\overline{U_{p_j}}$.

For each $m$, find $j_m$ such that $x_m\in \overline{U_{p_{j_m}}}$.

Given~$n$, since $(x_m)\to x$ there exist infinitely many~$m$ such that
$d(x,x_m)<\frac{1}{2n}$.
In particular, since only finitely many~$x_m$ lie in any~$\overline{U_{p_j}}$,
there exists $m$ such that $d(x,x_m)<\frac{1}{2n}$ and $j_m>n$.

Since $x_m\in \overline{U_{p_{j_m}}}$, we have
$$d(x_m, p_{j_m})<r_{j_m}<\frac{1}{2j_m}<\frac{1}{2n}.$$
Thus
$$d(x,p_{j_m})\le d(x,x_m)+d(x_m,p_{j_m})
<\frac{1}{2n}+\frac{1}{2n}=\frac{1}{n}.$$

However the existence for each $n$ of an element of $Crit(f)$ whose distance to $x$
is less than $1/n$ shows that $x$ is an accumulation point of $Crit(f)$, contrary
to the fact that $Crit(f)$ has no accumulation points.
Therefore there is no such sequence~$(x_m)\to x$ and so $V$ is open.
\end{proof}

Given a connected Riemannian Hilbert manifold $(M, g)$ and a Morse function $f$ on $M$ that satisfies Condition (C) with respect to $g$, the next technical lemma shows us that for each critical point $p$ of $f$, we can modify the metric $g$ in a neighbourhood of $p$ so that Condition (C) continues to hold for $f$ with respect to this new metric on this neighbourhood.  

\begin{lem}
	\label{condition C preserved}
	
	Let $M$ be a connected Hilbert manifold.  Let $f \colon M \rightarrow \mathbb{R}$ be a Morse function and let $g$ be a complete Riemannian metric on $M$ such that $f$ satisfies condition (C). Then there exist neighbourhoods $U_p$ of $p$ for each $p \in Crit(f)$ such that the $\overline{U_p}$ are disjoint and there exists a Riemannian metric $g_{new}$ on $M$ such that:

	 	\begin{list}{}{}
	\item[(i)] $g_{new}$ is standard. 
	\item[(ii)] $g_{new}$ coincides with $g$ outside of $U_p$.

	\item[(iii)] $g_{new}$ is complete, and $f$ satisfies condition (C) with respect to $g_{new}$.
	\end{list}	
\end{lem}	
	
\begin{proof}
Choose neighbourhoods $U_p$ so that Lemma \ref{Vopen} applies.  By Lemma \ref{modify stdmetric}  the existence of standard metrics (\ref{std rephrased})  $g_p$ on neighbourhoods of $p$ is guaranteed.  As in Palais \cite[Lemma $2$ pg $311$]{RP63}, we use similar ingredients to show that, 
for each $p \in Crit(f)$ we can shrink $U_p$ such that there exist constants $a_p$, $b_p >0$ such that $$\frac{1}{b_p}||v||_{g_p} \leq || v ||_g \leq a_p ||v||_{g_p}$$ \noindent for all $x \in U_p$, for all $v \in T_xM$.

For each $p$, choose a bump function $\kappa_p:M \to \mathbb{R}$ for the
neighbourhood~$U_p$.  That is, let $\kappa_p$ be a smooth function with: \\
	$\bullet \hspace{2mm} 0 \leq \kappa_p(x) \leq 1$, and \\
	$\bullet \hspace{2mm} \kappa_p(x) = 1$ near $p$, and \\
	$\bullet \hspace{2mm} \textrm{ supp} \left( \kappa_p(x) \right) \subseteq U_p$.\

For $x\in M$, define a new metric by

 	 \[ g_{new}|_x = \left\{ \begin{array}{ll}
	\displaystyle \left( 1 - \kappa_p(x) \right) g|_x + a_p \kappa_p(x) g_p|_x & \textrm{ if  $x \in U_p$} \\
	g|_x & \textrm{ if $x \not\in \cup_{p \in Crit(f)} U_p$ .} 
\end{array} \right. \]

Then $g_{new}$ satisfies (i)--(iii) by construction.  (Note that $g_{new}$ is a Riemannian metric because $V$ as defined in Lemma \ref{Vopen} is open).

\begin{claim}
	\label{compareg}
 $|| \cdot ||_{g_{new}} \geq || \cdot ||_g$. 
\end{claim}

\begin{proof}
If $x\in U_p$ then

\begin{eqnarray*}
||\cdot||^2_{g_{new}} \left|_x \right. & = & \left( 1 - \kappa_p(x) \right) ||\cdot||_g^2 \left|_x \right.+ \underbrace{a_p^2 \kappa_p(x) ||\cdot||^2_{g_p} |_x}_{ \geq  \kappa_p(x) || \cdot ||_g^2 |_x } \\
									& \geq &  \left( 1 - \kappa_p(x) \right) ||\cdot||_g^2 \left|_x \right. +  \kappa_p(x) || \cdot ||_g^2 |_x \\
									& = & \left( 1 - \kappa_p(x) + \kappa_p(x) \right) ||\cdot||_g^2 |_x \\
									& = & ||\cdot||_g^2 |_x . \\
									\end{eqnarray*}

and if $x$ lies outside $U_p$ for every~$p$ then $g_{new}|_x=g_x$.  So $|| \cdot ||_{g_{new}} || |_x  \geq  || \cdot ||_g |_x$.
\end{proof}

\bigskip
{\bf To show $g_{new}$ is complete}

Let $(x_m)$ be a Cauchy sequence in the distance function coming from~$g_{new}$.

By Claim \ref{compareg} the sequence~$(x_m)$ is also a Cauchy sequence in the
distance function coming from $g$.
Since $(M,d)$ is a complete metric space, there exists $y\in M$ such that
$(x_m)\to y$ in the distance function $d$.  Recall that convergence with respect to one of these metrics implies convergence with respect to the other  because the topology induced by these two metrics is the same (see \ref{same topology}).

\smallskip

{\bf To show $f$ satisfies condition $(C)$ with respect to $g_{new}$}

Let $\{x_n\}\subset M$ be a sequence for which $|f(x_n)|$ is bounded.
Suppose that
$$\|df|_{x_n}\|^2_{g_{new}}:=\langle df|_{x_n},df|_{x_n}\rangle_{g_{new}}\to 0.$$
We wish to show that $(x_n)$ has a subsequence which converges to a critical
point.

By Lemma~\ref{compareg}, 
$$\|df|_{x_n}\|^2_{g_{new}}\ge\|df|_{x_n}\|^2_g$$
and so
$$\|df|_{x_n}\|^2_g\to 0.$$
The fact that $f$ satisfies condition (C) with respect to $g$ gives a
subsequence $(x_{n_k})$ of $(x_n)$ which converges 
to a critical point $y$.  Say $(x_{n_k})\to y$. 
Again recall the fact that convergence with respect to one of these metrics implies convergence with respect to the other because the topology induced by these two metrics is the same  (see \ref{same topology}). 
So $(x_n)$ has a convergent subsequence, as desired.

{\bf End of Proof of Lemma \ref{condition C preserved}}

 \end{proof}

	We are now prepared to prove the connectivity for each level set of $f$. 

\begin{thm}[Connected Levels]
		\label{connectedlevels}
		
		Let $M$ be a connected Hilbert manifold and let $f \colon M \rightarrow \mathbb{R}$ be a Morse function that is bounded from below and none of whose critical points have index or coindex equal to $1$.  Suppose that there exists a complete Riemannian metric on $M$ such that $f$ satisfies condition $(C)$.  Then the level set $f^{-1}(c) \subset M$ is connected for every $c$ in $\mathbb{R}$.  
		
\end{thm}	

\begin{proof}
	By the definition of a Morse function, each of the critical points of $f$ is (strongly) nondegenerate.  By the Morse Lemma for Hilbert manifolds \cite{RP63}, each critical point of $f$ on $M$ is isolated. 
	
	By Lemma \ref{condition C preserved} there exists a complete Riemannian metric, call it $g$, on $M$ for which $f$ satisfies Condition (C) and such that $- \nabla_g f$ is standard near each critical point.  Consider the vector field $- \nabla_g f$.

	Recall that $f$ has only one critical point of index zero, say $p_0$.  Moreover, $f$ attains its global minimum value $c_0 := f(p_0)$ on $M$ (see remark \ref{global min value}).  Also recall that Palais (see \cite{RP63} Proposition$1$ pg $314$) proves that if $a$, $b \in \mathbb{R}$ then there is at most a finite number of critical points $p$ of $f$ satisfying $a < f(p) < b$.  Hence, the critical values of $f$ are isolated and there are at most a finite number of critical points of $f$ below any critical level since $f$ is bounded from below by assumption.  Let $c_0 < c_1 < c_2 < \cdots$ be the critical values of $f$.

	\indent Let $c \in Im(f)$ 
 such that $c > c_0$.  \\
 \noindent \textbf{Case I}: \textit{for any regular point of $f$ in $f^{-1}(c)$, $f^{-1}(c)$ is connected } \\ Let $E  = M_{0}^{c}$. 
 Note that by Lemma \ref{standardproperties} (ii), any regular point in $f^{-1}(c)$ can be connected by a continuous path in $f^{-1}(c)$ to a point that belongs to $M_0$ (a `totally descending point') of $f^{-1}(c)$.  Thus, following the method of Bryant \cite{RB03}, to prove the connectedness of $f^{-1}(c)$ it suffices to show that any two totally descending points of $f^{-1}(c)$ can be joined by a continuous path in $f^{-1}(c)$.  Let us give more details. 
	
	Suppose that $x$, $y \in f^{-1}(c)$ are regular points of $f$ such that $x \neq y$.  Then there exist neighbourhoods $U_x$ and $U_y$ of $x$ and $y$, respectively, that satisfy the properties of Lemma \ref{standardproperties}(ii).  So we can choose totally descending points, say $x'$ and $y'$, in $f^{-1}(c)$, which connect to $x$ and $y$ in $f^{-1}(c)$.  Moreover, by Lemma \ref{standardproperties}(i) and the Morse Lemma \ref{Morse lemma}, there exists a (`controlled') neighbourhood $U_0$ of $p_0 \in M$ such that for all $c$, the set $U_0 \cap f^{-1}(c)$ is connected and such that $U_0 \subset M_0$.
	
	We pass to the \textit{normalized} gradient flow.  Note that by Palais \cite{RP63} there exists a time $t \in \mathbb{R}$ 
	such that the \textit{normalized} forward (downward) flow lines of $x'$ and $y'$ belong to $U_0$.

	\vspace{5mm}
 \textit{The fact that the gradient flow lines are normalized means that their speed of descent is one, and therefore level sets map to level sets.}  We make explicit use of this fact throughout this proof.\
 \newline

	More precisely, let $\psi_t$ denote the downward normalized flow.  There are points $x'' := \psi_{t}(x')$ and $y'' := \psi_{t}(y')$ (i.e., $x''$ lies on the forward normalized flow line of $- \nabla_gf$ through $x'$, similarly $y''$ lies on the forward normalized downward flow line of $- \nabla_gf$ through $y'$) such that $f( x'') = f(y'') := c''$ and $x''$, $y'' \in U_0 \cap f^{-1}(c'')$.  Since $U_0 \cap f^{-1}(c'')$ is path connected, $x''$ and $y''$ may be connected by a continuous path in $U_0 \cap f^{-1}(c'')$.  

	Recall that there are a finite number of critical points below any level.  In particular, there are a finite number of critical points between $c$ and $c''$.  Call these points $p_1, \ldots, p_k$. Moreover, recall that for each $i$ ( $1 \leq i \leq k$) $W^u(p_i) \subset M$ is a submanifold with codimension equal to $index_{p_i}(f)$ by Lemma \ref{W^s submfld}.  In particular, by Lemma \ref{transverse} each $W^u(p_i)$ intersects the smooth part of $f^{-1}(c'')$ transversally in submanifolds of codimension at least $2$ because the coindex of $f$ cannot equal one for any critical point.  Consequently, we may choose a path $\gamma^\star \colon [0,1] \rightarrow U_0 \cap f^{-1}(c'')$ with $\gamma^\star (0)= x''$ and $\gamma^\star (1) = y''$ such that it is transverse to each of the unstable manifolds $W^u(p_i)$, $1 \leq i \leq k$.
	
	\begin{figure}[htpb]
	\begin{center}
	\resizebox{13.5cm}{12cm}{\includegraphics{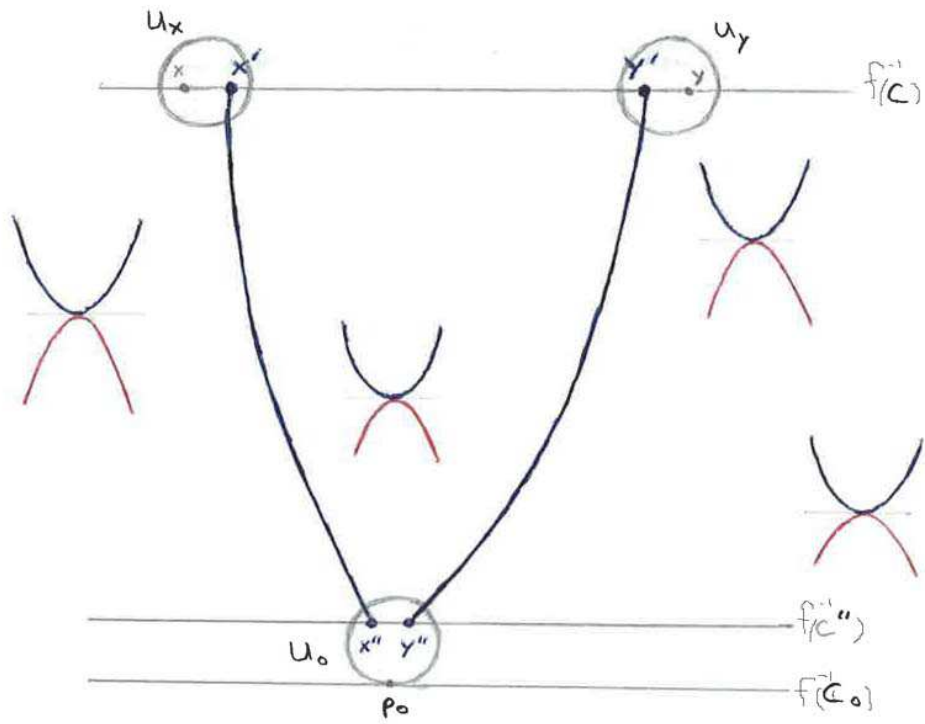}}
	\end{center}
	\caption{ $f^{-1}(c)$ connected for any $c \in \mathbb{R}$}
	\label{fig:connectedlevels}
	\end{figure}
	
	We can now use the gradient flow to move this path $\gamma^\star$ back up to the level of $f^{-1}(c)$;  Recall that $\psi_t$ denotes the downward normalized flow. Let $(t,m) \mapsto \psi_t(m)$ be defined on an open subset $U \subseteq  \mathbb{R} \times M$.  Define an open subset $U_t := \{ m \in M \hspace{2mm} | \hspace{2mm} (t, m) \in U \}  \subseteq M$.  So by Lang (\cite{SL01} Chapter IV Theorem 2.9), for all $t$ the map $\psi_t \colon U_t \rightarrow U_{-t}$ is a diffeomorphism with inverse $\psi_{-t}$.  Note that $\psi_t$ restricts to a diffeomorphism $U_t \cap f^{-1}(c) \rightarrow U_{-t} \cap f^{-1}(c'')$ with inverse the restriction of $\psi_{-t}$ when $t = c - c''$.	 It then follows that for $t = c - c''$, 
	the set $\psi_{-t} \left( \gamma^\star ( \cdot ) \right) $ is an open and dense subset of $f^{-1}(c)$; this can be arranged because from Palais it follows that $$f^{-1}(c'') \smallsetminus \left( U_{-t} \cap f^{-1}(c'') \right) = f^{-1}(c'') \cap \bigcup_{{\tiny \begin{array}{c} p \in Crit(f) \hspace{1mm} such \\ that \hspace{.5mm} c' \leq f(p) \leq c \end{array}}} W^u(p).$$ \noindent where $c''$ is a regular value of $f$.  Now $x'$ and $y'$ can be joined by a path in $f^{-1}(c)$, as desired. 
	
	This proves that $x'$ and $y'$ can be joined by a path in $f^{-1}(c)$.  Therefore, we may conclude that $f^{-1}(c)$ is connected for every 
	$c \in \mathbb{R}$ in this case.
	
\noindent \textbf{Case II}: \textit{for any critical point of $f$ in $f^{-1}(c)$, $f^{-1}(c)$ is connected } \\ 
  \indent We first prove that $f^{-1}(c_0)$ is connected.  Recall that $c_0$ is the global minimum value of $f$ (See remark 4.0.35, $f(p_0)=c_0$).  We know that $index_{p_0}(f) = 0$.  By Lemma \ref{W^s submfld} the stable manifold of $p_0$, $W^s(p_0) \subset M$, is connected.  
  Hence, $f^{-1}(c_0) \subseteq W^s(p_0)$ must also be connected.  To see this suppose that $f^{-1}(c_0)$ is not connected, i.e. $f^{-1}(c_0) = U \sqcup V$ such that $U$, $V \neq \emptyset$ and $U \neq V$.  Then $W^s(p_0) = W^s(U) \sqcup W^s(V)$ with both $W^s(U)$, $W^s(V)$ nonempty and not equal to each other.  But this means that $W^s(p_0)$ is not connected, a contradiction.  Therefore, $f^{-1}(c_0)$ is connected.
  
  Next, note that for every singluar point in $M$ there exists a regular point of $f$ in the same level set such that we can connect them through a path that lies entirely within the level set.  Then we may connect any two regular points in this level of $f$ as in Case $I$, so as to obtain that $f^{-1}(c)$ is connected.  Thus is it sufficient to show that a critical point can be connected to a regular point within the level.  Let us provide more details.
  
Suppose that $x \in f^{-1}(c)$ is a singular point of $f$.  By the Morse Lemma \ref{Morse lemma}, there exists a neighbourhood $U_x$ of $x$ and a chart $\phi$ such that $\phi (x) = 0$, $f^{\mathbb{H}} (x_+, x_-) = ||x_+||^2 - ||x_-||^2$ on $\phi(U_x)$ and that $\mathbb{H} = \mathbb{H}_+ \oplus \mathbb{H}_-$.  Fix such a Morse chart $\phi \colon U_x \rightarrow B_0 \subset \mathbb{H}$ (where $B_0$ is a neighbourhood of $0$) of $x$ with the properties:    

	\begin{list}{}{}
	\item $\bullet \hspace{2mm} \phi$ is an isometry,
	\item $\bullet \hspace{2mm} \phi(U_x) = B_+ \times B_-$, where $B_{\pm} \subset \mathbb{H}_{\pm}$ are unit balls in $\mathbb{H}_{\pm}$ respectively. 
	\end{list}

So $ \left( B_+ \times B_- \right) \cap (f^{\mathbb{H}})^{-1}(0)   =  \{ (x_+, x_- ) \in \mathbb{H} \hspace{2mm} | \hspace{2mm}  ||x_+||^2 = ||x_-||^2 \}. $ Observe that this is homeomorphic to a cone on $S_+(1) \times S_-(1)$, where $S_{\pm}(1)$ are unit spheres in $\mathbb{H}_{\pm}$ respectively.  Note that the set $\left( B_+ \times B_- \right) \cap (f^{\mathbb{H}})^{-1}(0)$ collapses at the origin to give a cone over  $S_+(1) \times S_-(1)$).

		\begin{figure}[htpb]
	\begin{center}
	\includegraphics{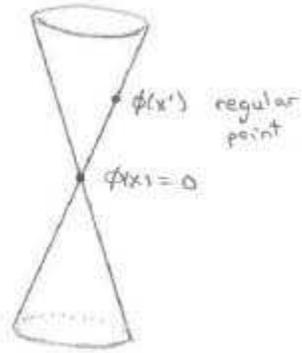}
	\end{center}
	\caption{Cone over $S_+(1) \times S_-(1)$.}
	\label{fig:cone over}
	\end{figure}
					
Recall that the critical points of $f$ are isolated.  If we start at the origin $(0,0)$ then $\{ (tx_+, tx_-) \hspace{1mm} | \hspace{1mm}  0 \leq t \leq 1 \}$ is the path connecting $(0,0)$ to a regular point, say $\phi(x')$, in $\left( B_+ \times B_- \right) \cap (f^{\mathbb{H}})^{-1}(0) $.  This implies that we can connect $x$ to a regular point of $f$, say $x'$, in $f^{-1}(f(x))$ in $M_0 \cap f^{-1}(f(x))$, as desired.

	This proves that $f^{-1}(c)$ is connected in this case.

	Taken as a whole, we see that the proof is complete.

	\end{proof}

\chapter{Convexity and Connectedness}
	 In this chapter we will state and prove the main results of this thesis.  
	
\section{Almost Periodic $\mathbb{R}^n$ Actions and Complex Structures}

\begin{note}
	\label{almost periodic R}
	An $\mathbb{R}$-action on a manifold $M$ is said to be \textbf{almost periodic} if there exists a torus action $(S^1)^N \circlearrowright M$ and a one-parameter subgroup $\mathbb{R} \rightarrow (S^1)^N$ such that the $\mathbb{R}$-action is the composition $\left( \mathbb{R}, + \right) \rightarrow (S^1)^N \circlearrowright M$.
\end{note}

\begin{note}
	\label{almost periodic}
	An $\mathbb{R}^n$-action on a manifold $M$ is said to be \textbf{almost periodic} if there exists a torus action $(S^1)^N \circlearrowright M$ and a homomorphism $(\mathbb{R}^n, +) \rightarrow (S^1)^N$ such that the $\mathbb{R}^n$-action is the composition $\mathbb{R}^n \rightarrow (S^1)^N \circlearrowright M$.
\end{note}

\begin{rem} 
	\label{T setup}
Let $T$ be the closure of the image of the homomorphism $(\mathbb{R}^n, +) \rightarrow (S^1)^N$.  
\end{rem}

\begin{note}
	\label{T generated}
	In the notation of Remark \ref{T setup}, we define the \textbf{generated torus action} on $M$ to be $T$ with its action on $M$.
\end{note}

	From now let $M$ be a connected strongly symplectic Hilbert manifold.  Suppose that we have an almost periodic $\mathbb{R}^n$ action on $M$ with momentum map $\mu \colon M \rightarrow \mathbb{R}^n$.  Suppose that the $\mathbb{R}^n$ action has isolated fixed points.  Fix a $\xi \in \mathbb{R}^n$ such that the momentum map component $\mu^{\xi} := \langle \mu ( \cdot ), \xi \rangle \colon M \rightarrow \mathbb{R}$ has only nondegenerate critical points (i.e., $\mu^{\xi}$ is a Morse function).  Recall that
\begin{enumerate}
\item[(i)] the critical point set of every Morse component is fixed by $T$, the generated torus action on $M$; and conversely
\item[(ii)] if the set of critical points of a component of $\mu$ is fixed by $T$ then that component is Morse.
\end{enumerate}

	Let $x \in M^T$, the fixed point set of the generated torus action.  By continuity, $M^T = M^{\mathbb{R}^n}$, the fixed point set of the almost periodic $\mathbb{R}^n$ action on $M$.  Note that $M^{\mathbb{R}^n}$ only depends on $\mu$.  In what follows, we will show that there exists a $T$-invariant compatible complex structure on the symplectic vector space $(T_xM, \omega )$.  We will also establish that no critical points of $\mu^{\xi} \colon M \rightarrow \mathbb{R}$ have index or coindex equal to one.

\begin{lem}
\label{complex structure}

	Let $( M, \omega )$ be a connected strongly symplectic Hilbert manifold.  Suppose that we have an almost periodic $\mathbb{R}^n$ action on $M$ with momentum map $\mu : M \rightarrow \mathbb{R}^n$.  Suppose that this $\mathbb{R}^n$ action has isolated fixed points.  Let $T$ be the torus generated by the almost periodic $\mathbb{R}^n$-action and let $p \in M^T$. There exists an $\omega$-compatible and $T$-invariant complex structure $J$ on $T_pM$.
\end{lem}

We establish Lemma \ref{complex structure} in a manner similar to Weinstein \cite[Lecture $2$, pg $8$]{AL79}.

\begin{proof}

	By averaging over the torus $T$, we may choose a positive $T$-invariant inner product $\langle \cdot , \cdot \rangle$ on $T_pM$.  Observe that $T_pM$ is a strongly symplectic (real) vector space since it is equipped with a strongly symplectic nondegenerate $2$-form $\omega$.  Since $\omega$ and $ \langle \cdot ,  \cdot \rangle$ are nondegenerate, 
	
	 \[ 
\begin{array}{ll}
	u \in T_pM & \mapsto \omega ( u, \cdot ) \in T_p^*M \\
	v \in T_pM & \mapsto  \langle v , \cdot \rangle \in T_p^*M 
\end{array}
			\Bigg\rbrace  \] 
			
\vspace{3mm}			
			
			\noindent are isomorphisms between $T_pM$ and $T_p^*M$.  Hence, $\omega$ can be represented by some linear (skew-adjoint) operator $A \colon T_pM \rightarrow T_pM$, i.e., $\omega( u , v ) = \langle Au , v \rangle$ for $u, v \in T_pM$.  Note that $A$ is skew-adjoint (with respect to $\langle \cdot , \cdot \rangle$ ) because
	\begin{eqnarray*}
	\langle A^Tu,v \rangle & = & \langle u, Av \rangle \textrm{ ,by definition of $A^T$}\\
		& = & \langle Av , u \rangle \textrm{ , since $\langle \cdot , \cdot \rangle$ is symmetric}\\
		& = & \omega(v,u) \textrm{ , by definition of $A$} \\
		& = & - \omega(u,v) \textrm{ , since $\omega$ is skew-symmetric} \\
		& = & - \langle Au , v \rangle \textrm{ , by definition of $A$.}
	\end{eqnarray*}
	
\noindent We wish to find a $T$-invariant and $\omega$-compatible complex structure $J$ on $T_pM$.  We claim that: $J = \sqrt{(AA^T)}^{- 1}A$ has these properties.

Note that $(AA^T)^{-1}$ is an operator on $T_pM$ that is positive definite and symmetric with respect to $\langle \cdot , \cdot \rangle$.  By the Spectral Theorem we can obtain an operator $\sqrt{(AA^T)^{-1}}$ such that $\left( \sqrt{(AA^T)^{-1}} \right)^2 = (AA^T)^{-1}$.  Moreover, $\sqrt{(AA^T)}^{-1}$ commutes with every operator that commutes with $(AA^T)^{-1}$:  See \cite[Chap. $4$, Prop. $4.33$ page $86$]{RD98}.  In particular, since $A$ commutes with $\left( AA^T \right)^{-1} = - (A^2)^{-1}, \hspace{2mm} \sqrt{AA^T}^{-1}$ comutes with $A$.  Moreover $\sqrt{(AA^T)^{-1}}$ is symmetric and positive definite.  Let $$J := (AA^T)^{- \frac{1}{2}}A.$$

$J$ is orthogonal (with respect to $\langle \cdot , \cdot \rangle$ ) because
	\begin{eqnarray*}
	\langle Ju,Jv \rangle & = & \langle (AA^T)^{- \frac{1}{2}}Au, (AA^T)^{- \frac{1}{2}}Av \rangle \textrm{ , by definition of $J$}\\
		& = & \langle Au , (AA^T)^{-1}Av \rangle \textrm{ , since $(AA^T)^{- \frac{1}{2}}$ is symmetric}\\
		& = & \langle Au , (A^T)^{-1}A^{-1}Av \rangle  \\
		& = & \langle Au , (A^T)^{-1}v \rangle \\
		& = & \langle u , A^T (A^T)^{-1}v \rangle \\
		& = & \langle u , v \rangle
	\end{eqnarray*}

From $A$ skew-adjoint ($A^T = -A$), we can deduce that $J^T=-J$: 
	\begin{eqnarray*}
	A^T = -A & \Rightarrow & (AA^T)^{-\frac{1}{2}} A^T = - (AA^T)^{-\frac{1}{2}} A = - J \\
		& \Leftrightarrow & \left( A (AA^T)^{-\frac{1}{2}} \right)^T = -J \textrm{ , since $(AA^T)^T=AA^T$} \\
		& \Leftrightarrow & \left( (AA^T)^{-\frac{1}{2}}A \right)^T = -J \textrm{ , as $A$ and $(AA^T)^{-\frac{1}{2}}$ commute} \\
		& \Leftrightarrow & J^T = - J
	\end{eqnarray*}

Hence,  
	\begin{eqnarray*}
	J^2 & = & J(-J^T) , \textrm{ because $J^T=-J$}\\
		& = & -(AA^T)^{-\frac{1}{2}}A \left( (AA^T)^{-\frac{1}{2}}A \right)^T \\
		& = & -(AA^T)^{-\frac{1}{2}}A A^T (AA^T)^{-\frac{1}{2}} \textrm{ , since $(AA^T)^T=AA^T$} \\
		& = & -AA^T(AA^T)^{-\frac{1}{2}} (AA^T)^{-\frac{1}{2}} \textrm{ , as $AA^T$ and $(AA^T)^{-\frac{1}{2}}$ commute} \\
				& = & -AA^T (AA^T)^{-1} \\
		& = & - \textrm{Id}
	\end{eqnarray*}
	\noindent That is, $J$ is a complex structure on $T_pM$.  Moreover, $J$ is $T$-invariant (because $\langle \cdot , \cdot \rangle$ is and $\omega$ is) and $\omega$-compatible because
 		\begin{eqnarray*}
	\omega(Ju, Jv)  & = &\langle AJu,Jv \rangle \textrm{ , by definition of $A$}\\
		& = & \langle JAu,Jv \rangle \textrm{ , since $J$ and $A$ commute}\\
		& = & \langle Au,v \rangle \textrm{ , since $J$ is orthogonal} \\
		& = & \omega (u,v) \textrm{ , by definition of $A$}
	\end{eqnarray*}

		\begin{eqnarray*}
	\omega(u, Ju)  & = &\langle Au,Ju \rangle \textrm{ , by definition of $A$}\\
		& = & \langle JAu,J^2u \rangle \textrm{ , since $J$ is orthogonal}\\		
		& = & \langle JAu,-u \rangle \textrm{ , since $J^2 =-$Id }\\
		& = & - \langle - \sqrt{AA^T}u,u \rangle \textrm{ , using definition of  $J$ in terms of $A$} \\		
		& = &  \langle \sqrt{AA^T}u,u \rangle \\
		& > & 0 \textrm{ , for $u \neq 0$}
	\end{eqnarray*}
	
	Therefore, $J$ is a $T$-invariant and $\omega$-compatible complex structure on $T_pM$ as wanted.
\end{proof}

\begin{rem} 
\begin{enumerate}
\item The factorization $\sqrt{(AA^T)}J = A$ (equivalently, $J=(AA^T)^{- \frac{1}{2}}A$ as written in the proof) is known as the \textit{polar decomposition} of $A$.

\item In general (as indicated in the proof), the positive inner product defined by $\omega (u , Jv) = \langle \sqrt{AA^T}u,v \rangle$ is different from $\langle u , v \rangle$.

\item This construction of $J$ is canonical after an initial choice of Riemannian metric $M$.
\end{enumerate}
\end{rem}

	We are now ready to examine a Morse component of the momentum map $\mu \colon M \rightarrow \mathbb{R}^n$.  The next lemma show us that no critical points of this component $\mu^{\xi}$ have index or coindex equal to one.   
	
\begin{thm}
	\label{even index}
	
Let $M$ be a strongly symplectic Hilbert manifold.  Suppose that we have an almost periodic $\mathbb{R}^n$ action on $M$ with momentum map $\mu : M \rightarrow \mathbb{R}^n$.  Fix a $\xi \in \mathbb{R}^n$ such that $\mu^{\xi} := \langle \mu(\cdot), \xi \rangle$ is a Morse function.  Then none of the critical points of $\mu^{\xi}$ have index or coindex equal to $1$.
\end{thm}

\begin{rem}  This Lemma is the infinite-dimensional analogue of a lemma in Atiyah, \cite[Lemma $(2.2)$]{MA82} and Guillemin-Sternberg \cite[Theorem $5.3$]{GS82}.
\end{rem}

\begin{proof}[Proof of Lemma \ref{even index}]
	Let $T$ be the torus generated by the almost periodic $\mathbb{R}^n$-action on $M$.  The critical points of $\mu^{\xi}$ are the fixed points of $T$, i.e., $Crit(\mu^{\xi}) = M^T$.  Let $p \in M^T$ and let $\mathbb{H}$ be a strongly symplectic (real) Hilbert space on which $M$ is modelled.  By an appropriate choice of charts we may identify $T_pM$ with $\mathbb{H}$.  Note that different charts induce on $T_pM$ different inner products (but with the same topology).  
	
	We will show that we may choose symplectic coordinates which linearize the action.  In such coordinates, $\mu^{\xi}$ looks like a quadratic.  Note in particular that the eigenspaces of the Hessian of $\mu^{\xi}$ at $p$ are even-dimensional.  The details are as follows:
	
	\noindent \underline{Step 1}:  \textit{existence of a $T$-invariant metric on $T_pM$} \
	
	 Fix some Riemannian metric on $M$.  Choose a $T$-invariant inner product, say $\langle \cdot , \cdot \rangle $, on $T_pM$.  Observe that $T_pM$ is a strongly symplectic (real) vector space since it is equipped with a strongly nondegenerate $2$-form $\omega$.    Then $\omega$ can be identified with some skew-adjoint operator $A \colon T_pM \rightarrow T_pM$  such that $\omega (u,v) = \langle Au,v \rangle$.  
	
	\noindent \underline{Step 2}:  \textit{obtain a $T$-invariant, $\omega$-compatible complex structure on $T_pM$} \
	
	By Theorem \ref{complex structure}, there exists a $T$-invariant and $\omega$-compatible complex structure $J$ on $T_pM$. Namely,  $J = \sqrt{(AA^T)}^{-1}A$. 

\newpage 
	\noindent \underline{Step 3}:  \textit{obtain a $J$-invariant orthogonal decomposition of $T_pM$} \

	Given a complex structure $J$, $T_pM$ becomes a complex vector space where the Hermitian inner product is $T$-invariant.  We may now decompose $T_pM$ into irreducible complex representations according to the weights associated with the linear isotropy representation of $T$ on $T_pM$.

	We obtain a $J$-invariant orthogonal decomposition  $$T_pM = \left( \bigoplus_{{\tiny \begin{array}{c} \alpha \in \mathfrak{t}^*_{\mathbb{Z}} \textrm{ such} \\ \textrm{ that } \left\langle \alpha , \ \xi \right\rangle > 0 \end{array}}} V_{\alpha} \right) \bigoplus \left( \bigoplus_{{\tiny \begin{array}{c} \alpha \in \mathfrak{t}^*_{\mathbb{Z}} \textrm{ such} \\ \textrm{ that } \left\langle \alpha , \ \xi \right\rangle < 0 \end{array}}} V_{\alpha} \right) $$ where each $V_{\alpha}$, for $\alpha >0$, corresponds toa non-trivial character of $T$ while the vector spact $V_0 = T_pM^T$ and is fixed by $T$.  Note that the summands in the above decomposition of $T_pM$ are orthogonal with respect to $\omega$ as well as with respect to the inner product.
	
 	\noindent \underline{Step 4}:  \textit{$\mu^{\xi}$ is a quadratic} 
 	\vspace{2mm}
 	
 	For each $z \in V_{\alpha}$ we claim that $\mu^{\xi}(z) = - \frac{1}{2}|| z ||^2 \alpha $ (meaning that the Hessian $H(z,z) = \left\langle z , z \right\rangle$ by compatibility).  To see this note that the $S^1$ action on $V_{ \alpha}$ is generated by
 	\begin{eqnarray*}
 	\left. \frac{d}{dt}\right|_{t=0} \left( e^{it} \cdot z \right) & = & \left. i e^{it} z \right|_{t=0}  \\
 																	& = & iz \\
 																	& = & Jz.\\
 	\end{eqnarray*}

		Let $X$ be the vector field (associated to the linearized flow) on $T_pM$ that satisfies $X|_z = Jz$. 
		Now, by the Local Linearization Theorem \ref{locallinearizationsymplectic}, there is a $G$-equivariant symplectomorphism (say $\phi$) from an invariant neighbourhood of the origin in $T_pM$ onto an invariant neighbourhood of $p \in M$.  
		
		It follows that 
\begin{eqnarray*}	\left. d \mu^{\xi}\right|_{z}(v) & = & - \left\langle z , v \right\rangle \\
																	& = & - \omega ( v, Jz ) \\
																	& = & \omega ( Jz, v) \\																	
																	& =& \left. \omega_p \right|_{z} ( X, v ) .\\
										\end{eqnarray*}

	In other words, the momentum map is given by $\left\langle \alpha , \xi \right\rangle$ $$\mu^{\xi} (z) =  \sum_{\alpha \in \mathfrak{t}^*_{\mathbb{Z} }} - \frac{1}{2} ||z||^2 \alpha $$ in a coordinate system on $T_pM$.  Hence, all of the eigenspaces of the Hessian of $\mu^{\xi}$ at any $p \in Crit(\mu^{\xi})$ are even-dimensional.  This proves that the critical points of $\mu^{\xi}$ have even index and coindex.  In particular, index$_p(\mu^{\xi})$ and coindex$_p(\mu^{\xi})$ are not equal to one, as wanted.
	
\end{proof}
	
\begin{cor}
	\label{base case}
	
	Let $M$ be a connected strongly symplectic Hilbert manifold.  Suppose that we have an almost periodic $\mathbb{R}$-action on $M$ with momentum map $\mu \colon M \rightarrow \mathbb{R}$.  Suppose that the $\mathbb{R}$ action has isolated fixed points.  Suppose that there exists a complete invariant Riemannian metric on $M$ such that 
	either the map  $\mu \textrm{ or } -\mu \colon \textrm{M} \rightarrow \mathbb{R}$ is bounded from below and satisfies Condition (C).  Then for every $c \in \mathbb{R}$, the level set $\mu^{-1}(c)$ is connected (or empty).  
\end{cor}

\begin{rem}  Note that Corollary \ref{base case} is a stronger version of the main Convexity Theorem, Theorem \ref{convexity}, where $n=1$ and $H=\{ 0 \}$.
\end{rem}

\begin{proof}[Proof of Lemma \ref{base case}]

We have an almost periodic $\mathbb{R}$-action and thus $\mathfrak{t} = \mathbb{R}$ and $\mathfrak{t}^* = \mathbb{R}$.  Hence, the momentum mapping $\mu \colon M \rightarrow \mathbb{R}$ is a smooth $\mathbb{R}$-valued function.  
Without loss of generality, suppose that $\mu$ is bounded from below (otherwise apply the below argument to $- \mu$).  Since the critical points of $\mu$ are nondegenerate (by assumption) note that $\mu$ is a Morse function.    By Theorem \ref{even index}, none of the critical points of $\mu$ have index or coindex equal to $1$.  Therefore, by Theorem \ref{connectedlevels} the level set $\mu^{-1}(c)$ is connected for every $c \in \mathbb{R}$.  

\end{proof}

\section{Rational Independence and Consequences}

\begin{note}
	\label{rationally independent set}
	
	A collection of real numbers $\theta_1, \ldots , \theta_n$ is said to be \textbf{rationally independent} over $\mathbb{Q}$ if the only $n$-tuple of integers $s_1, \dots , s_n$ such that $s_1 \theta_1 + \cdots + s_n \theta_n = 0$ is the trivial solution in which every $s_i = 0$.

\end{note}
\medskip
\begin{eg}  $\underbrace{\overbrace{3, \hspace{2mm} \sqrt{8}}^{\textrm{rationally independent}}, \hspace{2mm} 1 + \sqrt{2}}_{\textrm{rationally dependent}}$
\end{eg}
\medskip

\begin{note}
	\label{rationally independent torus}
	
	Let $T$ be an $N$-dimensional torus.  Choose a splitting of $T$, then $\mathfrak{t} = \mathbb{R}^N$ and $\ker (\exp) = \mathbb{Z}^N$.  Let $\theta \in \mathbb{R}^N$.  We say that $\theta := ( \theta_1, \ldots, \theta_N )$ has \textbf{rationally independent components}  if the numbers $\theta_1, \ldots, \theta_N$ are rationally independent over $\mathbb{Q}$.
\end{note}

\begin{rem}  Definition \ref{rationally independent torus} is independent of the choice of splitting.  Observe that if definition \ref{rationally independent torus} is satisfied with respect to one splitting of $T$ then it is satisfied with respect to every splitting of $T$ since they differ by a linear invertible map over $\mathbb{Q}$.
\end{rem}

\begin{note}
	\label{rationally independent general}
		
	Suppose that we have an almost periodic $\mathbb{R}^n$ action on $M$.  Let $T$ be the $N$-dimensional generated torus action on $M$ (where $n \leq N$).  We say that $\theta \in \mathbb{R}^n$ has \textbf{rationaly independent components} with respect to the almost periodic $\mathbb{R}^n$ action if the image of $\theta$ in $\mathfrak{t} \cong \mathbb{R}^N$ has rationally independent components. 
$$\xymatrix{
&\mathfrak{t} = \mathbb{R}^N  \ar[d]^{\exp} \\
 \mathbb{R}^n \ar[ur]|{\textrm{{\tiny linear map}}} \ar[r] & T = \mathbb{R}^N / \mathbb{Z}^N }
$$

\end{note}
 \medskip

	The following Lemma \ref{rationally independent} shows us that if the components of $\theta \in \mathbb{R}^n$ are rationally independent then the $\theta$ component of $\mu$, $\mu^{\theta}$, satisfies the two equivalent conditions $(i)$ and $(ii)$ in section $\S 5.1$, i.e., that $\mu^{\theta}$ is Morse and its critical point set is fixed by $T$.  This result will play an important role in establishing our convexity result, Theorem \ref{convexity}, for a generic set of regular values 
of the momentum map  (Cf. Lemma \ref{good projection}).  Moreover, this lemma will illustrate another consequence of the complex structure from the prior section, $\S 5.1$, when we prove that the critical point set of these components of the momentum map are themselves symplectic submanifolds of $M$.  In our case these are just points.
 
\begin{lem}
	\label{rationally independent}
	
	Let $M$ be a connected strongly symplectic Hilbert manifold.  Suppose that we have an almost periodic $\mathbb{R}^n$ action on $M$ with momentum map $\mu : M \rightarrow \mathbb{R}^n$.  Let $T$ be the torus generated by the almost periodic $\mathbb{R}^n$ action.  For every $\theta \in \mathbb{R}^n$, let $\mu^{\theta} \colon M \rightarrow \mathbb{R}$ where $\mu^{\theta}( \cdot) := \langle \mu (\cdot) , \theta \rangle$ be the corresponding component of the momentum map.  If $\theta$ has rationally independent components, then the critical set  of $\mu^{\theta}$ is equal to the fixed point set $M^T$. 
and $Crit ( \mu^{\theta} )$ is a symplectic submanifold of $M$.

\end{lem}

\begin{rem}  This Lemma \ref{rationally independent} is the almost periodic $\mathbb{R}^n$ action analogue of the well known torus action result \cite{MS98} pg 186:  Let $(M, \omega )$ be a compact connected symplectic manifold and $\mathbb{T}^n$ be a torus action on $M$ with momentum map $\mu \colon M \rightarrow \mathbb{R}^n$.  Then for every $\theta \in \mathbb{R}^n$ with rationally independent components, the critical set of the function $H_{\theta} := \langle \mu, \theta \rangle \colon M \rightarrow \mathbb{R}$ is fixed under the $\mathbb{T}^n$ action.  Moreover, the critical set of $H_{\theta}$ is a symplectic submanifold of $M$.
\end{rem}

\begin{proof}[Proof of Lemma \ref{rationally independent}]
Let $X$, $Y \in \mathfrak{t} = \mathbb{R}^N$. Note that 

\begin{eqnarray*} \mu^{kX}( \cdot ) & = & \langle \mu (\cdot), kX \rangle \\
													& = & k \langle \mu ( \cdot ), X \rangle \\
													& = & k \mu^X ( \cdot )
\end{eqnarray*}

\noindent for all $k \in \mathbb{Z}$, so $Crit(\mu^{kX}) = Crit( \mu^X)$. 
	
	Let $\theta \in \mathbb{R}^n$ such that $\theta$ has rationally independent components.  Recall that if $\theta = (\theta_1, \theta_2, \ldots, \theta_n)$ has rationally independent components then we can choose a lattice $\Lambda \subset \mathbb{Z}^N$ of $\mathfrak{t}$ such that the closure of the one parameter subgroup $\{ \exp(s \textrm{Im}( \theta )) \hspace{2mm} | \hspace{2mm} s \in \mathbb{R} \}$ is $T \cong U(1)^N$.  Said another way, the set of vectors $\{ s \textrm{Im} (\theta )+k \hspace{2mm} | \hspace{2mm} \textrm{ for all } s \in \mathbb{R} \textrm{ and } k \in \mathbb{Z}^N \}$ form a dense set in $\mathbb{R}^N$.  
	Then, since $\overline{ \{ s \theta + k \hspace{2mm} | \hspace{2mm} s \in \mathbb{R}, \hspace{2mm} k \in \mathbb{Z}^N \} } = \mathbb{R}^N$, we may conclude that $$Crit(\mu^{\theta}) = \bigcap_{t \in T} \hspace{1mm} Crit(\mu^t).$$  
	\noindent But for $\mathbb{R}$-valued  momentum maps a critical point of the momentum map is the same as a fixed point of the action.  Therefore,
\begin{eqnarray*} Crit( \mu^\theta ) & = & \bigcap_{t \in T} \hspace{1mm} Crit( \mu^t ) \\
												& = & \bigcap_{t \in T} \hspace{1mm} Fix( \mu^t ) \textrm{ , where $Fix(\mu^t)$ are the fixed points}\\
												& = & M^T \textrm{ ,  where $M^T$ denotes the $T$-fixed points in $M$} 
\end{eqnarray*}

\noindent as desired.

We can use this to prove that $Crit(\mu^{\theta})$ is a symplectic manifold:  
Since $M^T$ is a discrete set it is a symplectic submanifold.  It then follows that $Crit(\mu^{\theta}) = M^T$ is a symplectic submanifold of $M$.
\end{proof}

\section{Good Projections}

In this section we use the notation Fix($\star$) to denote the fixed point set of the $\mathbb{R}^n$ action whose momentum map is the function $\star$.

\begin{lem}
	\label{good projection}

		Let $M$ be a connected strongly symplectic Hilbert manifold.  Suppose that we have an almost periodic $\mathbb{R}^{n+1}$ action on $M$ with momentum map $\mu : M \rightarrow \mathbb{R}^{n+1}$.  Suppose that the $\mathbb{R}^{n+1}$ action has isolated fixed points.  Suppose that there exists a complete invariant Riemannian metric on $M$ such that there exists a hyperplane $H$ of $\mathbb{R}^{n+1}$ such that for all $\xi \in \mathbb{R}^{n+1} \smallsetminus H$ the component $\mu^{\xi} := \langle \mu, \xi \rangle \colon \textrm{M} \rightarrow \mathbb{R}$ is bounded from one side  and satisfies Condition (C).  
Then 
there exists a projection $\pi \colon \mathbb{R}^{n+1} \rightarrow \mathbb{R}^n$ satisfying 
		
\begin{list}{}{} 
	\item{(i)} the $\mathbb{R}^n$ action generated by $\mu' := \pi \circ \mu$ is almost periodic and has isolated fixed points; and
	
	\item{(ii)} there exists a hyperplane $H' \subset \mathbb{R}^n$ such that for all $\xi' \in \mathbb{R}^n \smallsetminus H'$ the component $(\mu')^{\xi'} \colon M \rightarrow \mathbb{R}$ is bounded from one side and satisfies condition (C).\
	\end{list}
	
\end{lem}

\begin{proof}
	We first prove Lemma \ref{good projection} in the special case of a \textit{torus action} on $M$, that is, in the case when we have a \textit{periodic} $\mathbb{R}^n$ action on $M$.  This is in preparation to set up for the almost periodic case.
	
	For property $(i)$:  Let $A_{RI} \subseteq \mathbb{R}^{n+1}$ be the set of elements whose members are rationally independent in $\mathbb{R}^{n+1}$ and denote its complement by $A_{RD} \subseteq \mathbb{R}^{n+1}$, i.e. 
 \begin{eqnarray*} A_{RD} & = &  \mathbb{R}^{n+1} \smallsetminus A_{RI} \\
 										& = & \{  v \in \mathbb{R}^{n+1}  \hspace{2mm} |\hspace{2mm} \exists \hspace{1mm} s_1, \ldots , s_{n+1} \in \mathbb{Q}  \textrm{, not all zero,  such that } \Sigma s_iv_i = 0 \} \\
 										& = & \{ v \in \mathbb{R}^{n+1} \hspace{2mm} | \hspace{2mm} \exists w \in \mathbb{Q}^{n+1} \smallsetminus \{ 0 \} \textrm{ such that } \langle v , w \rangle = 0 \}. \
\end{eqnarray*} 
 
 	\begin{list}{}{}
	\item $\bullet$ {\bf Proposition 1}:  Let $\pi \colon \mathbb{R}^{n+1} \rightarrow \mathbb{R}^n$.  Write $\ker(\pi) = \langle p \rangle$.  Suppose that there exists $\theta \in A_{RI}$ such that $\theta \perp p$ (i.e. $\langle \theta , p \rangle = 0$). Then $Fix( \mu') = M^T$. \\  \textit{Proof}: \quad We know that $M^T \subseteq Fix(\mu')$.  So we need to show that the opposite containment holds, i.e., show $Fix(\mu') \subseteq M^T$.  Choose $\theta$ as in the hypothesis.  Let $x \in Fix(\mu')$.  Then $0 = d \mu'_x = \pi \circ d \mu_x$.  So Im$( d \mu_x) \subseteq \ker(\pi)$.  But $\theta \perp \textrm{Im}(d \mu_x)$ by hypothesis.  That is, $\langle d \mu_x( \cdot), \theta \rangle = 0$, i.e., $d \mu^{\theta}_x = 0$.  Hence $x \in Fix(\mu^{\theta}) = Crit(\mu^{\theta})$.  Then we have that $Crit(\mu^{\theta}) = Fix( \mu^{\theta} ) \subseteq M^T.$  However, the inclusion is an equality because $Crit(\mu^{\theta}) = M^T$ since $\theta$ has rationally independent components by Lemma \ref{rationally independent}.  Thus $Fix(\mu') = M^T$ ending the proof of Proposition $1$. $\blacksquare$
	
	\end{list}

 Let 
\begin{eqnarray*} S & = & \{ p \in \mathbb{R}^{n+1} \hspace{2mm} | \hspace{2mm} p^{\perp} \subseteq A_{RD} \} \\
							& = & \{  p \in \mathbb{R}^{n+1} \hspace{2mm} | \hspace{2mm} \forall a \in p^\perp, \exists q \in \mathbb{Q}^{n+1} \smallsetminus \{ 0 \} 
\textrm{ with } \langle q,a \rangle = 0 \} \\
							& = & \{ p \in \mathbb{R}^{n+1} \hspace{2mm} | \hspace{2mm} p^\perp \subseteq \cup_{q \in \mathbb{Q}^{n+1} \smallsetminus \{ 0 \} } \hspace{1mm} q^\perp \} \\
							& \subseteq & \mathbb{R}^{n+1}.\
							\end{eqnarray*}
					
	By Proposition $1$, in order to show that there is a projection $\pi$ such that the periodic $\mathbb{R}^n$ action generated by $\mu'$ has isolated fixed points, it suffices to show that the set $S$ has measure zero.  The complement of $S$ is the union of kernels $\langle p \rangle$ of desired projections.  So if $S$ has measure zero then its complement must be nonempty.				
			
	\begin{list}{}{}	
	\item $\bullet$ {\bf Proposition 2}:  The set $S$ has measure zero in $\mathbb{R}^{n+1}$.  \
						
		\textit{Proof:} \quad To prove this we will require a preliminary result.  Let $H$ and $\{ H_i \}_{i=1}^{\infty}$ be hyperplanes in $\mathbb{R}^{n+1}$. 
		Suppose that $H \subseteq \displaystyle \cup_{i=1}^{\infty} H_i$.  Then there exists an $i \in \mathbb{N}$ such that $H \subset H_i$.  To prove this, first note that $$H = \bigcup_{i=1}^{\infty} ( H \cap H_i).$$  Suppose for contradiction that for all $i$ we have that $H \cap H_i \subsetneq H$.  If $H \cap H_i \ne H$ then $H \cap H_i$ has measure zero in $H$.  So if there is no $H_i$ with $H \cap H_i = H$, then $H$ is a countable union of sets of measure zero in $H$, which means that $H$ itself has measure zero in $H$.  This is a contradiction. Thus $ H \subset H_i$ for some $i$ and this ends the proof of the preliminary result.

	It follows that 
\begin{eqnarray*} S & = & \{ p \hspace{2mm} | \hspace{2mm} p^\perp \subseteq \cup_{q \in \mathbb{Q}^{n+1} \smallsetminus \{ 0 \}} \hspace{1mm} q^\perp \} \\
							& = & \bigcup_{q \in \mathbb{Q}^{n+1} \smallsetminus \{ 0 \} } \{ p \hspace{2mm} | \hspace{2mm} p^\perp \subseteq q^\perp \} \textrm{ , by the above claim}\\
							& = & \bigcup_{q \in \mathbb{Q}^{n+1} \smallsetminus \{ 0 \} } \{ p \hspace{2mm} | \hspace{2mm} p^\perp =  q^\perp \} \textrm{, since $p^{\perp}$ cannot be a proper subset of $q^{\perp}$}\\
\end{eqnarray*}

\noindent This is a countable union of lines. 
This completes the proof of Proposition$2$. $\blacksquare$
	\end{list}

	We now generalize the preceding arguments to establish the \textit{almost periodic} $\mathbb{R}^{n+1}$ action on $M$ case.
	
	Let $i \colon \mathbb{R}^{n+1} \rightarrow \mathbb{R}^N$ be a linear map such that the composition $$\mathbb{R}^{n+1} \rightarrow \mathbb{R}^N \rightarrow \mathbb{R}^N / \mathbb{Z}^N := T$$ has dense image in $T$.  
	
$$\xymatrix{
&&\mathfrak{t} = \mathbb{R}^N  \ar[d]^{\exp} \\
 \mathbb{R}^{n+1} \ar[r]^{\pi} & \mathbb{R}^n \ar[ur]^i \ar[r] & T = \mathbb{R}^N / \mathbb{Z}^N }
$$ 	
	\medskip

	Let $\tilde{A}_{RI} \subseteq \mathbb{R}^{n+1}$ be the set of rationally independent elements in $\mathbb{R}^{n+1}$.  That is, 
	$$\tilde{A}_{RI} = \{ \theta \in \mathbb{R}^{n+1} \hspace{2mm} | \hspace{2mm} (i \circ \pi) ( \theta ) \in \mathbb{R}^N \textrm{ has rationally independent components} \}$$

Let $\tilde{A}_{RD} \subseteq \mathbb{R}^{n+1}$ denote its complement.

\item $\bullet$ {\bf Proposition 3}:   Let $\pi \colon \mathbb{R}^{n+1} \rightarrow \mathbb{R}^n$.  Write $\ker(\pi) = \langle p \rangle$.  Suppose that there exists $\theta \in \tilde{A}_{RI}$ such that $\theta \perp p$ (i.e. $\langle \theta , p \rangle = 0$). Then $Fix( \mu') = M^T$. \\

\textit{Proof:} \quad We know that $M^T \subseteq Fix(\mu')$.  So we need to show that the opposite containment holds.  Choose $\theta$ as in the hypothesis.  Let $x \in Fix(\mu')$.  Then $0 = d \mu'_x = \pi \circ d \mu_x$.  So Im$( d \mu_x) \subseteq \ker(\pi)$.  But $\theta \perp \textrm{Im}(d \mu_x)$ by hypothesis.  That is, $\langle d \mu_x( \cdot), \theta \rangle = 0$, i.e., $d \mu^{\theta}_x = 0$.  Hence $x \in Fix(\mu^{\theta}) = Crit(\mu^{\theta})$ since $\mu^{\theta}$ is a real-valued function.  But $Crit(\mu^{\theta}) = M^T$ by Lemma \ref{rationally independent}, because $\theta$ is rationally independent. Thus, $Fix( \mu^{\theta} )  = Crit(\mu^{\theta}) \subseteq M^T$ ending the proof of Proposition $3$.  $\blacksquare$
	
	By Proposition $3$, in order to show that there is a projection $\pi$ such that the almost periodic $\mathbb{R}^n$ action generated by $\mu'$ has isolated fixed points, it suffices to show that the set $\tilde{A}_{RD}$ has measure zero in $\mathbb{R}^{n+1}$.  This is sufficient because if $\tilde{A}_{RD}$ has measure zero in $\mathbb{R}^{n+1}$ then its complement $\tilde{A}_{RI}$ must be nonempty. 
\item $\bullet$ {\bf Proposition 4}:  The complement of the set $$ \tilde{A }_{RI} = \{ \theta \in \mathbb{R}^{n+1} \hspace{2mm} | \hspace{2mm} (i \circ \pi)( \theta) \in \mathbb{R}^N \textrm{ has rationally independent components } \}$$ has measure zero in $\mathbb{R}^{n+1}$. 

	\textit{Proof:} \quad Let $\theta \in \mathbb{R}^{n+1}$.  Denote its image $(i \circ \pi)( \theta) $ by $(i \circ \pi) ( \theta ) := \tilde{ \theta } = ( \tilde{ \theta_1}, \ldots , \tilde{ \theta_N} )$.  Note that
 \begin{eqnarray*} \tilde{A}_{RD} & = &  \mathbb{R}^{n+1} \smallsetminus \tilde{A}_{RI} \\
 										& = & \{ \theta \in \mathbb{R}^{n+1} \hspace{2mm} | \hspace{2mm} \exists c \in \mathbb{Z}^N \smallsetminus \{ 0 \} \textrm{ such that } \langle c , \tilde{\theta} \rangle := \Sigma_{j=1}^N c_j \tilde{\theta}_j =  0 \}. \\
 										& = & \bigcup_{c \hspace{1mm} \in \hspace{1mm} \mathbb{Z}^{N} \smallsetminus \{ 0 \} } \{ \theta \in \mathbb{R}^{n+1} \hspace{2mm} | \hspace{2mm} \langle c, \tilde{ \theta } \rangle = 0 \} \
\end{eqnarray*} 

\noindent is a countable union of hyperplanes in $\mathbb{R}^{n+1}$.  Hence the complement of $\tilde{A}_{RI}$ has measure zero in $\mathbb{R}^{n+1}$.  $\blacksquare$

To summarize what we have done, Proposition $3$ shows us that to establish $(i)$ it is sufficient to show that the set $\tilde{A}_{RD} \subset \mathbb{R}^{n+1}$ has measure zero in $\mathbb{R}^{n+1}$.  Then by Proposition $4$ we know that $\tilde{A}_{RD}$ has measure zero.  This completes the proof of $(i)$.

For $(ii)$:  
Let $ \pi \colon \mathbb{R}^{n+1} \rightarrow \mathbb{R}^n$ be any projection such that $\pi^* (\mathbb{R}^n) \neq H$ (in $\mathbb{R}^{n+1}$), where $\pi^* := i \colon \mathbb{R}^{n} \rightarrow \mathbb{R}^{n+1}$.  Choose hyperplane $H' = \left( \pi^* \right)^{-1}H \textrm{ (the pre-image of $H$) } = \{ \xi' \in \mathbb{R}^n \hspace{2mm} | \hspace{2mm} \pi^*( \xi' ) \in H \} \subset \mathbb{R}^n$.  Observe that $H'$ has dimension $n-1$. Let $\xi' \in (H')^c$.  Let $\xi = \pi^* \xi'$.  Then the component \begin{eqnarray*} \mu^{\xi} & = & \langle \mu, \xi \rangle \\
										& = &  \langle \mu, \pi^* \xi' \rangle \\
										& = & \langle \pi \mu , \xi' \rangle \\
										& = & \langle \mu' , \xi' \rangle \\
										& = & (\mu')^{\xi'}. \
\end{eqnarray*}	

We claim that $\xi = \pi^* \xi' \in (H)^c$.
\medskip
$$ \xymatrix{
M \ar@(ul,dl)[]_{ \mathbb{R}^{n+1}} \ar[r]^(.3){\mu} \ar[dr]_(.35){\mu'} 
& (\mathbb{R}^{n+1})^* \ar[d]^{\pi = i^*} \ar[dr]^{\cdot \hspace{1mm} \xi \in \mathbb{R}^{n+1}} \\
& (\mathbb{R}^n)^* \ar[r]_(.6){\cdot \hspace{1mm} \xi' \in \mathbb{R}^n} & \mathbb{R} } $$
\medskip

This is clear from the above diagram together with the definition of $H'$.  %

Thus by hypothesis $\mu^{\xi}$ is bounded from below and satisfies condition (C).  But we saw that $\mu^{\xi} = (\mu')^{\xi'}$.  Therefore (ii) holds as wanted.

\textbf{End of Proof of Lemma \ref{good projection}.}
\end{proof}

\section{The Connectivity and Convexity Theorems}

The next Theorem, Theorem \ref{reg values residual}, may be of independent interest.  We prove that in the presence of an almost periodic $\mathbb{R}^n$ action on $M$, the set of singular values of the resulting momentum map is contained in a countable union of hyperplanes.  In particular, the set of regular values of the momentum map is residual in $\mathbb{R}^n$.  It is tempting to use the Sard-Smale Theorem \cite{SS65}, an infinite-dimensional versions of Sard's Theorem, but we cannot in the setting of this thesis.  The Sard-Smale Theorem requires that the map be Fredholm.
	
\begin{thm}
	\label{reg values residual}
	
	Let $M$ be a connected strongly symplectic Hilbert manifold.  Suppose that we have an almost periodic $\mathbb{R}^n$ action on $M$ with momentum map $\mu \colon M \rightarrow \mathbb{R}^n$.  Suppose that the $\mathbb{R}^{n}$ action has isolated fixed points.  Then the set of singular values of $\mu$ is contained in a countable union of hyperplanes.  In particular, the set of regular values of $\mu$ is residual in $\mathbb{R}^n$.
\end{thm}
	
\begin{rem}
	We will only use that the regular values of the momentum map are residual in $\mathbb{R}^n$ for the purpose of this thesis.
\end{rem}

\begin{proof}
	Let $T$ be the $N$-dimensional generated torus action on $M$ and let $H \subset T$ be a connected subgroup with $dim(H)>0$.  Note that $H$ must be a torus.  Let $x \in M$. 	
	
	Note that the critical points of $\mu$ are exactly those points whose stabilizer has positive dimension, and a connected component of the set of points with a fixed stabilizer of positive dimension gets mapped into a proper affine subspace of $\mathfrak{t}^* \cong \mathbb{R}^N$.  Because $M$ is second countable, it is sufficient to show that each point in $M$ has a neighbourhood in which at most countably many stabilizers occur.  Recall that 
\begin{list}{}{}
	\item[$\bullet \hspace{2mm}$] a linear representation of a compact abelian Lie group decomposes into a direct sum (in the Hilbert space sense) of subspaces, on each of which the group acts through a homomorphism to $S^1$; and
	\item[$\bullet \hspace{2mm}$] a strictly decreasing sequence of subgroups of a compact abelian group must be finite.
\end{list}

	First, the fixed point set of $H$, denoted $M^H$, coincides with that of the closure of $H$ (by continuity), so we can assume that $H$ is closed.  Consider a connected component $N$ of $M^H$, and $x \in N$.  By the Local Linearization Theorem \ref{locallinearization}, $M^H$ is a locally finite disjoint union of closed connected submanifolds.  It follows that $$Crit(\mu) = \bigcup_{\textrm{\begin{tiny}
subtori $H \subseteq T$ \end{tiny}} } M^H$$  \noindent is a countable union.

Let $j \colon \mathbb{R}^n \rightarrow T$.  Observe that $Stab_{\mathbb{R}^n}(x) = j^{-1}(H)$ where $H = Stab_T(x)$.  Then  by definition of the momentum map $$\textrm{CritValues}(\mu) =  \bigcup_{\begin{tiny} \begin{array}{c} \textrm{subtori} \hspace{0.5mm} H \subseteq T \textrm{such} \\ \textrm{that}  \hspace{0.5mm} j^{-1}(H) \subseteq \mathbb{R}^n \\ \textrm{and} \textrm{dim} \hspace{0.5mm}  ( j^{-1}(H) ) > 0 \end{array} \end{tiny}} \underbrace{\bigcup_{\begin{tiny} \begin{array}{c} \textrm{components} \\ N \hspace{0.5mm} \textrm{of} \hspace{0.5mm} M^H \end{array} \end{tiny}} \mu(N)}_{\textrm{countable union}}.$$ 
\medskip

\noindent Note that each $\mu(N)$ is contained in an affine subspace of $\mathbb{R}^n$ of positive codimension.  It follows that the complement of the set CritValues($\mu$) is a countable intersection of residual sets, and hence residual.  That is, the regular values of $\mu$ are residual.
\end{proof}

	We require one last ingredient for the proof of the Convexity Theorem, Theorem \ref{convexity}.  Namely, we require a lemma which makes explicit the relationship between statements $(A_n)$ and $(B_n)$ below.  We now state and prove this result.

\begin{lem}
	\label{a_n B_n}
	
	For every $n \in \mathbb{N}$, consider the following two statements
	
	\begin{list}{}{}
	\item{($A_n$)} Let $M$ be a connected strongly symplectic Hilbert manifold.  Suppose that we have an almost periodic $\mathbb{R}^n$ action on $M$ with momentum map $\mu \colon M \rightarrow \mathbb{R}^n$.  Suppose that the $\mathbb{R}^{n}$ action has isolated fixed points. Suppose that there exists a complete invariant Riemannian metric on $M$ such that there exists a hyperplane $H$ of $\mathbb{R}^n$ such that for all $\xi \in \mathbb{R}^n \smallsetminus H$ the map  $\mu^{\xi} \colon \textrm{M} \rightarrow \mathbb{R}$ is bounded from one side and satisfies Condition (C).  Then the set $$\{ c \in \mathbb{R}^n \hspace{2mm} | \hspace{2mm} c \textrm{ is a regular value of $\mu$ and } \mu^{-1}(c) \textrm{ is connected } \} \subseteq \mathbb{R}^n$$ is residual;
	\item{($B_n$)} Let $M$ be a connected strongly symplectic Hilbert manifold.  Suppose that we have an almost periodic $\mathbb{R}^n$ action on $M$ with momentum map $\mu \colon M \rightarrow \mathbb{R}^n$.  Suppose that the $\mathbb{R}^{n}$ action has isolated fixed points and suppose that $\mu(M)$ is closed.  Suppose that there exists a complete invariant Riemannian metric on $M$ such that there exists a hyperplane $H$ of $\mathbb{R}^n$ such that for all $\xi \in \mathbb{R}^n \smallsetminus H$ the map  $\mu^{\xi} \colon \textrm{M} \rightarrow \mathbb{R}$ is bounded from one side and satisfies Condition (C). Then the image $\mu(M) \subset \mathbb{R}^n$ is convex.
\end{list}

	Suppose that ($A_n$) is true for all $n$.  Then ($B_{n}$) is true for all $n$.
\end{lem}

\begin{proof}

	Note that $(B_1)$ trivially holds; For an almost periodic $\mathbb{R}$ action the momentum mapping $\mu \colon M \rightarrow \mathbb{R}$ is continuous.  Since $M$ is connected, it follows that $\mu(M) \subset \mathbb{R}$ is connected; $\mu(M)$ is an interval.  But connectedness is convexity in $\mathbb{R}$.  Therefore, $(B_1)$ is true.

	We want to show that $(B_{n+1})$ is true, i.e., we want to show that given any two distinct points in $\mu(M) \subset \mathbb{R}^{n+1}$ then the line segment joining them is also in $\mu(M)$. This proof follows the method of McDuff and Salamon \cite{MS98}.

		\noindent \textbf{Case 1}: \textit{The ``regular value'' case } \ 
	
	Choose an injective matrix  $A \in \mathbb{R}^{(n+1) \times n}$ such that (good projection) $\pi := A ^T \colon \mathbb{R}^{n+1} \rightarrow \mathbb{R}^n$ satisfies conditions (i) and (ii) of Lemma \ref{good projection} and such that $c' \in \mathbb{R}^n$ is a regular value of the restricted momentum map and is in the (residual) set  of values for which the restricted momentum map is connected.  Consider the restricted almost periodic $\mathbb{R}^n$ action on $M$.  This action is Hamiltonian with momentum map $\mu_A := A^T \circ \mu \colon M \rightarrow \mathbb{R}^n$.
	$$ \xymatrix{
                            & M \ar[r]^{\mu} 
                             \ar@{-->}[dr]_{\mu_A}
                            & \mathbb{R}^{n+1}            \ar[d]^{A^T}   \\
&& \mathbb{R}^{n}
} $$ \		
	
	Choose $x_0' \in M$ such that it is in the $c'$ level set of $\mu_A$. 
	Notice that $x \in \mu_A^{-1}(c') \Leftrightarrow A^T \mu_A(x) = c' = A^T \mu_A(x_0').$
\noindent Therefore the set $\mu_A^{-1}(c')$ can be written in the form 
$$\mu_A^{-1}(c') = \{ x \in M \hspace{2mm} | \hspace{2mm} \mu(x) - \mu(x_0') \in \ker(A^T) \}.$$

By assumption, $\mu_A^{-1}(c')$ is connected, in fact path connected.  

Let $x_1' \in \mu_A^{-1}(c')$ be another point in the same level set.   If $\mu(x_1') - \mu(x_0') \in \ker(A^T)$ then every convex combination of $\mu(x_0')$ and $\mu(x_1')$ is in $\mu(M)$.  We provide the details:  

\noindent Let $\gamma \colon [0,1] \rightarrow \mu^{-1}_A(c')$ with $\gamma(0) = x_0'$, $\gamma(1)=x_1'$ be the path connecting $x_0'$ and $x_1'$.  Observe that dim$\left( \ker(A^T) \right) = 1$ because $A$ is injective by hypothesis.  This implies that $A^T$ is surjective.  Then $\mu \left( \gamma (t) \right) - \mu(x_0') \in \ker(A^T)$ for each $t \in [0,1]$.  Hence, every convex combination of $\mu(x_0')$ and $\mu(x_1')$ must lie in $\mu(M)$, thus completing the proof of Case $1$.
	
		\begin{figure}[htpb]
	\begin{center}
	\resizebox{12cm}{6.8cm}{\includegraphics{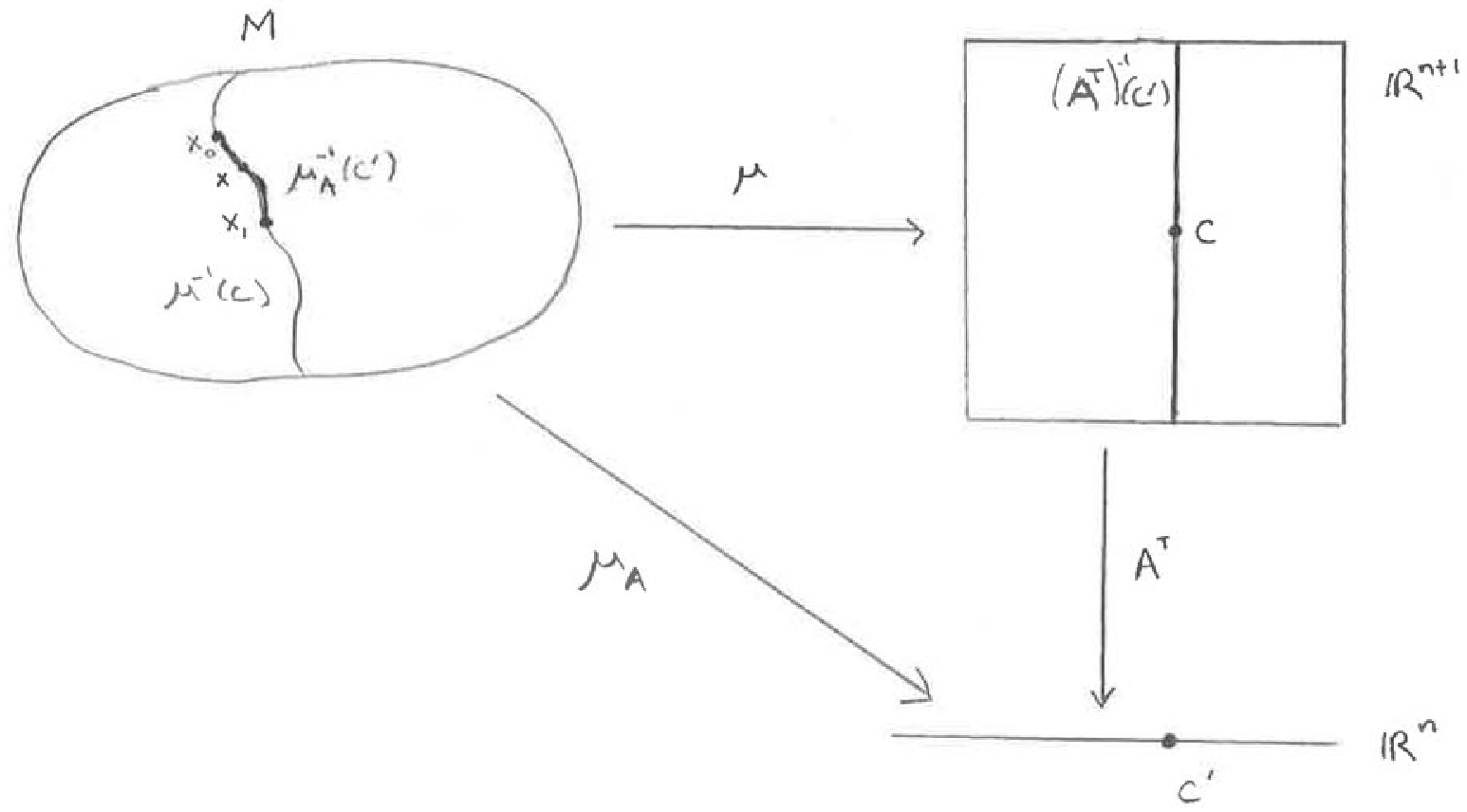}}
	\end{center}
	\caption{}
	\label{fig:ker(A^T)}
	\end{figure}
	
	\noindent \textbf{Case 2}: \textit{The ``general'' case } \
	
	Let $x_0$, $x_1$ be distinct arbitrary points in $M$.  
		
	We claim that $x_0$ and $x_1$ can be approximated arbitrarily closely by points $x_0'$, $x_1'$ with the property that $\mu(x_1') - \mu(x_0') \in \ker(A^T)$ for some injective matrix $A \in \mathbb{R}^{(n+1) \times n}$ such that $\pi := A^T$ satisfies conditions (i) and (ii) of Lemma \ref{good projection}.  With a further perturbation we may assume that $A^T \mu(x_0')$ is a regular value of $\mu_A$ and is in the (residual) set of values for which the level set of $\mu_A$ is connected  (by applying hypothesis ($A_n$)  to $\mu_A$).  To see this, first recall that the set of regular values of $\mu$ is residual in $\mathbb{R}^{n+1}$ by Theorem \ref{reg values residual}.  But a residual set in a complete metric space (such as $\mathbb{R}^{n+1}$ ) is dense in $\mathbb{R}^{n+1}$.  It follows that the set of regular values of $\mu$ is dense in $\mu(M)$.  By a similar argument applied to $\mu_A$ it can be established that the set of regular values of $\mu_A$ is dense in $\mu_A(M)$; Note that our assumptions imply that this restricted almost periodic $\mathbb{R}^n$ action on $M$ with momentum map $\mu_A$  satisfies conditions $(i)$ and $(ii)$ of Lemma \ref{good projection} (in particular, $\mu_A$ has isolated fixed points). Moroever, note that the intersection of the image of $\mu_A$ with the residual set described in ($A_n$) is dense in the momentum image.

Now, by Case $1$, every convex combination of $\mu( x_{0}' )$ and $\mu(x_{1}' )$ lies in $\mu(M)$.  Then our convexity result follows; since the image of $\mu$ is closed, by taking limits as $x_{0}' \rightarrow x_0$ and $x_{1}' \rightarrow x_1$ we obtain that $(1-t)\mu(x_0) + t \mu(x_1) \in \mu(M)$ \noindent for all $ 0 \leq t \leq 1$.

Taken as a whole, the statement ($B_{n}$) holds.
\end{proof}

	Our main result is the following.

\begin{thm}[Connectivity Theorem]
	\label{connectedness}

	Let $M$ be a connected strongly symplectic Hilbert manifold.  Suppose that we have an almost periodic $\mathbb{R}^n$ action on $M$ with momentum map $\mu : M \rightarrow \mathbb{R}^n$.  Suppose that the $\mathbb{R}^{n}$ action has isolated fixed points.  Suppose that there exists a complete invariant Riemannian metric on $M$ such that there exists a hyperplane $H$ of $\mathbb{R}^n$ such that for all $\xi \in \mathbb{R}^n \smallsetminus H$ the map  $\mu^{\xi} \colon \textrm{M} \rightarrow \mathbb{R}$ is bounded from one side and satisfies Condition (C).  Then the momentum mapping $\mu$ satisfies
	
	\begin{list}{}{}
	\item[($A$)] The set $\{ c \in \mathbb{R}^n \hspace{2mm} | \hspace{2mm} c \textrm{ is a regular value of $\mu$ and } \mu^{-1}(c) \textrm{ is connected } \} \subseteq \mathbb{R}^n$ is residual.

	\end{list}
	
\end{thm}

\begin{thm}[Convexity Theorem]
	\label{convexity}
	
	Let $M$ be a connected strongly symplectic Hilbert manifold.  Suppose that we have an almost periodic $\mathbb{R}^n$ action on $M$ with momentum map $\mu : M \rightarrow \mathbb{R}^n$.  Suppose that the $\mathbb{R}^{n}$ action has isolated fixed points and suppose that $\mu(M)$ is closed.  Suppose that there exists a complete invariant Riemannian metric on $M$ such that there exists a hyperplane $H$ of $\mathbb{R}^n$ such that for all $\xi \in \mathbb{R}^n \smallsetminus H$ the map  $\mu^{\xi} \colon \textrm{M} \rightarrow \mathbb{R}$ is bounded from one side and satisfies Condition (C).  Then the momentum mapping $\mu$ satisfies
	
	\begin{list}{}{}
	
	\item[($B$)] the image $\mu(M)$ is convex.
	
	\end{list}
	
\end{thm}

 \begin{rem}
The Convexity Theorem, Theorem \ref{convexity}, applies to finite-dimensional connected symplectic manifolds but eliminates the compactness assumption in the Atiyah-Guillemin-Sternberg Convexity Theorem \ref{finite convexity}. 

 \end{rem}

	We are ready to prove the main result of this thesis, the Connectivity Theorem \ref{connectedness}.

\begin{proof}[\textbf{Proof of Theorem \ref{connectedness}}]

	Consider the statement 
\begin{list}{}{}	
	\item{($A_n$):} Let $M$ be a connected strongly symplectic Hilbert manifold.  Suppose that we have an almost periodic $\mathbb{R}^n$ action on $M$ with momentum map $\mu \colon M \rightarrow \mathbb{R}^n$.  Suppose that the $\mathbb{R}^{n}$ action has isolated fixed points. Suppose that there exists a complete invariant Riemannian metric on $M$ such that there exists a hyperplane $H$ of $\mathbb{R}^n$ such that for all $\xi \in\mathbb{R}^n \smallsetminus H$ the map  $\mu^{\xi} \colon \textrm{M} \rightarrow \mathbb{R}$ is bounded from one side and satisfies Condition (C).  Then the set $$\{ c \in \mathbb{R}^n \hspace{2mm} | \hspace{2mm} c \textrm{ is a regular value of $\mu$ and } \mu^{-1}(c) \textrm{ is connected } \} \subseteq \mathbb{R}^n$$ is residual.
\end{list}

	\noindent Notice that $(A_n)$ applies to all $M$ and every $\mu$ on $M$.  By Lemma \ref{a_n B_n}, it is sufficient to prove statement ($A_n$) holds for all $n \in \mathbb{N}$.  We proceed by induction on $n$.
	
	Base Case:  In the case $n=1$, we have an almost periodic $\mathbb{R}$-action and thus $\mathfrak{t} = \mathbb{R}$ and $\mathfrak{t}^* = \mathbb{R}$, hence the momentum mapping $\mu \colon M \rightarrow \mathbb{R}$ is a smooth $\mathbb{R}$-valued function.  By Corollary \ref{base case}, $\mu^{-1}(c)$ is connected for every $c \in \mathbb{R}$, i.e., the set $$\{ c \in \mathbb{R} \hspace{2mm} | \hspace{2mm} \mu^{-1}(c) \textrm{ is connected } \} = \mathbb{R}.$$  Then the base case $(A_1)$ holds because the set $\{ c \in \mathbb{R} \hspace{2mm} | \hspace{2mm} c \textrm{ is a regular value of $\mu$ and } \mu^{-1}(c) \textrm{ is connected } \}$ is residual.  
	
	Induction Step:  Let $k \in \mathbb{N}$ be arbitrary.  Assume that $(A_k)$ is true for all possible almost periodic $\mathbb{R}^k$ actions on $M$ and let $\mu_1, \mu_2, \ldots , \mu_{k+1}$ be the components of a momentum mapping $\mu \colon M \rightarrow \mathbb{R}^{k+1}$ satisfying the hypothesis conditions of Theorem \ref{convexity}.  We want to show that $(A_{k+1})$ is true.  We have two cases to consider: \\
	1.  $\mu$ is reducible; and \\
	2. $\mu$ is irreducible.
	
	We say that $\mu$ is said to be \textbf{irreducible} if the $1$-forms $d\mu_1, d\mu_2, \ldots , d\mu_{n+1}$ are linearly independent, i.e.,  $$\alpha_1 d\mu_1 (m)(v) + \cdots \alpha_{n+1} d \mu_{n+1}(m)(v) = 0$$ \noindent at all points $m \in M$ and all vectors $v \in T_mM$ if and only if $\alpha_1 = \cdots = \alpha_{n+1} =0$.  We say that $\mu$ is \textbf{reducible} otherwise.
	  
 	If $\mu$ is reducible, then we are finished;  in this case there exists an $i \in \mathbb{N}$, $1 \leq i \leq k+1$, such that $d \mu_i$ is a linear combination of the other $1$-forms.  So we can drop $d \mu_i$ and apply our inductive hypothesis.  Thus, by the induction hypothesis the set of $c \in \mathbb{R}^{k+1}$ such that $c$ is a regular value of $\mu$ and $\mu^{-1}(c)$ is connected, is residual in $\mathbb{R}^{k+1}$.
 	
 	Let us assume that $\mu$ is irreducible.
 	By Lemma \ref{good projection}, there exists a projection $\pi := A^T \colon \mathbb{R}^{k+1} \rightarrow \mathbb{R}^k$ such that the restricted momentum map $\mu' := \pi \circ \mu$ satisfies all of the properties $(i)$ and $(ii)$ in Lemma \ref{good projection}.  Fix such a projection $\pi$.  Let $$G_{\mu'} := \{ c' \in \mathbb{R}^{k} \hspace{2mm} | \hspace{2mm} c' \textrm{ is a regular value of $\mu'$ and } (\mu')^{-1}\left( c' \right)  \textrm{ is connected} \} \subseteq \mathbb{R}^k.$$ 	
 	$$ \xymatrix{
                            & M \ar[r]^{\mu} 
                             \ar@{-->}[dr]_{\mu'}
                            & \mathbb{R}^{k+1}            \ar[d]^{\pi}   \\
&& \mathbb{R}^{k}
} $$ \

	Notice that $\mu'$ is the momentum map of a restricted almost periodic $\mathbb{R}^k$ action on $M$.  Note that there exists a basis of $\mathbb{R}^{k+1}$ so that $\pi$ drops the last coordinate.  Without loss of generality we may assume this is the standard basis.

  Let $c = ( c_1, \ldots , c_{k+1} ) \in \mathbb{R}^{k+1}$.  Consider $N := \mu^{-1}_1(c_1) \cap \cdots \cap \mu^{-1}_k(c_k)$.  Suppose that $c'$ is a regular value of $\mu'$.  It follows that: 
  
  	\begin{list}{}{}
	\item $\bullet$ the subset $N \subset M$ is a submanifold (of codimension $k$) in $M$ by the Implicit Function Theorem, and
	\item $\bullet$ the $1$-forms $(d\mu_i) (x)$, $1 \leq i \leq k$, are linearly independent for \textit{all} $x \in N$.
	\end{list}

\noindent Moreover, suppose that $\pi(c) \in G_{\mu'}$.  Then $N$ is connected by the definition of $G_{\mu'}$.
 	
 	\noindent Next, let us consider the restricted function $\mu_{k+1}|_N \colon N \rightarrow \mathbb{R}$.  
 
	\item {\bf Proposition}:  The function  $\mu_{k+1}|_N$ is a Morse function none of whose critical points have index or coindex equal to one in $N$.

	\begin{list}{}{}		
\item[\underline{Step 1}]:  \textit{We define a function $\phi \colon M \rightarrow \mathbb{R}$  and show that it has nondegenerate critical points of even index and coindex in $M$} \
 	
  Note that 
  given some $\lambda = ( \lambda_1, \ldots , \lambda_k) \in \mathbb{R}^k$ and $\mu'(x) = ( \mu_i(x), \ldots, \mu_k(x)) \in \mathbb{R}^k$ then $\langle \mu'(x), \lambda \rangle = \displaystyle \sum_{i=1}^k \lambda_i \mu_i(x)$.  Recall that a point $x \in N$ is a critical point of $\mu_{k+1}|_N$ if and only if there exist some constant $\lambda = ( \lambda_1, \ldots , \lambda_k ) \in \mathbb{R}^k$ such that $$d \mu_{k+1}(x)(v) + \sum_{i=1}^k \lambda_i d \mu_i(x)(v) = 0$$ for all $v \in T_xM$.  Therefore, $x$ is a critical point \textit{on $M$} for the function $\phi := \langle \mu, \lambda \rangle \colon M \rightarrow \mathbb{R}$ where $\lambda = (\lambda_1, \ldots, \lambda_k, 1 ) \in \mathbb{R}^{k+1}$.  That is, $$\phi = \mu_{k+1} + \sum_{i=1}^k \lambda_i \mu_i.$$   	Notice that $\phi$ is a Morse function because it has nondegenerate critical points (since $\mu_{k+1}$ has nondegenerate critical points and $\mu_{k+1}$ and $\phi$ differ only by the constant $\sum_{i=1}^k \lambda_i \mu_i$).  Thus, by Lemma \ref{even index} we know that no critical points of $\phi$ have index (coindex) equal to one in $M$.
 	
\item[\underline{Step 2}]:  \textit{Show that the restricted function $\phi|_N$ is a Morse function} \  	
  	
 Let $C := Crit(\phi) \subset M$ be the critical point set of $\phi$.  Let $x \in N$.  We wish to demonstrate that the manifold $C$ intersects $N$ transversally at $x$ (i.e. $T_xM = T_xC + T_xN$ ).  This means that the $1$-forms $d \mu_i(x) \colon T_xM \rightarrow \mathbb{R}$, $1 \leq i \leq k$, remain linearly independent when restricted to the subspace $T_xC$ (because this would show that the dual vector space to $T_xN + T_xC$ has the same codimension as $T_xM$ since the $d \mu_i(x)$, $ 1 \leq i \leq k$, vanish on $T_xN$).  Thus, it is sufficient to prove that $d \mu_i(x), \ldots , d \mu_k(x)$ remain linearly independent on $T_xC$.
 	
	To begin with observe that\
	\item $\bullet$ the vector fields $X_i := X_{\mu_i}$ (given by $d \mu_i = \iota_{X_i}$) for $ i = 1, \ldots , k$ must all lie tangent to $C$;\

 	We have that 
 	\begin{eqnarray*} 0 & = & d \mu_i (X_{\phi}) \\
								& = & \iota_{X_i} \omega(X_{\phi}) \\
								& = & \omega (X_i , X_{\phi}) \\
								& = & - \omega (X_{\phi} , X_i ) \\
								& = & - \iota_{X_{\phi}} \omega (X_i) \\
								& = & - d \phi( X_i). \
	\end{eqnarray*}
Thus, $\phi$ is constant on the level curves of $\mu_i$.  But then the Hamiltonian flow of $\mu_i$ must preserve $C$.  Therefore the (Hamiltonian) vector fields $X_i$ are tangent to $C$.

Thus $X_i(x) \in T_xC$ for $i = 1, \ldots , k$.
 	
 	\item $\bullet$ $\hspace{2mm} T_xC$ is a symplectic vector space;\
 	
 	$C$ is a symplectic submanifold of $M$ by Lemma \ref{rationally independent} because $C$ is a fixed point set of a torus action.  Therefore $T_xC$ is a symplectic vector space.
 	
 	This means that $\omega_x$ is nondegenerate on $T_xC$.  So for all $\lambda = ( \lambda_1, \ldots , \lambda_k) \in \mathbb{R}^k$ with not all $\lambda_i$ zero, there exists a nonzero vector $v \in T_xC$ such that
 	\begin{eqnarray*} 0 & \neq & \omega_x \left( \sum_{i=1}^k \lambda_i X_i(x) , v \right) \\
 								& = &  \sum_{i=1}^k \lambda_i \iota_{X_i(x)} \omega_x(v) \\
 								& = &  \sum_{i=1}^k \lambda_id \mu_i(x)(v). \
 	\end{eqnarray*}
 	Hence $d \mu_i(x)$, for $i = 1 , \ldots , k$, are linearly independent on $T_xC$.  Therefore $C$ is transverse to $N$.
 	
 	Now the fact that $T_xM = T_xN + T_xC$ implies that 
 	$\left( T_xC \right)^{\perp} \subseteq T_xN$.  From this notice that $H_x(\phi)$, the Hessian of $\phi$ at $x$, is nondegenerate on $T_xN \cap \left( T_xC \right)^{\perp}$ because $T_xM  \cap \left( T_xC \right)^{\perp} = T_xN \cap \left( T_xC \right)^{\perp}$ and so $$T_xN = T_xN \cap T_xC + T_xN \cap \left( T_xC \right)^{\perp}.$$  In particular, this means that the restricted function $\phi|_N \colon N \rightarrow \mathbb{R}$ is a Morse function with critical point set $C \cap N$.
 	  
 \item[\underline{Step 3}]:  \textit{Show that the function $\mu_{k+1}|_N$ has no critical points of index or coindex equal to one in $N$} \
 	
 	\indent Observe that by Lemma \ref{even index}, the function $\phi|_N$  has critical points of even index and coindex since $\phi|_N$ has nondegenerate critical points (by Step 2).  It then follows that $\mu_{k+1}|_N$ has nondegenerate critical points with even index and coindex because $\mu_{k+1}|_N$ only differs from $\phi$ by a constant, namely the constant $\sum_{i=1}^k \lambda_i c_i$, by definition of $\phi$.  This completes the proof of the proposition.
 	
 	\end{list}
 	
	By the proposition and by Theorem \ref{connectedlevels}, the level set of $\mu_{k+1}|_N$ is connected for \textit{every} $c_{k+1} \in \mathbb{R}$, i.e., $\left( \mu_{k+1}|_N \right)^{-1}(c_{k+1}) \subseteq N$ is connected for all $c_{k+1} \in \mathbb{R}$.  Hence $$\mu^{-1}(c) = N \cap \mu_{k+1}^{-1}(c_{k+1})$$ is connected 
for all $ c \in \pi^{-1}(c')$.  So the level set $\mu^{-1}(c)$ is connected for all $ c \in \pi^{-1} \left( G_{\mu'} \right).$  
But by the induction hypothesis 
the set $G_{\mu'}$ 
is residual in $\mathbb{R}^k$.  This implies that the set $$\pi^{-1} \left(  G_{\mu'} \right) \subseteq \mathbb{R}^{k+1}$$ is residual in $\mathbb{R}^{k+1}$ because $\pi^{-1} \left(  G_{\mu'} \right)$ is homeomorphic to $G_{\mu'} \times \mathbb{R}$. 
	
Let $G_{\mu} := \{ c \in \mathbb{R}^{k+1} \hspace{2mm} | \hspace{2mm} c \textrm{ is a regular value of $\mu$ and $\mu^{-1}(c)$ is connected } \} \subseteq \mathbb{R}^{k+1}.$ 	

By the definition of $G_{\mu'}$, the result just proven, and the definition of $G_{\mu}$, the set 
$$\pi^{-1} \left(  G_{\mu'} \right) \bigcap \left\lbrace \textrm{ regular values of $\mu$ } \right\rbrace \subseteq G_{\mu}.$$
It follows that $G_{\mu}$ is residual in $\mathbb{R}^{k+1}$.

\end{proof}

 \begin{proof}[\textbf{Proof of Theorem \ref{convexity}}]

This proof follows the method of Atiyah \cite{MA82} where $n = \textrm{dim}(\mathbb{R}^n)$.   Consider the statements ($A_n$) and ($B_n$) of Lemma \ref{a_n B_n}.

\noindent Then the statement ``image of $\mu$ is convex" holds if and only if ($B_n$) holds for all $n$.

Note that ($A_n$) holds for all $n$ by the Connectedness Theorem, Theorem \ref{connectedness}.  It follows that ($B_n$) holds for all $n$ by Lemma \ref{a_n B_n}.  Hence, the image $\mu(M)$ is convex.

\end{proof}

\begin{rem}
The results of the Connectivity Theorem \ref{connectedness} and the Convexity Theorem \ref{convexity} also apply to finite-dimensions where the manifold is not required to be compact or where the map is not required to be proper.
\end{rem}

\begin{rem}
We wonder whether the assumptions of our Connectedness Theorem, Theorem \ref{connectedness}, imply that the image of the momentum map is closed.  We do not know counterexamples.  Moreover, from Palais we know that for real-valued functions many consequences that follow from the image being closed are true.
\end{rem}

 \begin{rem}
 
 In light of the Connectivity Theorem \ref{connectedness} and the Convexity Theorem \ref{convexity}, directions for future research could include:
 
 \begin{list}{}{}
	\item[$\bullet \hspace{2mm}$] establishing connectivity of the level set $\mu^{-1}(c)$ for all regular values $c$ of the momentum map $\mu$;
	\item[$\bullet \hspace{2mm}$] establishing connectivity of the level set $\mu^{-1}(c)$ for all critical values $c$ of the momentum map $\mu$;
	\item[$\bullet \hspace{2mm}$] generalizing the Connectivity and Convexity Theorems so as to apply to Morse-Bott functions;
	\item[$\bullet \hspace{2mm}$] developing an infinite-dimensional non-abelian convexity result.
	
\end{list}

\end{rem}

\chapter{Example - The Based Loop Group}

	The purpose of this chapter is to provide examples of Theorem \ref{convexity}, the convexity main theorem.

\section{Example:  The Based Loop Group }

\subsubsection*{The Loop Group}  

	Let $G$ be a compact, connected and simply connected Lie group.  Fix a $G$-invariant inner product $\left\langle \cdot , \cdot \right\rangle$ on the Lie algebra $\mathfrak{g}$.  The \textit{loop group}, which we denote by $M_1$, is defined as the set of maps $S^1 \rightarrow G$ that are Sobolev class $H^1$.  Recall that a map $f \colon S^1 \rightarrow G$ is said to be \textit{Sobolev class $H^1$} if $f$ is absolutely continuous and $f^{-1}f' \in L^2(S^1, \mathfrak{g})$. 
	
	The space $M_1 = H^1(S^1, G)$ is an infinite-dimensional Hilbert manifold (cf. \cite[section $\S 13$]{RP63} and \cite[Section $\S 3$]{PS86}).  It carries a left invariant Riemannian metric, called the $H^1$ metric.  
The $H^1$ metric is uniquely determined by its restriction to the Lie algebra of $M_1$ which is $H^1(S^1, \mathfrak{g} )$ (the tangent space at the constant loop $e$ ).  That is, if we fix an $Ad(G)$-invariant metric, $( \cdot , \cdot )$, on $\mathfrak{g}$ then the $H^1$ metric is determined by 
	$$ \left\langle \gamma , \eta \right\rangle_e = \frac{1}{2 \pi} \int_0^{2\pi} \left( \gamma ( \theta) , \eta (\theta) \right) d \theta + \frac{1}{2 \pi} \int_0^{2 \pi} \left( \gamma'( \theta) , \eta' ( \theta) \right) d \theta ,$$

for $\gamma, \eta \in $Lie$(M_1) = H^1(S^1 , \mathfrak{g})$.

\subsubsection*{The Based Loop Group}	
	
	The  subset $\Omega G$ of $M_1$ consisting of those loops $f \colon S^1 \rightarrow G$ for which $f(1) \hspace{1mm} ( = e)$ is the identity element in $G$ is called the \textit{based loop group}.  Notice that $\Omega G$ is a closed submanifold of $M_1$ whose Lie algebra consists of those maps $\tilde{f} \colon S^1 \rightarrow \mathfrak{g}$ such that $\tilde{f}(1)=0$, i.e., $T_{\tilde{f}} \Omega G \cong H^1(S^1 , \mathfrak{g} ) / \mathfrak{g}$.  Moreover, the $H^1$ metric defined on $M_1$ induces a complete metric on $\Omega G$ (which we will denote by $\left\langle \cdot , \cdot \right\rangle$).  See Palais \cite[Section $\S 13$ Theorem $6$]{RP63}.  So $\Omega G$ is a connected Riemannian Hilbert manifold.

It can be seen (see \cite[Atiyah-Pressley, 
Section $\S 2$]{AP83}) that the formula $$\omega( \gamma, \eta ) = \frac{1}{2 \pi} \int_0^{2 \pi} \left\langle \gamma'( \theta), \eta ( \theta ) \right\rangle d \theta$$ \noindent where $\gamma, \eta \in H^1(S^1, \mathfrak{g} )$, defines a skew-symmetric bilinear form on $H^1(S^1, \mathfrak{g})$.  Moreover, $\omega$  is strongly nondegenerate.  Extending $\omega$ by left translations gives a left invariant closed $2$-form $\omega$ on $\Omega G$ (cf. \cite{PS86}, \cite[Section $\S 4$]{AP83}).  Thus, $\left( \Omega G, \omega \right)$ is strongly symplectic.

\subsubsection*{Group Actions on $\Omega G$}	

The rotation group $S^1$ acts on $\Omega G$ by ``rotating the loop":   \\ 
 		if $\gamma \in \Omega G$ and $e^{i \theta} \in S^1$, $\theta \in [0, 2 \pi]$, then $\left( e^{i \theta} \gamma \right)(s) := \gamma(s+ \theta) \gamma(\theta)^{-1}$.
 		
 		Let $T$ be the maximal torus of $G$.  Then $T$ acts on $\Omega G$ by conjugation:  \\ 
 		if  $\gamma \in \Omega G$ and $t \in T$, then $(t \gamma)(s) := t \gamma(s) t^{-1}$.
 		
 		Note that these actions commute and they are Hamiltonian \cite{PS86}.
 		
\begin{rem}
 The action $T \times S^1 \circlearrowright \Omega G$ is a special case of an almost periodic $\mathbb{R}^n$ action on $\Omega G$. 
\end{rem} 		
 		
 		The resulting $T \times S^1$ momentum map $\mu \colon \Omega G \rightarrow Lie( T \times S^1 ) \cong \mathfrak{t}^* \oplus \mathbb{R}^* \cong \mathfrak{t} \oplus \mathbb{R}$ is given by $\mu = p \oplus E$ with
  \begin{eqnarray*}
  E(f) & := & \frac{1}{4 \pi} \int_0^{2 \pi} || f( \theta )^{-1} f'(\theta) ||^2 \hspace{1mm} d \theta \hspace{11mm} \textrm{\textbf{Energy Functional}} \\
  p(f)	& := & pr_{\mathfrak{t}} \left( \frac{1}{2 \pi} \int_0^{2 \pi} \underbrace{f (\theta)^{-1} f'( \theta )}_{\in \mathfrak{g}} \hspace{1mm} d \theta \right) \hspace{5mm} \textrm{\textbf{Momentum Functional}} \
\end{eqnarray*}

\noindent where $pr_{\mathfrak{t}} \colon \mathfrak{g} \rightarrow \mathfrak{t}$ is the projection onto the Lie algebra of $T$.

\subsubsection*{Morse Theory for the Components of $\mu$}	

In this subsection we discuss the fact that a certain set of components of the momentum map $\mu \colon \Omega G \rightarrow \mathfrak{t} \oplus \mathbb{R}$ satisfy Condition (C) with respect to the $H^1$ metric.  

Note that the image of the momentum map $\mu = p \oplus E$ lies in $\mathfrak{t} \oplus \mathbb{R}$ which we can identify with its dual and with $\mathbb{R}^{N-1} \oplus \mathbb{R} \cong \mathbb{R}^N$.  Choose a hyperplane $H \subset \mathbb{R}^N$ such that $H = \{ x \in \mathbb{R}^N \hspace{1mm} | \hspace{1mm} x = ( 0 , x_2, \ldots , x_{N} ) \}$.  Then observe that for each $\xi \in \mathbb{R}^N \smallsetminus  H$ the $\mu^{\xi}$ component of the momentum map may be written as $$\displaystyle \mu^{\xi}(f) = x_1E(f) + \sum_{i=2}^N x_i p_i(f),$$ \noindent where $x_1 \neq 0$ and $f \in \Omega G$.  The fact that for each $\xi \in \mathbb{R}^N \smallsetminus  H$, $\mu^{\xi}$ is bounded from one side and satisfies Condition (C) follows from \cite[Proposition $2.9$]{HHJM06} whose proof relies on results of \cite{CLT89}.

\subsubsection*{Connectedness of Level Sets}

Let us briefly review what is known about the connectivity with regards to the based loop group.  

Recall that in \cite{HHJM06} Harada, Holm, Jeffrey, and Mare proved that any level set of the momentum map $\mu$ of the $T \times S^1$ action restricted to $\Omega_{\textrm{alg}}$ is connected (for regular or singular values of the momentum map) \ref{lisa 1}.  Note that the subset $\Omega_{\textrm{alg}}$ of $\Omega G$ could be equipped with the \textit{subspace topology} induced from the inclusion $\Omega_{\textrm{alg}} \hookrightarrow \Omega G$.  However, $\Omega_{\textrm{alg}}$ can also be equipped with a \textbf{direct limit topology} induced by the Grassmannian model (see \cite{HHJM06}, Section $\S 2$) for the algebraic loop group.  It turns out that the direct limit topology on $\Omega_{\textrm{alg}}$ is the appropriate topology for Theorem \ref{lisa 1}.  Harada, Holm, Jeffrey, and Mare also proved in \cite{HHJM06} that any level set of the momentum map $\mu$ for the $T \times S^1$ action on $\Omega G$ is connected provided that $c$ is a regular value of $\mu$ (with respect to the $H^1$ metric) \ref{lisa 2 regular}. 
In \cite{ALM10} Mare proved that the level set of $\mu^{-1}(c)$ of the momentum map for the $T \times S^1$ action on $\Omega G$ is connected for singular values of $\mu$.  His argument works for the space of $C^{\infty}$ loops and also for the space of loops of Sobolev class $H^s$ for any $ s \geq 1$.  

In terms of the results for this thesis, the Connectivity Theorem \ref{connectedness} establishes that in the presence of an almost periodic $\mathbb{R}^n$ action on $\Omega G$ (with momentum map $\mu$),  the set $\{ c \in \mathbb{R}^n \hspace{2mm} | \hspace{2mm} c \textrm{ is a regular value of $\mu$ and } \mu^{-1}(c) \textrm{ is connected } \} \subseteq \mathbb{R}^n$ is residual.

\subsubsection*{Convexity}

Let $R := T \times S^1$ act on $\Omega G$ as described above in the subsection ``Group Actions on $\Omega G$".  Atiyah and Pressley \cite{AP83}  showed in Theorem \ref{loop group convexity} that the image of the momentum map $\mu = p \oplus E$ is convex.  So the Convexity Theorem \ref{convexity} reproduces this known convexity result when $M = \Omega G$.

\addcontentsline{toc}{chapter}{Bibliography}
\nocite{*}
\bibliographystyle{plain}
\bibliography{ut-thesis}

\end{document}